\documentclass[reqno,11pt]{amsart}
\usepackage{amssymb, amsmath,amssymb, amscd, amsbsy,latexsym}
\usepackage{amsthm, amsfonts,mathrsfs}
\usepackage[T1]{fontenc}
\usepackage{times}
\usepackage{epsfig}
\usepackage{color}
\input amssym.def
\input amssym.tex
\def\inte#1{
\displaystyle\mathop{#1\kern0pt}^\circ }

\def\virgp{\raise 2pt\hbox{,}}
\def\cdotpv{\raise 2pt\hbox{;}}

\def\C{\mathop{\mathbb C\kern 0pt}\nolimits}
\def\DD{\mathop{\mathbb D\kern 0pt}\nolimits}
\def\EE{\mathop{{\mathbb E \kern 0pt}}\nolimits}
\def\K{\mathop{\mathbb K\kern 0pt}\nolimits}
\def\N{\mathop{\mathbb N\kern 0pt}\nolimits}
\def\Q{\mathop{\mathbb Q\kern 0pt}\nolimits}
\def\R{\mathop{\mathbb R\kern 0pt}\nolimits}
\def\SS{\mathop{\mathbb S\kern 0pt}\nolimits}
\def\ZZ{\mathop{\mathbb Z\kern 0pt}\nolimits}
\def\TT{\mathop{\mathbb T\kern 0pt}\nolimits}
\def\P{\mathop{\mathbb P\kern 0pt}\nolimits}

\newcommand{\beq}{\begin{equation}}
\newcommand{\eeq}{\end{equation}}
\newcommand{\ben}{\begin{eqnarray}}
\newcommand{\een}{\end{eqnarray}}
\newcommand{\beno}{\begin{eqnarray*}}
\newcommand{\eeno}{\end{eqnarray*}}
\newtheorem{defi}{Definition}[section]
\newtheorem{thm}{Theorem}[section]

\newtheorem{rmk}{Remark}[section]

\renewcommand{\theequation}{\thesection.\arabic{equation}}
\newtheorem{remark}{Remark}[section]
\newtheorem{lemma}{Lemma}[section]
\newtheorem{cor}{Corollary}[section]
\newtheorem{prop}{Proposition}[section]

\newdimen\eqjot \eqjot = 1\jot
\def\openupeq{\openup \the\eqjot}

\def\pofbox#1 #2$#3${\setbox0=\hbox{$#3$}\ht0=0pt\dp0=0pt\wd0=0pt\hskip-#1pt\raise#2pt\box0\hskip#1pt}
\begin{document}

\title[Reducibility of the Camassa-Holm equation with unbounded perturbations]
{Reducibility of the dispersive Camassa-Holm equation with unbounded perturbations}

\author{Xiaoping Wu}
\address{Xiaoping Wu\newline
Center for Nonlinear Studies and School of Mathematics, Northwest University, Xi'an 710127, P. R.
China}
\email{wuxiaopingnwu@163.com}

\author{Ying Fu\textsuperscript{*} }
\address{Ying Fu\newline
Center for Nonlinear Studies and School of Mathematics\\
Northwest University\\
Xi'an 710127\\
P. R. China} 
\email{fuying@nwu.edu.cn}

\author{Changzheng Qu}
\address{Changzheng Qu\newline
School of Mathematics and Statistics\\ Ningbo University\\Ningbo 315211\\P.R. China} \email{quchangzheng@nbu.edu.cn}

\maketitle

\begin{abstract}

Considered herein is  the reducibility of the quasi-periodically time dependent linear dynamical system with a diophantine frequency vector $\omega \in \mathcal{O}_0 \subset \mathbb{R}^{\nu}$. This system is derived from linearizing the dispersive Camassa-Holm equation with unbounded perturbations  at a small amplitude quasi-periodic function.
It is shown that there is a set $\mathcal{O}_{\infty} \subset \mathcal{O}_0$ of asymptotically full Lebesgue measure such that for  any $\omega \in \mathcal{O}_{\infty}$, the system can be reduced to the one with constant coefficients by a quasi-periodic linear transformation. The strategy adopted in this paper consists of two steps:
 (a) A reduction based on the   orders of the pseudo differential operators in the system which conjugates the linearized operator to a one with constant coefficients up to a small remainder; (b)  A perturbative reducibility scheme which completely diagonalizes the remainder of the previous step.  The main difficulties in the reducibility we need to tackle come from the operator $J=(1-\partial_{xx})^{-1}\partial_{x}$, which induces the symplectic structure of the dispersive Camassa-Holm equation.

\end{abstract}

\date{July 27, 2022}

\maketitle


\vskip 0.2cm


%%%%%%%%%%%%%%%%%%%%%%%%%%%%%%%%%%%%%%%%%%%%%%%%%%%%%%%%%%%%
\renewcommand{\theequation}{\thesection.\arabic{equation}}
\setcounter{equation}{0}
%%%%%%%%%%%%%%%%%%%%%%%%%%%%%%%%%%%%%%%%%%%%%%%%%%%%%%%%%%%%
\section{Introduction}
In this paper, we are mainly concerned with  the reducibility of the quasi-periodically time dependent linear dynamical system
\begin{equation*}
h_t - J \circ \big((a_0(\omega t, x)+m_0)+(a_2(\omega t, x))_x \partial_x+(a_2(\omega t, x)+m_2)\partial_{xx}\big)h=0,
\end{equation*}
where $x\in \mathbb{T}= \mathbb{R} / 2\pi \mathbb{Z},$ $m_0, m_2 \in \mathbb{R},$
 $ a_0, a_2$ are small functions, and $\omega \in \mathcal{O}_0 \subset \mathbb{R}^{\nu}$ is a diophantine frequency vector.
The  above system arises from linearizing the dispersive Camassa-Holm (CH) equation with unbounded perturbations on the circle at a small amplitude quasi-periodic function.

There have been a number of literatures to study reducibility of ordinary differential equations (ODEs) and partial differential equations (PDEs). Indeed, the concept of reducibility originates from ODEs (see \cite{bms, eli, hy, js} and the references therein). The problem of reducibility of quasi-periodic linear systems in the  PDE's context has received much attention, mostly in a perturbative regime in these years \cite{bg, bgmr, ek-1, gp, mon}.

A quasi-periodic linear system
 \begin{equation}\label{ht=Lh}
 \partial_t h = L(\omega t) h
 \end{equation}
 is said to be reducible, if there exists a bounded  invertible change of coordinates depending on time quasi-periodically: $h=\Upsilon(\omega t)g$ such that the transformed system is a linear one with constant coefficients, namely,
 \begin{equation*}
 \partial_t g=D_{\omega}g,  \quad  D_{\omega}:={\rm diag}_{j \in \mathbb{Z}}{d_j},  \quad d_j \in \mathbb{C}.
 \end{equation*}
 Notice that, if  $d_j$ $(\forall j \in \mathbb{Z})$ is purely imaginary, then the  system is linear stable.

In the periodic case, i.e. $\omega \in \mathbb{R}$, the classical Floquet theory shows that any time periodic linear system \eqref{ht=Lh} is reducible. In the quasi-periodic case, this is  not  always true, see e.g., \cite{bhs}.
However, if $L(\omega t)$ is a so-called almost-reducible quasi-periodic vector field in the sense that it can be reduced to the one with constant coefficients up to a small remainder, viz., the linear differential equation  of the form
\begin{equation*}
\partial_t \tilde{g}=D_{\omega}\tilde{g}+\varepsilon P(\omega t)\tilde{g},
\end{equation*}
 where $P(\omega t)$ is a linear quasi-periodically forced operator with non-constant coefficients and $\varepsilon$ is the size of the perturbation, then the quasi-periodic  system is reducible by perturbative reducibility algorithm. In general, it is assumed that $\varepsilon$ is small enough,  $\omega$ and $d_j$ satisfy the second-order Melnikov non-resonance conditions involving the differences of the eigenvalues of the operator $D_{\omega}$.

%We refer to \cite{Johnson-Sell, Coppel}for the  works that study the the sufficient or necessary conditions
%for reducibility or almost reducibility of the quasi-periodic systems.

Among the literatures that related to reducibility, a strong motivation of a vast of them is the reducible KAM theory. As is well known, the KAM theory for PDEs is to find a family of approximately invariant tori of perturbed autonomous integrable equation. It is a natural extension of the classical KAM theory for finite dimensional phase spaces. For PDEs in one spacial dimension with bounded perturbations and Dirichlet boundary conditions, we refer to  \cite{kuk1, ku-p, pos2, way},  and for those with periodic boundary conditions,  we quote \cite{cy, cw} for instance. In higher spacial dimension,  we mention \cite{bb, bou, ek-2, gy, pp} among others in which the authors have to overcome the difficulties produced by the multiple eigenvalues. In all these aforementioned results, the perturbations are bounded. In the case of  unbounded perturbations, we mention   \cite{bbp-1, bbp-2, ka-p, kuk2, ly, zgy} among others for semi-linear PDEs, and \cite{bbhm, bbm-1, bbm-2, bm, fp, giu} for quasi-linear or fully nonlinear PDEs. In particular, these results with regard to the KAM theory involve the quasi-linear perturbations of the Airy, KdV  equations, etc.

 The CH equation
 \begin{equation}\label{ndch}
 u_t-u_{xxt}=6uu_x-4u_xu_{xx}-2uu_{xxx}
\end{equation}
describes the unidirectional propagation of shallow water waves over a flat bottom \cite{ch, cl, ff, joh}, where $u(t,x)$  is a function of time $t$ and a single spatial variable $x$. Equation \eqref{ndch} is proved to be completely integrable since it admits the  Lax-pair and  bi-Hamiltonian structure \cite{ch, clop, fs, ff, len}.  We refer to \cite{ce-3, lo} and references therein for more literature for the study of  the well-posedness  of the CH-equation. A tremendous amount of work has been done on strong nonlinear effects  of the CH-equation,  such as  peakon, multi-peakons \cite{achm, ch} and wave-breaking phenomena \cite{con1, con2, ce-1, ce-2, ce-3, lo}.

Note that  $\int_{\Bbb T}u(t,x){\rm{d}}x$ is a constant for soluitions of the CH-equation \eqref{ndch}, the set
\begin{equation*}
{\mathcal{G}}_c:=\{ u:\int_{\Bbb T}u(t,x){\rm{d}}x=c\}
 \end{equation*}
is invariant under  the  flow governed by \eqref{ndch}. Consequently, the dynamics of equation \eqref{ndch} on the invariant subsets ${\mathcal{G}}_c$ with $c \neq 0$ are equivalent to the ones of the equation
\begin{equation*}
u_t-u_{xxt}-6cu_x+2cu_{xxx}=6uu_x-4u_xu_{xx}-2uu_{xxx}
\end{equation*}
on the invariant subset ${\mathcal{G}_0}$. We assume $c=1$ without loss of generality and impose the unbounded perturbations of the form
\begin{equation*}
N(u, u_x, u_{xx}, u_{xxx})=\partial_x[(\partial_u f)(u, u_x)-\partial_x((\partial_{u_x}f)(u, u_x))],  
\end{equation*}
where $f(u, u_x)=C^{\infty}(\mathbb{R} \times \mathbb{R}, \mathbb{R})$, and $f=f_{\geq 3}(u, u_x)$ denotes a function with a zero of order at least three at the origin.
Then the equation  we are concerned with is  the perturbed dispersive CH-equation
\begin{equation}\label{dch}
u_t-u_{xxt}-6u_x+2u_{xxx}=6uu_x-4u_xu_{xx}-2uu_{xxx}+N(u, u_x, u_{xx}, u_{xxx}),
\end{equation}
under periodic boundary condition: $x\in \mathbb{T}=\mathbb{R} / 2\pi \mathbb{Z}$.

The Degasperis-Procesi equation \cite{dkk} exhibits several similar properties as the CH equation, for example, wave breaking phenomena, peakon and soliton solutions \cite{cl, dp}. However,  the Hamiltonians and   symplectic structures are intrinsically different.  In \cite{fgp-2}, the authors developed the KAM theory  for   Hamiltonian perturbations of the  Degasperis-Procesi equation.

At the best of our knowledge, up to now there is no work to study the KAM theory and the reducibility of the  CH equation with a small perturbation, which was also emphasized in \cite{fgp-2}.  In fact,   the reducibility of the linearized equations at a small quasi-periodic approximate solutions  is the fundamental step for the existence and  linear stability of quasi-periodic solutions (KAM tori) for nonlinear PDEs via the Nash-Moser iterative algorithm. This motivates us to study the reducibility of the CH-equation.

Let us  first introduce the Hamiltonian setting for the CH-equation and its linearized equation at a small quasi-periodic function.

Equation \eqref{dch} can be formulated as a Hamiltonian PDE of the form $u_t=J\nabla H(u)$, where $ J :=(1-\partial_{xx})^{-1}\partial_{x}$ and $\nabla H(u)$ is the $L^2( {\Bbb T},{\Bbb R})$ gradient of the Hamiltonian
\begin{equation*}
H(u)=\int_{\Bbb T} \big(u^3+uu^2_x+3u^2+u^2_x+f(u, u_x)\big)  {\rm{d}}x,
\end{equation*}
defined on the real phase space $H^1_0(\mathbb{T}):=H^1\cap \mathcal{G}_0$.
The symplectic structure is provided by a non-degenerate 2-form $\varrho$ defined by
\begin{equation*}
\varrho(u, v):=\int_{\mathbb{T}}(J^{-1}u)v {\rm d}x, \quad u,v \in H^1_0.
\end{equation*}
 Note that $J$ can be written as
\begin{equation*}
J = \Lambda \partial_x, \quad \Lambda := (1-\partial_{xx})^{-1}.
\end{equation*}

The functional space be concerned is the Sobolev space $H^s({\Bbb T}^{\nu+1};{\Bbb R}),$ $\nu \geq 0, s \in \mathbb{R},$
equipped with the norm
\begin{equation*}
{\Vert u(\varphi, x)\Vert}^2_{s} := \sum\limits_{j \in \mathbb{Z}, l \in \mathbb{Z}^{\nu} }
{\vert u_{lj}\vert}^2{\langle  l, j \rangle}^{2s}< \infty,
\end{equation*}
where $\langle l, j \rangle := \max \{ 1,|j|, |l| \}$,
$|l|:= \max\limits_{i=1, \cdots, \nu} |l_i|.$
If $u(\varphi, x)$ depends on a parameter $\omega \in \mathcal{O}$, where $\mathcal{O}$ is a compact subset of $\mathbb{R}^{\nu}$, we define the sup-norm  and  Lipschitz semi-norm  of $u$ respectively as follows:
\begin{equation*}
\|u\|_{s}^{sup}:=\|u\|_{s}^{sup, \mathcal{O}}:=\sup \limits_{\omega \in \mathcal{O}}\|u\|_{s},
\end{equation*}
\begin{equation*}
\|u\|_{s}^{lip}:=\|u\|_{s}^{lip, \mathcal{O}}:=\sup \limits_{\substack{\omega, \omega' \in \mathcal{O},\\  \omega \neq \omega'}}
\|\Delta_{\omega, \omega'}u\|_{s},
\end{equation*}
where
\begin{equation*}
\Delta_{\omega, \omega'}u:=\frac{u(\omega)-u(\omega')}{\omega-\omega'}.
\end{equation*}
For $\gamma >0$, the weighted Lipchitz norm of $u$ is defined as
\begin{equation*}
\|u\|_{s}^{\gamma, \mathcal{O}}:=\|u\|_{s}^{sup, \mathcal{O}}+ \gamma \|u\|_{ s-1}^{lip, \mathcal{O}}.
\end{equation*}
Let $m : \mathcal{O} \rightarrow \mathbb{R}$, the sup-norm, Lipschitz semi-norm  and weighted Lipchitz norm of $m$  are defined respectively as
\begin{equation*}
|m|^{sup}:=|m|^{sup, \mathcal{O}}:=\sup \limits_{\omega \in \mathcal{O}}|m|,
\end{equation*}
\begin{equation*}
|m|^{lip}:=|m|^{lip, \mathcal{O}}:=\sup \limits_{\substack{\omega, \omega' \in \mathcal{O},\\  \omega \neq \omega'}}
|\frac{m(\omega)-m(\omega')}{\omega-\omega'}|,
\end{equation*}
\begin{equation*}
|m|^{\gamma, \mathcal{O}}:=|m|^{sup, \mathcal{O}}+ \gamma |m|^{lip, \mathcal{O}}.
\end{equation*}

In the sequel, we fix $s_0:=[(\nu+1)/2]+2$. For all $s \geq s_0$, the Sobolev space $H^s$ is a Banach algebra.

\noindent {\it Notation}. Throughout the paper, the notation $a\leq_{s, \alpha}b$ indicates that $a\leq C(s, \alpha) b$ for some positive constant $C(s, \alpha)$ depending on the variables $s, \alpha$. $s$ will be omitted in the notation  $\leq_s$ while $s$ is a fixed constant. As usual, the positive constants $C$'s may be different from line to line.

Fix $\nu \in \mathbb{N} \setminus \{ 0 \}$, $L>0$, let $\gamma \in (0,1)$, the frequency vector of oscillations $\omega=(\omega_1, \dots, \omega_{\nu} )$
satisfies the diophantine condition, i.e., $\omega \in \mathcal{O}_0 \subset \Omega$, where
 \begin{equation}\label{O_0}
 \mathcal{O}_0:= \big\{ \omega \in \Omega : |\omega \cdot l| \geq \frac{2\gamma}{\langle l\rangle^{\nu}}, l \in \mathbb{Z}^{\nu}\setminus \{ 0 \} \big\},   \quad \langle l\rangle:=\max\{|l|,1\},
   \end{equation}
 \begin{equation}\label{Omega}
 \Omega:= \{ \omega \in \mathbb{R}^{\nu}:  \omega \in [L, 2L]^{\nu} \}.
 \end{equation}

A quasi-periodic function with $\nu$ frequencies is defined by an embedding
\begin{equation*}
\mathbb{T}^{\nu} \ni \varphi \mapsto  \mathfrak{J}(\varphi, x), \quad \varphi=\omega t
\end{equation*}
 with the frequency vector $\omega \in \mathcal{O}_0$.
Then a small-amplitude quasi-periodic solution $u(t,x)$ of equation \eqref{dch} can be represented as
\begin{equation*}
u(t,x)=\varepsilon \mathfrak{J}(\varphi, x),  \quad \varepsilon \ll 1,    \quad   \mathfrak{J}(\varphi, x)\in C^{\infty}(\mathbb{T}^{\nu+1}, \mathbb{R}).
\end{equation*}

The linearized equation of \eqref{dch} at the quasi-periodic solution $u(t,x)$ is
\begin{eqnarray}\label{dynam-eq}
\begin{aligned}
h_t=&\;J(\partial_u \nabla H)[h] \\
=&J \circ \big((a_0(\omega t, x)+6)
+a_{2,x}(\omega t, x)\partial_x+(a_2(\omega t, x)-2)\partial_{xx}\big)h,
\end{aligned}
\end{eqnarray}
 where $a_i(\varphi, x)\in C^{\infty}(\mathbb{T}^{\nu+1}, \mathbb{R})$, $i=0, 2$, depend on the  the parameter $\omega \in \mathcal{O}_0$ in a Lipschitz way, as well as on the quasi-periodic function $ \mathfrak{J}$.
The operator associated with the above equation is
\begin{equation}\label{L^*}
\begin{aligned}
 \mathcal{L}^{*} =& \omega \cdot \partial_{\varphi}-J \circ \big((a_0(\omega t, x)+6)\\
 &\quad+a_{2,x}(\omega t, x)\partial_x+(a_2(\omega t, x)-2)\partial_{xx}\big).
   \end{aligned}
  \end{equation}
It is noticed that this operator is a Hamiltonian operator, and  exhibits the reversible structure.

One can check that the operator $L^*(\omega t)$ has the form $L^*(\omega t)=-(a_2-2)\partial_x+{\rm Op}(r)$ where ${\rm Op}(r)$ is a pseudo differential operator of order $-1$ (see Definition \ref{pd-def}). Therefore, it is a new type of operator different from the ones for the KdV  and  Degasperis-Procesi equations, etc.

In the following, we  consider a class of generalized  linear operators of the form
  \begin{equation}\label{L}
  \mathcal{L}= \omega \cdot \partial_{\varphi}-J \circ \big((a_0(\varphi, x)+m_0)+a_{2,x}(\varphi, x)\partial_x+(a_2(\varphi, x)+m_2)\partial_{xx}\big).
  \end{equation}
Now we make the following assumptions for the operators  $\mathcal{L}$:

(A1). $a_i(\varphi, x) \in C^{\infty}(\mathbb{T}^{\nu+1}, \mathbb{R})$,  $i=0,2$, are even functions, $m_0, m_2 \in \mathbb{R}$.

(A2). There exists a positive constant $\mu$ sufficiently large such that
\begin{equation}\label{sc-J}
\|\mathfrak{J}\|_{s_0+\mu}^{\gamma, \mathcal{O}_0}\leq 1.
\end{equation}

(A3). Assume that $a_i$, $i=0,2$, depend on $\omega \in \mathcal{O}_0$ in a Lipschitz way, satisfying
\begin{equation}\label{sc-a_i}
\|a_i\|_s^{\gamma, \mathcal{O}_0}\leq_{s} \varepsilon \|\mathfrak{J}\|_{s+\eta_0}^{\gamma, \mathcal{O}_0}, \quad i=0, 2, \quad \forall s\geq s_0
\end{equation}
for some $\eta_0>0$.

(A4). Assume that $a_i$, $i=0,2$, depend on  the  quasi-periodic function $\mathfrak{J}$.  Let $\mathfrak{J}_1, \mathfrak{J}_2 \in C^{\infty}(\mathbb{T}^{\nu+1}, \mathbb{R})$ satisfying  \eqref{sc-J} and
\begin{equation*}
\| \Delta_{12}a_i\|_p \leq_p \varepsilon \|\mathfrak{J}_1-\mathfrak{J}_2\|_{p+\eta_0},  \quad  i=0, 2,
\end{equation*}
for any  $s_0 \leq p \leq s_0+\mu-\eta_0(\mu > \eta_0)$, where  $\Delta_{12}v:=v(\mathfrak{J}_1, \varphi, x)-v(\mathfrak{J}_2, \varphi, x)$ for any $v(\varphi, x)=v(\mathfrak{J}, \varphi, x)\in C^{\infty}(\mathbb{T}^{\nu+1}, \mathbb{R})$.

The main goal in this paper is to prove that the linear operators \eqref{L} arising from the linearized CH-equation with unbounded
perturbations at a small and sufficiently smooth quasi-periodic function on $\mathbb{T}$ are reducible.

It is well known that the most common algebraic structures that  ensure  the existence of quasi-periodic motions are the Hamiltonian or reversible structures. The system considered here exhibits both of them. From the perspective of  the Hamiltonian structure, we have to find  symplectic transformations induced by the flow of the vector field generated by a Hamiltonian function. The flow is given by $\partial_{\tau}A^{\tau}u=B^{\tau}A^{\tau}u$, where $B^{\tau}$ is a pseudo differential operator of order $\leq -1$  (see Definition \ref{pd-def}) so as to guarantee the transformations are bounded. The solution $A^{\tau}u$ of the flow is of the form: $C+OPS^{-1}$ for some $C \in \mathbb{R}$, while $OPS^{-1}$ is a pseudo differential operator of order $-1$. Unfortunately, $A^{\tau}u$ can not change the leading order terms of the operator \eqref{L}.
  Thus in order to prove the reducibility of the system, we utilize the reversible structure and choose the transformations preserving the reversible structure.

 It is remarked that for the Degasperis-Procesi equation, the operator $J$ associated with the symplectic structure is a pseudo differential operator of order 1, which is also true for the KdV equation. Nevertheless, for the CH-equation, $J = (1-\partial_{xx})^{-1}\partial_x$ is a pseudo differential operator of order $-1$. Such difference needs us to develop a different strategy to verify the corresponding reducibility.

 The existence and  stability of the  solution for the Cauchy problem of the linearized equation \eqref{dynam-eq} are also the direct results of this paper (see Corollary \ref{cor}). The stability  entails the purely-imaginary eigenvalues of the  diagonal operator. Indeed, the reversible structure ensures that the eigenvalues of the diagonal matrices are all purely-imaginary.

The reducibility result for the linearized CH-equation \eqref{dch} at a small and sufficiently
smooth quasi-periodic function is as follows.

\begin{thm}\label{thm1}
Suppose $\gamma \in (0, 1)$ small enough,  $\| \mathfrak{J} \|_{s_0+\mu} \leq 1$ with $\mu>0$ sufficiently large, and the linear operator $\mathcal{L}^{*}(\omega t)$ in \eqref{L^*} defined on $\mathcal{O}_0$. Then there exist $\varepsilon_0>0$ and a Cantor-like set $\mathcal{O}_{\infty} \subset \mathcal{O}_0$, for all $\varepsilon \in (0, \varepsilon_0)$ and $\omega \in \mathcal{O}_{\infty}$, and a real, invertible, bounded and reversibility-preserving transformation $\Upsilon(\omega t)$ depending quasi-periodically on time which reduces $\mathcal{L}^{*}(\omega t)$  to a diagonal operator with  constant coefficients and purely imaginary spectrum. Moreover, as $\gamma\rightarrow 0$, the Lebesgue measure $|\mathcal{O}_0 \setminus \mathcal{O}_{\infty}|$ tends to 0.

 \end{thm}

The following corollary indicates the dynamical consequence of Theorem \ref{thm1}.
\begin{cor}\label{cor}
Assume that the hypotheses in Theorem \ref{thm1} are satisfied. Let $ h(0, x)=h_0(x) \in H^s(\mathbb{T})$, $s \geq s_0$. Then the  Cauchy problem of equation \eqref{dynam-eq} with the initial value $h_0(x)$ has a unique solution $h(t,x)$, which satisfies
 \begin{equation*}
 \sup_{t\in \mathbb{R}} \| h(t, \cdot) \|_s \leq_s  \| h(0, \cdot) \|_s.
 \end{equation*}
\end{cor}

%Corollary \ref{cor} implies that Sobolev norms of the solutions for the equation \eqref{dynam-eq} with the initial condition \eqref{i-c} remain bounded in time. This is due to the quasi-periodic dependence on time of the dynamical system.

In the case of $m_2=-2, m_0=6$, Theorem \ref{thm1} is a conclusion of the following result.

\begin{thm}\label{thm2}(Reducibility)
Let $\omega \in \mathcal{O}_0$ (see \eqref{O_0}). Suppose that the operator in \eqref{L} satisfies the  assumptions (A1)-(A4) for some constant $\mu>0$ sufficiently large. Assume that $\gamma, \varepsilon \gamma^{-1-\kappa}(\kappa>1)$ are small enough, $\tau \geq 2\nu+3$, $m_2<0$, $m_0+m_2\geq 0$ and  $m_0 < -5 m_2- 4\delta_0\ $ for some positive constant $\delta_0$.
Then there exists a sequence
\begin{equation*}
d_j^{\infty}:=m_{\infty}j-\frac{(m_0+m_2)j}{1+j^2}+r_j^{\infty}, \quad j \in \mathbb{Z}\setminus \{0\}, \quad r_{-j}^{\infty}=-r_{j}^{\infty}
\end{equation*}
where $m_{\infty}, r_j^{\infty} \in \mathbb{R}$ depend on $\omega$ in a Lipschitz way satisfying
$$|m_{\infty}-m_2|^{\gamma, \mathcal{O}_0} \leq C\varepsilon,$$
\begin{equation*}
\sup_{j} \langle j \rangle |r_j^{\infty}|^{\gamma^{\kappa}, \mathcal{O}_0}\leq C\varepsilon \gamma^{-1}.
  \end{equation*}

For all $\omega \in \mathcal{O}_{\infty}:=\mathcal{O}_{1} \cap \mathcal{O}_{2}$, where
 \begin{equation*}
  \mathcal{O}_1:= \big\{\omega \in \mathcal{O}_0:  |\omega \cdot l+m_{\infty}(\omega)j|\geq \frac{2\gamma}{\langle l, j \rangle^{\tau}}, \quad l \in \mathbb{Z}^{\nu}, j \in \mathbb{Z}\setminus \{0\} \big\},
\end{equation*}
\begin{equation*}
\begin{split}
\mathcal{O}_{2} := \big\{ \omega \in \mathcal{O}_0: & |\omega \cdot l + d_j^{\infty}- d_{j'}^{\infty}| \geq \frac{2\gamma^{\kappa}|j-j'|}{\langle l \rangle^{\tau}}, \\
& \forall l \in \mathbb{Z}^{\nu}, j,j' \in \mathbb{Z}\setminus \{0\}, (j, j', l)\neq (j, j, 0)\big\},\\
\end{split}
\end{equation*}
there exists a real linear, bounded, T\"{o}plitz-in-time  transformation $\Upsilon: \mathcal{O}_{\infty} \times H^s\rightarrow H^s$ with bounded inverse $\Upsilon^{-1}$ such that for all $\omega \in \mathcal{O}_{\infty}$,
  \begin{equation*}
\mathcal{L}_{\infty}:=\Upsilon^{-1} \mathcal{L} \Upsilon= \omega \cdot \partial_{\varphi}+D_{\infty}, \quad D_{\infty}:= {\rm diag}_{j\in \mathbb{Z}\setminus \{0\}} ({\rm i}d_j^{\infty}).
 \end{equation*}
$\mathcal{L}_{\infty}$ is real and reversible. The transformations $\Upsilon^{\pm 1}$ are reversibility-preserving and satisfy the following tame estimates:
 \begin{equation*}
\| \Upsilon^{\pm 1} u\|_s^{\gamma^{\kappa}, \mathcal{O}_{\infty}} \leq_s \| u \|_{s} + \varepsilon \gamma^{-1-\kappa}\| \mathfrak{J} \|_{s+\mu}^{\gamma, \mathcal{O}_0}\| u \|_{s_0}.
   \end{equation*}
In addition, for $\omega \in \mathcal{O}_{\infty}(\mathfrak{J}_1)\cap \mathcal{O}_{\infty}(\mathfrak{J}_2)$,
 \begin{equation*}
  |\Delta_{12}m_{\infty}| \leq \varepsilon \| \mathfrak{J}_1 -\mathfrak{J}_2 \|_{s_0+\mu}, \quad
\sup_{j} \langle j\rangle|\Delta_{12}r_j^{\infty}| \leq C \varepsilon \gamma^{-1}\| \mathfrak{J}_1 -\mathfrak{J}_2 \|_{s_0+\mu}.
\end{equation*}
Finally,  the following estimate holds for $\mathcal{O}_{\infty}:$
\begin{equation*}
|\mathcal{O}_{0} - \mathcal{O}_{\infty}|\leq C \gamma^{\min\{1, \kappa-1\}} L^{\nu-1}.
\end{equation*}

\end{thm}

We explain the key ideas and the  outline of this paper  as follows.

In Section 2, we introduce some fundamental definitions, notations and lemmas. Particularly, we introduce the classes of Lip-$-1$-modulo-tame linear operators (see Definition \ref{mod-def}) and study their properties which will be used in Section 4.
The reduction procedure is split into two sections.

In Section 3, we perform  the regularization procedure. This step  consists in  conjugating $\mathcal{L}(\omega t)$ in \eqref{L} to the linear operator
which is a small bounded regularizing perturbation of a constant coefficient. This is achieved by applying a suitable quasi-periodically change of variables depending on time  so that the highest order (order 1) term has a constant coefficient. Meanwhile, the term of order 0 is eliminated. Then we extract the terms that is not "small"  from the ones of order $-1$, to constitute the diagonal part. This is the content of Theorem 3.1.

In Section 4, the KAM reducibility scheme is proposed. After the preliminary reduction of the order of derivatives,  we can perform a KAM reducibility scheme to complete the diagonalization, see Theorem 4.1. We use the Lip-$-1$-modulo-tame constants to estimate
the size of the remainders along the iteration. This is convenient since the class of Lip-$-1$-modulo-tame operators are closed under the composition (Lemma \ref{mod-sum-com}), the solution map of the homological equation (Lemma
\ref{Hom-eq}) and the projections (Lemma \ref{mod-smoo}). As a matter of fact, the properties mentioned above also hold for Lip-$0$-modulo-tame operators (see Definition \ref{mod0-def}). The  reason to adopt Lip-$-1$-modulo-tame operators instead of Lip-$0$-modulo-tame operators is that we require that $\langle j\rangle r_j^{\infty}$ is small enough in order  to prove the relations between the bad sets $P_{ljj'}$ and $Q_{l, j}$ in \eqref{PQ-def}.

In Section 5, we give the measure estimate for  the set $\mathcal{O}_0-\mathcal{O}_{\infty}$. Therefore,  the operator in \eqref{L} can be reduced
to the one with  constant coefficients for almost all $\omega$ in the sense of Lebesgue
measure. At the same time, Corollary \ref{cor} is verified.

Finally, in the appendices, we introduce some classes of linear operators which are mentioned in the paper, such as  pseudo-differential operators, the classes of operators $\mathfrak{L}_{\rho, p}$. The classical results for  the change of variable are provided as well.

%%%%%%%%%%%%%%%%%%%%%%%%%%%%%%%%%%%%%%%%%%%%%%%%%%%%%%%%%%%%
\renewcommand{\theequation}{\thesection.\arabic{equation}}
\setcounter{equation}{0}
%%%%%%%%%%%%%%%%%%%%%%%%%%%%%%%%%%%%%%%%%%%%%%%%%%%%%%%%%%%%
\section{Preliminaries}

In this section, we elaborate some conceptions, notations and lemmas, which will be used in the subsequent sections.

\begin{defi}(Majorant function)  %\cite{bm, fgp-1, fgp-2}
Given a function $u(\varphi,x)$ of the form
\begin{equation*}
u(\varphi, x)=\sum_{l \in \mathbb{Z}^{\nu}, j \in \mathbb{Z}}u_{l, j}{\rm e}^{{\rm i}(l\cdot \varphi+jx)} \in L^2(\mathbb{T}^{\nu}\times \mathbb{T}, \mathbb{C}),
\end{equation*}
  the majorant function $\underline{u}(\varphi, x)$ of $u(\varphi, x)$ is defined as
\begin{equation*}
\underline{u}(\varphi, x):=\sum_{l \in \mathbb{Z}^{\nu}, j \in \mathbb{Z}}|u_{l, j}|{\rm e}^{{\rm i}(l\cdot \varphi+jx)}.
\end{equation*}
\end{defi}

It is obvious that the Sobolev norms of $u$ and $\underline{u}$  are equal, namely,
\begin{equation*}
\| u \|_s=\| \underline{u} \|_s.
\end{equation*}

The following lemma is an  important property of $\| \cdot\|_{s}$ and $\| \cdot \|_{s}^{\gamma, \mathcal{O}}$.

\begin{lemma}\label{Ip}(Interpolation)\cite{bbm-1, fgmp}
For all $s\geq s_0$, there are $C(s)\geq C(s_0)\geq 1$ such that
\begin{equation*}
\| uv \|_s \leq C(s) \| u \|_s  \| v \|_{s_0} + C(s_0)\| u \|_{s_0}  \| v \|_s.
\end{equation*}
The above inequality also holds for weighted Lipchitz norm $\| \cdot \|_{s}^{\gamma, \mathcal{O}}$.
\end{lemma}

In the following, we introduce a class of operators, whose matrix representations are T\"{o}plitz in time, and the relevant properties.

\begin{defi}
(Linear operators)
Let $A(\varphi)$ be a $\varphi$-dependent family of linear operators acting on $L^2(\mathbb{T}_x, \mathbb{R})$.
We regard $A$ as a linear operator acting on $H^s(\mathbb{T}^{\nu+1},\mathbb{R})$ defined by
\begin{equation*}
A: u(\varphi, x) \mapsto \big(A(\varphi)u(\varphi, \cdot)\big)(x).
\end{equation*}
In  Fourier coordinates, the action is represented as
 \begin{equation}\label{Au-act}
 (Au)(\varphi, x)=\sum_{j,{j'} \in \mathbb{Z}} A_j^{j'}(\varphi)u_{j'}(\varphi){\rm e}^{{\rm i}jx}=\sum_{j \in \mathbb{Z}, l\in \mathbb{Z}^{\nu}}
\sum_{j' \in \mathbb{Z}, l'\in \mathbb{Z}^{\nu}} A_j^{j'}(l-l')u_{j'l'} {\rm e}^{{\rm i}(jx+l\cdot \varphi)}.
 \end{equation}
\end{defi}

\begin{defi}\label{Lo-var}
Given a linear  operator $A$ as in \eqref{Au-act}, we define the following operators:

 (1). $\underline{A}$\ , called majorant operator, whose matrix elements are
\begin{equation*}
(\underline{A})_j^{j'}(l-l'):=|A_j^{j'}(l-l')|,  \quad \forall j, j' \in \mathbb{Z}, l, l' \in \mathbb{Z}^{\nu}.
\end{equation*}

 (2). $\Pi_N A$, called smoothed operator, whose matrix elements are
\begin{equation*}
(\Pi_N A)_j^{j'}(l-l'):=
 \left \{
\begin{array}{ll}
A_j^{j'}(l-l'), &  {\rm if}    \quad |l-l'| \leq N,\\
0, & {\rm otherwise}.\\
\end{array}
\right.
 \end{equation*}
We also denote
\begin{equation}\label{smo-c}
\Pi_N^{\perp} := {\rm{Id}} - \Pi_N.
\end{equation}

(3). $\langle \partial_{\varphi} \rangle^b A$, whose matrix elements are $\langle l-l' \rangle^b A_j^{j'}(l-l'), \quad b\geq 0.$

(4). $\partial_{\varphi_m} A$, whose matrix elements are $i(l_m-l'_m)A_j^{j'}(l-l'), \quad m= 1, \ldots, \nu.$

(5). $[\partial_{x}, A] :=  \partial_{x}\circ A - A \circ \partial_{x}$, whose matrix elements are $i(j-j')A_j^{j'}(l-l').$
\end{defi}

\begin{lemma}\label{M-N}
Let $M, N$ be  two  linear operators as in \eqref{Au-act}.

(1). For any $s\geq 0$,
we have $\| Mu \|_s \leq   \| \underline{M}\ \underline{u} \|_s$.

(2). If for any $j,j'\in \mathbb{Z}$ and $l\in \mathbb{Z}^{\nu}$,  $ |M_j^{j'}(l)|\leq |N_j^{j'}(l)|$, then
\begin{equation*}
 \| \underline{M} \|_{\mathfrak{L}(H^s)} \leq \| \underline{N} \|_{\mathfrak{L}(H^s)},
\quad  \| \underline{M}\ \underline{u} \|_s \leq \| \underline{N}\ \underline{u} \|_s, \quad s \geq 0.
\end{equation*}
\end{lemma}

\begin{defi}\label{rev-def}
Let $Z$, $X, Y$ denote the set of real functions, even functions and odd functions, respectively. An operator $A$ is called

(1) real, if $A:Z \rightarrow Z$.

(2) reversible, if $A:X \rightarrow Y$.

(3) reversibility-preserving, if $A:X \rightarrow X$ and $ A:Y \rightarrow Y$.
\end{defi}

In terms of   matrix coefficients, the above definitions  are equivalent to the following:

(1). $A$ is real $\Leftrightarrow A_{-j}^{-j'}(-l)=\overline{A_{j}^{j'}(l)}$.

(2). $A$ is reversible $ \Leftrightarrow A_{-j}^{-j'}(-l)=-A_{j}^{j'}(l)$.

(3). $A$ is reversibility-preserving $\Leftrightarrow A_{-j}^{-j'}(-l)=A_{j}^{j'}(l)$.

\noindent {\it Notation}. If a linear operator $A$ depends on the variable $i$, we denote
\begin{equation*}
\Delta_{12}A := A(i_1)-A(i_2).
\end{equation*}
If $A$ depends on the parameter $\omega \in \mathcal{O} \subset \mathbb{R}^{\nu}$, we denote
\begin{equation*}
\Delta_{\omega, \omega'}A:=\frac{A(\omega)-A(\omega')}{|\omega-\omega'|}.
\end{equation*}

\begin{defi}\label{mod0-def}(Lip-$\sigma$-modulo-tame operators) \cite{bbhm, bm, fgp-1, fgp-2}
Let the linear operator $A(\omega)$ be defined for $\omega \in \mathcal{O} \subset \mathbb{R}^{\nu}$.  A is Lip-$\sigma$-modulo-tame
if there exist $\sigma \geq 0$ and  a non-decreasing sequence $\{ \mathfrak{M}_A^{\sharp, \gamma}(\sigma, s) \}_{s=s_0}^\mathcal{S}$ (with possibly $\mathcal{S}=+\infty$)
such that for all $u \in H^{s+\sigma},$
\begin{equation*}
\sup \limits_{\omega \in \mathcal{O}} \|\underline{A}u\|_s, \sup \limits_{\substack{\omega, \omega' \in \mathcal{O},\\  \omega \neq \omega'}}\gamma
\| \underline{\Delta_{\omega, \omega'}A}u \|_{s}
\leq \mathfrak{M}_A^{\sharp, \gamma}(\sigma, s)\|u\|_{s_0+\sigma} + \mathfrak{M}_A^{\sharp, \gamma}(\sigma, s_0)\|u\|_{s+\sigma}.
\end{equation*}

\end{defi}

\begin{defi}\label{mod-def}(Lip-$-1$-modulo-tame operators) \cite{fgp-1,fgp-2}
If $\langle D_x\rangle^{1/2}A \langle D_x\rangle^{1/2}$ is Lip-0-modulo-tame, we say $A$ is Lip-$-1$-modulo-tame, where
\begin{equation*}
\mathfrak{M}_A^{\sharp, \gamma}(-1, s):= \mathfrak{M}_{\langle D_x\rangle^{1/2}A \langle D_x\rangle^{1/2}}^{\sharp, \gamma}(0, s)
\end{equation*}
is the Lip-$-1$-modulo-tame constant for the operator $A$ or modulo-tame constant for short.

If $\langle \partial_{\varphi}\rangle^{b} A$ is Lip-$-1$-modulo-tame, viz. $\langle D_x\rangle^{1/2} \langle \partial_{\varphi}\rangle^{b} A \langle D_x\rangle^{1/2} $ is Lip-0-modulo-tame, we set
\begin{equation*}
\mathfrak{M}_A^{\sharp, \gamma}(-1, s, b):=
\mathfrak{M}_{ \langle \partial_{\varphi}\rangle^{b} A }^{\sharp, \gamma}(-1, s)
:=\mathfrak{M}_{\langle D_x\rangle^{1/2} \langle \partial_{\varphi}\rangle^{b} A \langle D_x\rangle^{1/2}}^{\sharp, \gamma}(0, s).
\end{equation*}

\end{defi}

\begin{remark}
In the sequel, we  rescale  $\gamma$ in order to adapt to the non-resonance condition. Precisely,  we
use $\mathfrak{M}_A^{\sharp, \gamma^{\kappa}}(\sigma, s)(\kappa>1)$ instead of $\mathfrak{M}_A^{\sharp, \gamma}(\sigma, s)$ in Section 4.
This is a typical trick of solving problems with (asymptotically) linear dispersion, for instance, the Klein-Gordon equation \cite{bbp-1, pos1} and  the Degasperis-Procesi equation \cite{fgp-2}.
\end{remark}

In the following, we  collect  several  properties of the Lip-$-1$-modulo-tame operators.

\begin{lemma}\label{L(Hs0)-mod}
Let $A$ be a  Lip-$-1$-modulo-tame operator with modulo-tame constant:  
$$\quad \mathfrak{M}_A^{\sharp, \gamma}(-1, s).$$ 
Then
\begin{equation*}
\| A  \|_{\mathfrak{L}(H^s)}
 \leq \| \underline{A}  \|_{\mathfrak{L}(H^s)}
 \leq \| \langle D_x \rangle^{1/2}\underline{A} \langle D_x \rangle^{1/2} \|_{\mathfrak{L}(H^s)} \leq 2 \mathfrak{M}_A^{\sharp, \gamma}(-1, s).
 \end{equation*}
Moreover, if $u$ depends on  $\omega \in \mathcal{O} \subset \mathbb{R}^{\nu}$ in a Lipschitz way, then
\begin{equation*}
\|  A u \|_s^{\gamma, \mathcal{O}}
 \leq \mathfrak{M}_A^{\sharp, \gamma}(-1, s)\|  u \|_{s_0}^{\gamma, \mathcal{O}}+\mathfrak{M}_A^{\sharp, \gamma}(-1, s_0)\|  u \|_{s}^{\gamma, \mathcal{O}}.
 \end{equation*}
 \end{lemma}

\begin{proof}
Since $\|Au\|_s \leq \|\underline{A}\ \underline{u}\|_s$ and $\| u \|_s= \| \underline{u} \|_s$, we have $\| A  \|_{\mathfrak{L}(H^s)}
 \leq \| \underline{A}  \|_{\mathfrak{L}(H^s)}$.
In view of
  \begin{equation*}
  |\underline{A}_j^{j'}(l)| = |A_j^{j'}(l)| \leq  \langle j \rangle^{1/2} |A_j^{j'}(l)| \langle j' \rangle^{1/2},
  \end{equation*}
by Lemma \ref{M-N}, $$\| \underline{A}  \|_{\mathfrak{L}(H^s)}
 \leq \| \langle D_x \rangle^{1/2}\underline{A} \langle D_x \rangle^{1/2} \|_{\mathfrak{L}(H^s)}.$$
For the inequality
$$\| \langle D_x \rangle^{1/2}\underline{A} \langle D_x \rangle^{1/2} \|_{\mathfrak{L}(H^s)} \leq 2 \mathfrak{M}_A^{\sharp, \gamma}(-1, s),$$
its proof  can be found in \cite{fgp-1}. According to Definition \ref{mod-def}, the last inequality follows by
  \begin{equation*}
  \begin{split}
\|\Delta_{\omega, \omega'}(Au) \|_s
\leq &\| (\Delta_{\omega, \omega'}A)u(\omega)\|_s+ \| A(\omega')\Delta_{\omega, \omega'}u \|_s\\
 \leq &\| (\underline{\Delta_{\omega, \omega'}A})\underline{u(\omega)}\|_s+ \| \underline{A(\omega')}\ \underline{\Delta_{\omega, \omega'}u} \|_s\\
 \leq & \| \langle D_x \rangle^{1/2} (\underline{\Delta_{\omega, \omega'}A}) \langle D_x \rangle^{1/2} \underline{u(\omega)}\|_s
+ \| \langle D_x \rangle^{1/2} \underline{A(\omega')}\langle D_x \rangle^{1/2}\underline{\Delta_{\omega, \omega'}u} \|_s\\
\leq & \gamma^{-1}(\mathfrak{M}_A^{\sharp, \gamma}(-1, s_0)\|  \underline{u} \|_{s}+\mathfrak{M}_A^{\sharp, \gamma}(-1, s)\|  \underline{u} \|_{s_0}
+\mathfrak{M}_A^{\sharp, \gamma}(-1, s_0)\|  u\|^{\gamma, \mathcal{O}}_{s}\\
&+\mathfrak{M}_A^{\sharp, \gamma}(-1, s)\|  u \|^{\gamma, \mathcal{O}}_{s_0})\\
\leq & \gamma^{-1}( \mathfrak{M}_A^{\sharp, \gamma}(-1, s)\|  u \|_{s_0}^{\gamma, \mathcal{O}}+\mathfrak{M}_A^{\sharp, \gamma}(-1, s_0)\|  u \|_{s}^{\gamma, \mathcal{O}})\\
 \end{split}
  \end{equation*}
and
 \begin{equation*}
  \begin{split}
 \sup_{\omega \in \mathcal{O}}\| Au \|_s \leq & \sup_{\omega \in \mathcal{O}}\| \underline{A}\ \underline{u} \|_s
   \leq \| \langle D_x \rangle^{1/2} \underline{A} \langle D_x \rangle^{1/2} \underline{u}\|_s\\
   \leq & \mathfrak{M}_A^{\sharp, \gamma}(-1, s)\| \underline{u} \|_{s_0}+\mathfrak{M}_A^{\sharp, \gamma}(-1, s_0)\|  \underline{u} \|_{s}\\
   \leq & \mathfrak{M}_A^{\sharp, \gamma}(-1, s)\|u\|_{s_0}^{\gamma, \mathcal{O}}+\mathfrak{M}_A^{\sharp, \gamma}(-1, s_0)\|  u \|_{s}^{\gamma, \mathcal{O}}.\\
 \end{split}
  \end{equation*}
\end{proof}

\begin{lemma}\label{mod-compa}
Assume  $|A_j^{j'}(l)| \leq  |B_j^{j'}(l)|$ for any $j, j' \in \mathbb{Z}, l \in \mathbb{Z}^{\nu}$.

(1). If $B$ is  a Lip-$-1$-modulo-tame operator with modulo-tame constant $\mathfrak{M}_B^{\sharp, \gamma}(-1, s)$
and there exists $C \geq 0$ such that
 $|\Delta_{\omega, \omega'}A_j^{j'}(l)| \leq |\Delta_{\omega, \omega'}B_j^{j'}(l)| + C  |B_j^{j'}(l)|$ for all $\omega \neq \omega' \in \mathcal{O}$, then $A$ is a Lip-$-1$-modulo-tame operator as well and we can choose the modulo-tame constant of $A$ such that
\begin{equation*}
\mathfrak{M}_A^{\sharp, \gamma}(-1, s) \leq C' \mathfrak{M}_B^{\sharp, \gamma}(-1, s)
\end{equation*}
for some constant $C'>0$.

(2). If $B$ is an operator satisfying $\| \langle D_x \rangle^{1/2} \underline{B} \langle D_x \rangle^{1/2} \|_{\mathfrak{L}(H^{s})} < +\infty$,
 then
 \begin{equation*}
 \| \langle D_x \rangle^{1/2} \underline{A} \langle D_x \rangle^{1/2} \|_{\mathfrak{L}(H^{s})}
 \leq \| \langle D_x \rangle^{1/2} \underline{B} \langle D_x \rangle^{1/2} \|_{\mathfrak{L}(H^{s})}.
\end{equation*}
\end{lemma}
\begin{proof}
(1). From Lemma \ref{M-N} and Definition \ref{mod-def}, we derive that
 \begin{equation*}
  \begin{split}
\| \langle D_x \rangle^{1/2}\underline{A}\langle D_x \rangle^{1/2} u \|_s
\leq&\|\langle D_x \rangle^{1/2} \underline{A} \langle D_x \rangle^{1/2} \underline{u} \|_s
\leq \| \langle D_x \rangle^{1/2} \underline{B}\langle D_x \rangle^{1/2} \underline{u} \|_s\\
\leq & \mathfrak{M}_B^{\sharp, \gamma}(-1, s)\|\underline{u}\|_{s_0}+\mathfrak{M}_B^{\sharp, \gamma}(-1, s_0)\| \underline{ u} \|_{s}\\
=& \mathfrak{M}_B^{\sharp, \gamma}(-1, s)\|u\|_{s_0}+\mathfrak{M}_B^{\sharp, \gamma}(-1, s_0)\| u \|_{s}.\\
 \end{split}
  \end{equation*}
 On the other hand,
 \begin{equation*}
  \begin{split}
\| \langle D_x \rangle^{1/2} \underline{\Delta_{\omega, \omega'}A} \langle D_x \rangle^{1/2}u \|_s
\leq & \| \langle D_x \rangle^{1/2}\underline{\Delta_{\omega, \omega'}A}\langle D_x \rangle^{1/2} \underline{u} \|_s \\
\leq & \| \langle D_x \rangle^{1/2}\underline{\Delta_{\omega, \omega'}B}\langle D_x \rangle^{1/2} \underline{u} \|_s
+C \| \langle D_x \rangle^{1/2} \underline{B}\langle D_x \rangle^{1/2} \underline{u} \|_s\\
\leq & \gamma^{-1} (\mathfrak{M}_B^{\sharp, \gamma}(-1, s)\|\underline{u}\|_{s_0}+\mathfrak{M}_B^{\sharp, \gamma}(-1, s_0)\| \underline{ u} \|_{s})\\
 &+C(\mathfrak{M}_B^{\sharp, \gamma}(-1, s)\|\underline{u}\|_{s_0}+\mathfrak{M}_B^{\sharp, \gamma}(-1, s_0)\| \underline{ u} \|_{s})\\
 \leq &  \gamma^{-1} (\mathfrak{M}_B^{\sharp, \gamma}(-1, s)\|u\|_{s_0}+\mathfrak{M}_B^{\sharp, \gamma}(-1, s_0)\| u \|_{s}). \\
 \end{split}
  \end{equation*}
Hence, according to Definition \ref{mod-def}, $\mathfrak{M}_A^{\sharp, \gamma}(-1, s) \leq C' \mathfrak{M}_B^{\sharp, \gamma}(-1, s)$.

  (2). Since $|A_j^{j'}(l)| \leq  |B_j^{j'}(l)|$, one obtains
   \begin{equation*}
   \langle j\rangle^{1/2}|A_j^{j'}(l)|\langle j'\rangle^{1/2} \leq  \langle j\rangle^{1/2}|B_j^{j'}(l)|\langle j'\rangle^{1/2}.
   \end{equation*}
By Lemma \ref{M-N}, the inequality holds.
 \end{proof}

\begin{lemma}\label{mod-sum-com}(Sum and composition) \cite{fgp-1}
Let $A, B$ be two Lip-$-1$-modulo-tame operators with modulo-tame constants respectively $\mathfrak{M}_A^{\sharp, \gamma}(-1, s)$ and $\mathfrak{M}_B^{\sharp, \gamma}(-1, s)$. Then

(1). $A+B$ is Lip-$-1$-modulo-tame  with modulo-tame constant
\begin{equation*}
\mathfrak{M}_{A+B}^{\sharp, \gamma}(-1, s) \leq \mathfrak{M}_A^{\sharp, \gamma}(-1, s) + \mathfrak{M}_B^{\sharp, \gamma}(-1, s).
\end{equation*}

(2). The composed operator $A\circ B$ is Lip-$-1$-modulo-tame  with modulo-tame constant
\begin{equation*}
\mathfrak{M}_{AB}^{\sharp, \gamma}(-1, s) \leq_s \mathfrak{M}_A^{\sharp, \gamma}(-1, s) \mathfrak{M}_B^{\sharp, \gamma}(-1, s_0)+
\mathfrak{M}_A^{\sharp, \gamma}(-1, s_0) \mathfrak{M}_B^{\sharp, \gamma}(-1, s).
\end{equation*}

 (3). Assume in addition that $\langle \partial_{\varphi} \rangle^b A$ and $\langle \partial_{\varphi} \rangle^b B$ are Lip-$-1$-modulo-tame operators with modulo-tame constants  
 $$\mathfrak{M}_{\langle \partial_{\varphi} \rangle^b A}^{\sharp, \gamma}(-1, s) \quad \text{and} \quad \mathfrak{M}_{\langle \partial_{\varphi} \rangle^b B}^{\sharp, \gamma}(-1, s)$$
 respectively, then
$\langle \partial_{\varphi} \rangle^b (AB)$ is Lip-$-1$-modulo-tame  with modulo-tame constant
\begin{equation*}
\begin{split}
\mathfrak{M}^{\sharp, \gamma}_{\langle \partial_{\varphi} \rangle^b (AB)}(-1, s)
\leq_{s,b}
& \mathfrak{M}_{\langle \partial_{\varphi} \rangle^b A}^{\sharp, \gamma}(-1, s)
\mathfrak{M}_B^{\sharp, \gamma}(-1, s_0) + \mathfrak{M}_{\langle \partial_{\varphi} \rangle^b A}^{\sharp, \gamma}(-1, s_0)
\mathfrak{M}_B^{\sharp, \gamma}(-1, s)\\
&+
\mathfrak{M}_A^{\sharp, \gamma}(-1, s_0) \mathfrak{M}_{\langle \partial_{\varphi} \rangle^b B}^{\sharp, \gamma}(-1, s)+ \mathfrak{M}_A^{\sharp, \gamma}(-1, s)\mathfrak{M}_{\langle \partial_{\varphi} \rangle^b B}^{\sharp, \gamma}(-1, s_0).\\
 \end{split}
\end{equation*}
\end{lemma}

\begin{lemma}\label{mod-inv}(Invertibility) \cite{fgp-1}
Let $\Phi := {\rm{Id}} + \Psi$, where $\Psi$ and $\langle \partial_{\varphi} \rangle^b \Psi$ are both Lip-$-1$-modulo-tame operators. If there exists a constant $C(\mathcal{S}, b)>0$ such that the smallness condition
\begin{equation*}
4C(\mathcal{S}, b)\mathfrak{M}_{\Psi}^{\sharp, \gamma}(-1, s_0) \leq 1/2
\end{equation*}
holds,
then the operator $\Phi$ is invertible, $\breve{\Psi}:= \Phi^{-1}- {\rm{Id}}$ and   $\langle \partial_{\varphi} \rangle^b \breve{\Psi}$ are Lip-$-1$-modulo-tame operators with modulo-tame constants
\begin{eqnarray*}
\begin{aligned}
&\mathfrak{M}_{\breve{\Psi}}^{\sharp, \gamma}(-1, s) \leq 2 \mathfrak{M}_{\Psi}^{\sharp, \gamma}(-1, s),\\
&\mathfrak{M}_{\langle \partial_{\varphi} \rangle^b \breve{\Psi}}^{\sharp, \gamma}(-1, s) \leq
2 \mathfrak{M}_{\langle \partial_{\varphi} \rangle^b \Psi}^{\sharp, \gamma}(-1, s) + 8 C(\mathcal{S}, b) \mathfrak{M}_{\langle \partial_{\varphi} \rangle^b \Psi}^{\sharp, \gamma}(-1, s_0) \mathfrak{M}_{\Psi}^{\sharp, \gamma}(-1, s).
\end{aligned}
\end{eqnarray*}
\end{lemma}

\begin{lemma}\label{mod-smoo}(Smoothing) \cite{fgp-1}
Assume that $A$ and $\langle \partial_{\varphi} \rangle^b A$ are two Lip-$-1$-modulo-tame operators with modulo-tame constants respectively $\mathfrak{M}_A^{\sharp, \gamma}(-1, s), \mathfrak{M}_{\langle \partial_{\varphi} \rangle^b A}^{\sharp, \gamma}(-1, s).$
Then the operator $\Pi_N^{\perp} A$ is Lip-$-1$-modulo-tame with modulo-tame constant
\begin{equation*}
\mathfrak{M}_{ \Pi_N^{\perp} A}^{\sharp, \gamma}(-1, s) \leq
\min \{ N^{-b} \mathfrak{M}_{ \langle \partial_{\varphi} \rangle^b A}^{\sharp, \gamma}(-1, s),
\mathfrak{M}_A^{\sharp, \gamma}(-1, s)
 \}.
 \end{equation*}
\end{lemma}

\begin{lemma}\label{[A]-pro}
 Let $[A]:={\rm diag }_{j \in \mathbb{Z}} A_j^j(0)$.

 (1).
  Suppose that $A$ and $\langle \partial_{\varphi} \rangle^b A$ are  two Lip-$-1$-modulo-tame operators. Then
 \begin{equation*}
 \mathfrak{M}_{[A]}^{\sharp, \gamma}(-1, s)\leq  \mathfrak{M}_A^{\sharp, \gamma}(-1, s),   \quad   \mathfrak{M}_{[A]}^{\sharp, \gamma}(-1, s, b)\leq  \mathfrak{M}_A^{\sharp, \gamma}(-1, s, b),
 \end{equation*}
 and
\begin{equation*}
 \langle j \rangle |A_j^j(0)|^{\gamma, \mathcal{O}} \leq \mathfrak{M}_A^{\sharp, \gamma}(-1, s_0), \quad \forall j \in \mathbb{Z}.
 \end{equation*}

 (2). If $A$ is a linear operator such that $\langle D_x \rangle^{1/2}\underline{\Delta_{12}A}\langle D_x \rangle^{1/2}$ $\in \mathcal{L}(H^{s_0})$,
 then
  \begin{equation*}
  \langle j \rangle |\Delta_{12}A_j^j(0)| \leq \| \langle D_x \rangle^{1/2}\underline{\Delta_{12}A}\langle D_x \rangle^{1/2}  \|_{\mathcal{L}(H^{s_0})},
  \quad \forall j \in \mathbb{Z},
  \end{equation*}
  and
  \begin{equation*}
   \| \langle D_x \rangle^{1/2}\underline{\Delta_{12}[A]}\langle D_x \rangle^{1/2}  \|_{\mathcal{L}(H^{s_0})}
  \leq \| \langle D_x \rangle^{1/2}\underline{\Delta_{12}A}\langle D_x \rangle^{1/2}  \|_{\mathcal{L}(H^{s_0})}.
  \end{equation*}

 \end{lemma}

\begin{proof}
(1). Since
 \begin{equation*}
([A])_j^{j'}(l):=
\left \{
\begin{array}{ll}
A_j^{j'}(l), &   {\rm if} \quad (j, j', l)=(j, j, 0),\\
0, & {\rm otherwise},\\
\end{array}
\right.
 \end{equation*}
then we have
\begin{equation*}
 |([A])_j^{j'}(l)| \leq |(A)_j^{j'}(l)|,
\end{equation*}
 and
\begin{equation*}
|([A])_j^{j'}(l)(\omega)-([A])_j^{j'}(l)(\omega')| \leq |(A)_j^{j'}(l)(\omega)-(A)_j^{j'}(l)(\omega')|.
\end{equation*}
Thus the first bound is derived from Lemma \ref{mod-compa}.
The second one can be proved similarly.
The third one has been verified in \cite{fgp-1}.

(2).
For any $j_0$, we consider the unit vector
\begin{equation*}
u^{(j_0)} : = u_{j_0, 0}{\rm e}^{{\rm i} j_0 x}=\langle j_0 \rangle^{-s_0}{\rm e}^{{\rm i} j_0 x} \in H^{s_0}(\mathbb{T}^{\nu+1}).
\end{equation*}
It follows that
 \begin{equation*}
 \begin{split}
\| \langle D_x \rangle^{1/2}\underline{\Delta_{12}[A]}\langle D_x \rangle^{1/2}u^{(j_0)} \|_{s_0} & \leq \| \langle D_x \rangle^{1/2}\underline{\Delta_{12}A}\langle D_x \rangle^{1/2}u^{(j_0)} \|_{ s_0}\\
 & \leq  \| \langle D_x \rangle^{1/2}\underline{\Delta_{12}A}\langle D_x \rangle^{1/2}  \|_{\mathcal{L}(H^{s_0})}.\\
  \end{split}
  \end{equation*}
Meanwhile,
 \begin{equation*}
 \| \langle D_x \rangle^{1/2}\underline{\Delta_{12}[A]}\langle D_x \rangle^{1/2}u^{(j_0)} \|_{s_0}=\langle j_0 \rangle |\Delta_{12}A_{j_0}^{j_0}(0)|.
 \end{equation*}
Therefore the first inequality holds. The last one follows by Lemma \ref{M-N}.
 \end{proof}

The next two lemmas will be used in the proof of Proposition \ref{Ite-red}-$(3)_n$.
\begin{lemma}\label{L(H_s0)}
Let $A, B$ be two linear operators such that
\begin{equation*}
\langle D_x \rangle^{1/2}\underline{A}\langle D_x \rangle^{1/2}, \
\langle D_x \rangle^{1/2}\underline{\langle \partial_{\varphi} \rangle^b A}\langle D_x \rangle^{1/2},
\langle D_x \rangle^{1/2}\underline{B}\langle D_x \rangle^{1/2},
  \end{equation*}
 and
\begin{equation*}
\langle D_x \rangle^{1/2}\underline{\langle \partial_{\varphi} \rangle^b B}\langle D_x \rangle^{1/2} \in \mathcal{L}(H^{s_0}).
\end{equation*}
  Then the following estimates hold:
\begin{equation*}
\begin{split}
 \| \langle D_x \rangle^{1/2}\underline{A+B}\langle D_x \rangle^{1/2}  \|_{\mathcal{L}(H^{s_0})}
\leq & \| \langle D_x \rangle^{1/2}\underline{A}\langle D_x \rangle^{1/2}  \|_{\mathcal{L}(H^{s_0})} \\
& +  \| \langle D_x \rangle^{1/2}\underline{B} \langle D_x \rangle^{1/2} \|_{\mathcal{L}(H^{s_0})},\\
\| \langle D_x \rangle^{1/2}\underline{AB}\langle D_x \rangle^{1/2}  \|_{\mathcal{L}(H^{s_0})}
\leq  & \| \langle D_x \rangle^{1/2}\underline{A}  \langle D_x \rangle^{1/2}\|_{\mathcal{L}(H^{s_0})} \\
 &\times \| \langle D_x \rangle^{1/2}\underline{B}\langle D_x \rangle^{1/2}  \|_{\mathcal{L}(H^{s_0})},\\
  \| \langle D_x \rangle^{1/2}\underline{\langle \partial_{\varphi} \rangle^b (AB)} \langle D_x \rangle^{1/2} \|_{\mathcal{L}(H^{s_0})}
 \leq_b & \| \langle D_x \rangle^{1/2}\underline{\langle \partial_{\varphi} \rangle^b A}\langle D_x \rangle^{1/2}\|_{\mathcal{L}(H^{s_0})} \\
  & \times \| \langle D_x \rangle^{1/2}\underline{B}\langle D_x \rangle^{1/2}\|_{\mathcal{L} (H^{s_0})} \\
 & +\|\langle D_x \rangle^{1/2}\underline{\langle \partial_{\varphi} \rangle^b B}\langle D_x \rangle^{1/2}\|_{\mathcal{L}(H^{s_0})}\\
 &\times  \|\langle D_x \rangle^{1/2} \underline{A}\langle D_x \rangle^{1/2}\|_{\mathcal{L}(H^{s_0})}\\
  \end{split}
\end{equation*}
and
\begin{equation*}
\begin{split}
 \|\langle D_x \rangle^{1/2} \underline{\Pi_N^{\bot}A}\langle D_x \rangle^{1/2} \|_{\mathcal{L}(H^{s_0})}
\leq & \min \big\{ N^{-b} \| \langle D_x \rangle^{1/2}\underline{\langle \partial_{\varphi} \rangle^b A}\langle D_x \rangle^{1/2} \|_{\mathcal{L}(H^{s_0})},\\
& \quad  \quad \ \  \| \langle D_x \rangle^{1/2}\underline{A} \langle D_x \rangle^{1/2} \|_{\mathcal{L}(H^{s_0})} \big\}.\\
  \end{split}
\end{equation*}

\end{lemma}

\begin{proof}
In view of Lemma \ref{M-N}, we just need to prove the above inequalities in terms of matrix elements.  For the first inequality, one obtains
\begin{equation*}
\begin{split}
\langle j \rangle^{1/2}|(A+B)_j^{j'}(l)|\langle j' \rangle^{1/2}
\leq \langle j \rangle^{1/2}|A_j^{j'}(l)|\langle j' \rangle^{1/2}+\langle j \rangle^{1/2}|B_j^{j'}(l)|\langle j' \rangle^{1/2}.
  \end{split}
\end{equation*}
Regarding the second bound, one can check that
\begin{equation*}
\begin{split}
\langle j \rangle^{1/2}|(AB)_j^{j'}(l)|\langle j' \rangle^{1/2}
&\leq \langle j \rangle^{1/2} \sum_{j_1, l_1+l_2=l}|A_j^{j_1}(l_1)||B_{j_1}^{j'}(l_2)|\langle j' \rangle^{1/2}\\
&\leq \sum_{j_1, l_1+l_2=l}\langle j \rangle^{1/2}|A_j^{j_1}(l_1)|\langle j_1 \rangle^{1/2}  \langle j_1 \rangle^{1/2} |B_{j_1}^{j'}(l_2)|\langle j' \rangle^{1/2}\\
&=\big(\langle D_x \rangle^{1/2} \underline{A}  \langle D_x \rangle^{1/2}\langle D_x \rangle^{1/2} \underline{B}  \langle D_x \rangle^{1/2}\big)_j^{j'}(l).\\
 \end{split}
\end{equation*}
The third estimate follows by
\begin{equation*}
\begin{split}
\langle j \rangle^{1/2}\langle l  \rangle^{b}|(AB)_j^{j'}(l)|\langle j' &\rangle^{1/2}
\leq \langle j \rangle^{1/2} \langle l  \rangle^{b} \sum_{j_1, l_1+l_2=l}|A_j^{j_1}(l_1)||B_{j_1}^{j'}(l_2)|\langle j' \rangle^{1/2}\\
\leq&_b \sum_{j_1, l_1+l_2=l}\langle j \rangle^{1/2}\big(\langle l_1  \rangle^{b} +\langle l_2  \rangle^{b}\big) |A_j^{j_1}(l_1)| |B_{j_1}^{j'}(l_2)|\langle j' \rangle^{1/2}\\
\leq&_b \sum_{j_1, l_1+l_2=l} \Big(\langle j \rangle^{1/2}\langle l_1  \rangle^{b} |A_j^{j_1}(l_1)|\langle j_1 \rangle^{1/2}  \langle j_1 \rangle^{1/2} |B_{j_1}^{j'}(l_2)|\langle j' \rangle^{1/2}\\
&+\langle j \rangle^{1/2} |A_j^{j_1}(l_1)|\langle j_1 \rangle^{1/2}  \langle j_1 \rangle^{1/2}\langle l_2  \rangle^{b} |B_{j_1}^{j'}(l_2)|\langle j' \rangle^{1/2}\Big)\\
\leq&_b \big(\langle D_x \rangle^{1/2}\langle \partial_{\varphi}\rangle^b \underline{A}  \langle D_x \rangle^{1/2}\langle D_x \rangle^{1/2} \underline{B}  \langle D_x \rangle^{1/2}\big)_j^{j'}(l)\\
&+ \big(\langle D_x \rangle^{1/2} \underline{A}  \langle D_x \rangle^{1/2}\langle D_x \rangle^{1/2} \langle \partial_{\varphi}\rangle^b\underline{B}  \langle D_x \rangle^{1/2}\big)_j^{j'}(l).\\
 \end{split}
\end{equation*}
For the last inequality, from \eqref{smo-c}  it follows that
\begin{equation*}
\langle j \rangle^{1/2}|(\Pi_N^{\bot}A)_j^{j'}(l)|\langle  j' \rangle^{1/2}
\leq \langle j \rangle^{1/2}|(A)_j^{j'}(l)|\langle j' \rangle^{1/2},
\end{equation*}
and for $|l| \geq N$,
\begin{equation*}
\begin{split}
\langle j \rangle^{1/2}|(\Pi_N^{\bot}A)_j^{j'}(l)|\langle j' \rangle^{1/2}
&\leq \langle j \rangle^{1/2}N^{-b}\langle l  \rangle^{b}|(\Pi_N^{\bot}A)_j^{j'}(l)|\langle j' \rangle^{1/2}\\
&\leq \langle j \rangle^{1/2}N^{-b}\langle l  \rangle^{b}|A_j^{j'}(l)|\langle j' \rangle^{1/2}\\
&=N^{-b}\big(\langle D_x \rangle^{1/2}\langle \partial_{\varphi}\rangle^b \underline{A}  \langle D_x \rangle^{1/2}\big)_j^{j'}(l).\\
 \end{split}
\end{equation*}

\end{proof}

As a conclusion of this lemma, we have the following result.
\begin{lemma}\label{L(Hs0)-inv}
Let $\Phi_i:={\rm{Id}}+\Psi_i$, $i=1,2$, such that
\begin{equation*}
 \| \langle D_x \rangle^{1/2} \underline{\Psi_i}  \langle D_x \rangle^{1/2} \|_{\mathcal{L}(H^{s_0})} \leq 1/2, \quad i=1,2.
 \end{equation*}
Then the operators $\Phi_i (i=1,2)$ are invertible, and $\breve{\Psi}_i=\Phi_i^{-1}-{\rm Id} (i=1,2)$ satisfy
\begin{equation*}
\| \langle D_x \rangle^{1/2}(\underline{\breve{\Psi}_1-\breve{\Psi}_2})\langle D_x \rangle^{1/2}\|_{\mathcal{L}(H^{s_0})} \leq 4 \| \langle D_x \rangle^{1/2}(\underline{\Psi_1-\Psi_2})\langle D_x \rangle^{1/2} \|_{\mathcal{L}(H^{s_0})},
\end{equation*}
and
\begin{equation*}
\begin{split}
 \|\langle D_x \rangle^{1/2}& \underline{\langle \partial_{\varphi} \rangle^b(\breve{\Psi}_1-\breve{\Psi}_2)}\langle D_x \rangle^{1/2}\|_{\mathcal{L}(H^{s_0})}
\leq_b   \| \langle D_x \rangle^{1/2}\underline{\langle \partial_{\varphi} \rangle^b(\Psi_1-\Psi_2)}\langle D_x \rangle^{1/2} \|_{\mathcal{L}(H^{s_0})}\\
& + \big(1+\| \langle D_x \rangle^{1/2} \underline{\langle \partial_{\varphi} \rangle^b\breve{\Psi}_1}\langle D_x \rangle^{1/2}\|_{\mathcal{L}(H^{s_0})}
+ \| \langle D_x \rangle^{1/2}\underline{ \langle \partial_{\varphi} \rangle^b \breve{\Psi}_2}\langle D_x \rangle^{1/2} \|_{\mathcal{L}(H^{s_0})}\big)\\
  &\times \| \langle D_x \rangle^{1/2} ( \underline{\Psi_1-\Psi_2})\langle D_x \rangle^{1/2} \|_{\mathcal{L}(H^{s_0})}.\\
 \end{split}
\end{equation*}
\end{lemma}

%%%%%%%%%%%%%%%%%%%%%%%%%%%%%%%%%%%%%%%%%%%%%%%%%%%%%%%%%%%%
\renewcommand{\theequation}{\thesection.\arabic{equation}}
\setcounter{equation}{0}
%%%%%%%%%%%%%%%%%%%%%%%%%%%%%%%%%%%%%%%%%%%%%%%%%%%%%%%%%%%%
\section{ Regularization}

In this section, our goal is to conjugate the linear operator $\mathcal{L}$ in \eqref{L} to the operator $\hat{\mathcal{L}}$ in \eqref{hat-L}
which is  a time-dependent diagonal operator up to  a remainder whose norm is controlled by $\varepsilon {\gamma}^{-1}\|\mathfrak{J}\|_{s+\sigma}$ for some $\sigma$.  Notice that we will discuss this conjugation in the generalized $x$-variable phase space $H^1(\mathbb{T})$ instead of $H^1_0(\mathbb{T})$ temporarily in Section 3 and 4, see Remark \ref{h0s}.

 The change of variable that we construct in this section is of the type: 
 $$ h=A(\varphi, x)g.$$ 
 It can be viewed as acting on the $x$-variable phase spaces $H^s_x$ that depend quasi-periodically on time.

 The operator associated with the system  \eqref{ht=Lh} acting on quasi-periodic functions is
  \begin{equation*}
 \tilde{\mathcal{L}}:= \omega \cdot \partial_{\varphi} - L(\varphi).
 \end{equation*}
Under the action of transformation $h=A(\varphi, x)g$, system \eqref{ht=Lh} is transformed into the new one
\begin{equation}\label{gt=L+g}
 \partial_t g = L^{+}(\omega t) g, \quad L^{+}(\omega t)=A(\omega t)^{-1} L(\omega t)A(\omega t)-A(\omega t)^{-1} \{\partial_t A(\omega t)\}.
  \end{equation}
Meanwhile, the  operator  associated with system \eqref{gt=L+g} is
\begin{equation*}
A(\omega t)^{-1} \tilde{\mathcal{L}} A(\omega t)= \omega \cdot \partial_{\varphi} - L^{+}(\varphi).
\end{equation*}

In particular, we consider a $\varphi$-dependent family of diffeomorphisms of the 1-dimensional torus $\mathbb{T}$ of the form
$y=x+\beta(\varphi, x)$ where $\beta(\varphi, x)$ is a sufficiently smooth function to be determined.
The change of the space variable induces the linear operator $\mathcal{A}=\mathcal{A}_{\beta}$:
\begin{equation}\label{tr-A}
(\mathcal{A}g)(\varphi, x):=g(\varphi, x+\beta(\varphi, x)).
\end{equation}
Note that the operator $\mathcal{A}$ is invertible, with the inverse
 \begin{equation*}
 (\mathcal{A}^{-1}h)(\varphi, y):=h(\varphi, y+\widetilde{\beta}(\varphi, y)),
 \end{equation*}
where $\widetilde{\beta}$ is such that $y=x+\beta(\varphi, x)$ is the inverse of the diffeomorphism of the 1-dimensional torus $x=y+\widetilde{\beta}(\varphi, y)$.

To determine the conjugated operator $\mathcal{A}^{-1} \mathcal{L} \mathcal{A}$, we need the following result.

\begin{prop}\label{Conj} (Conjugation)
Let $\mathcal{O} \subset \mathbb{R}^{\nu}$ be  compact. Given $\rho \geq 3, \alpha \in \mathbb{N}, \mathcal{S}>s_0, s_0\leq s \leq \mathcal{S}.$ For the linear operators in \eqref{L},
 $m_0, m_2$ are real constants,
 \begin{equation*}
 a_i(\cdot, \cdot, \omega, \mathfrak{J}(\omega)) \in C^{\infty}(\mathbb{T}^{\nu+1}), \quad i=0,2
 \end{equation*}
  are even and real
valued functions depending on $\omega \in \mathcal{O}$ in a Lipschitz way and  also on the variable $\mathfrak{J}$.
There exist positive constants $\mu_1 \geq \tilde{\mu}_1>0, \delta \in (0, 1)$ such that if
 \begin{equation}\label{beta-ai-sc}
\| \beta \|_{s_0+\mu_1}^{\gamma, \mathcal{O}},\| a_i \|_{s_0+\mu_1}^{\gamma, \mathcal{O}}  \leq \delta, \quad i=0,2,
 \end{equation}
then under the invertible map $\mathcal{A}$ in \eqref{tr-A},
$\mathcal{L}$ is conjugated to the real   operator
 \begin{equation}\label{L+}
 \mathcal{L}^{+}:=\mathcal{A}^{-1}\mathcal{L}\mathcal{A}=\omega \cdot \partial_{\varphi} +  T_1(\varphi, y) \partial_{y} -(m_0+m_2)\Lambda \partial_{y}+ R,
 \end{equation}
 where
\begin{equation*}
\begin{aligned}
T_1(\varphi, y)=&(\omega \cdot \partial_{\varphi} \beta)(\varphi, y+\tilde{\beta}(\varphi, y))\\
&\qquad +\big(m_2+a_2(\varphi, y+\tilde{\beta}(\varphi, y))\big)\big(1+\beta_x(\varphi, y+\tilde{\beta}(\varphi, y))\big).
\end{aligned}
\end{equation*}
 For all $s_0\leq s \leq \mathcal{S}$, the remainder $R := {\rm Op}(r) + \mathfrak{R} $ where $r \in S^{-1}$ (see Definition \ref{pd-def}) and $\mathfrak{R} \in \mathfrak{L}_{\rho, p}$ (see Definition \ref{L-rhop}) satisfies the estimates:
 \begin{eqnarray}\label{r-R-bd}
 \begin{aligned}
 &|r|_{-1, s, \alpha}^{\gamma, \mathcal{O}} \leq_{s, \alpha, \rho}   \| \beta \|_{s+\mu_1}^{\gamma, \mathcal{O}}
     + \sum_{i=0,2}  \| a_i \|_{s+\mu_1}^{\gamma, \mathcal{O}},\\
 &\mathbb{M}_{\mathfrak{R}}^{\gamma}(s, b) \leq_{s, \rho}  \| \beta \|_{s+\mu_1}^{\gamma, \mathcal{O}}
     + \sum_{i=0,2}  \| a_i \|_{s+\mu_1}^{\gamma, \mathcal{O}}, \quad 0 \leq b \leq \rho-2.
 \end{aligned}
 \end{eqnarray}
 Moreover, for $s_0 \leq p \leq s_0+\mu_1- \tilde{\mu}_1$, there hold
 \begin{eqnarray}\label{delta-r-R-bd}
 \begin{aligned}
&|\Delta_{12}r|_{-1, p, \alpha} \leq_{p, \alpha, \rho} \| \Delta_{12} \beta \|_{p+\mu_1}+\sum_{i=0,2} \| \Delta_{12} a_i \|_{p+\mu_1},\\
&\mathbb{M}_{\Delta_{12}\mathfrak{R}}(p, b) \leq_{p,  \rho}  \| \Delta_{12}\beta \|_{p+\mu_1}
     + \sum_{i=0,2}  \| \Delta_{12}a_i \|_{p+\mu_1}, \quad 0 \leq b \leq \rho-3.
 \end{aligned}
 \end{eqnarray}
   Finally, if $\beta$ is an odd function, then $\mathcal{A}^{\pm 1}$ are reversibility-preserving, $R, \mathcal{L}^{+}$ are  reversible.
\end{prop}

\begin{proof}

  For the multiplication operator $f(\varphi, x): h(\varphi, x)\mapsto f(\varphi, x)h(\varphi, x)$, the conjugate
 \begin{equation*}
 \mathcal{A}^{-1} f(\varphi, x)\mathcal{A}=(\mathcal{A}^{-1}f)(\varphi, y)=  f(\varphi, y+\tilde{\beta}(\varphi, y))
 \end{equation*}
  is still a multiplication operator.

For the differential operators, one can  directly verify the following results:
\begin{equation*}
\begin{split}
 & \mathcal{A}^{-1} \omega \cdot \partial_{\varphi} \mathcal{A}= \omega \cdot \partial_{\varphi} + (\mathcal{A}^{-1}(\omega \cdot \partial_{\varphi} \beta)) \partial_{y},\\
&\mathcal{A}^{-1} \partial_x \mathcal{A}= (\mathcal{A}^{-1}(1+\beta_x))\partial_{y},\\
&\mathcal{A}^{-1} \partial_{xx} \mathcal{A}= (\mathcal{A}^{-1}(1+\beta_x)^2)\partial_{yy} + (\mathcal{A}^{-1}\beta_{xx})\partial_{y},\\
&\mathcal{A}^{-1} \Lambda \mathcal{A}= ({\rm{Id}}-\Lambda K)^{-1} \Lambda (1+\tilde{\beta}_y)^2,\\
 \end{split}
\end{equation*}
where
\begin{equation*}
K := f_0(\varphi, y)+f_1(\varphi, y)\partial_y,
\end{equation*}
\begin{equation*}
f_0:=(\mathcal{A}^{-1}\frac{\beta_x^2+2\beta_x}{(1+\beta_x)^2})=(\mathcal{A}^{-1}(\beta_x^2+2\beta_x))(1+\tilde{\beta}_y)^2,
\end{equation*}
\begin{equation*}
f_1:=(\mathcal{A}^{-1}\frac{\beta_{xx}}{(1+\beta_x)^2})=(\mathcal{A}^{-1}\beta_{xx}) (1+\tilde{\beta}_y)^2.
\end{equation*}
Note that all the coefficients $(\mathcal{A}^{-1}(\cdots))$ are functions of $(\varphi, y)$.
A direct computation implies
\begin{equation*}
\mathcal{L}^{+}:=\mathcal{A}^{-1}\mathcal{L}\mathcal{A}=\omega \cdot \partial_{\varphi}+T_1(\varphi, y) \partial_{y}+T_0 (\varphi, y)+\mathcal{R},
\end{equation*}
where
\begin{equation*}
T_1(\varphi, y):=(\mathcal{A}^{-1}(\omega \cdot \partial_{\varphi} \beta))+(\mathcal{A}^{-1}((m_2+a_2)(1+\beta_x))),
 \end{equation*}
\begin{equation*}
T_0(\varphi, y):=g_0 g_1- f_1 g_0 g_2- g_0 (g_2)_y -2(g_0)_y g_2 + g_0 g_3.
\end{equation*}
Here we use the notations for succinct writing:
\begin{equation*}
\begin{split}
 & g_0:=1+\tilde{\beta}_y, \quad g_1:=(\mathcal{A}^{-1}((m_2+a_2)_x(1+\beta_x))),\\
& g_2:=(\mathcal{A}^{-1}((m_2+a_2)(1+\beta_x)^2)), \quad g_3:=(\mathcal{A}^{-1}((m_2+a_2)\beta_{xx})).\\
 \end{split}
 \end{equation*}
In fact, $T_0(\varphi, y)$ can be reduced to $0$ by exploiting the following formulas:
 \begin{equation*}
 (1+\tilde{\beta}_y)\cdot (\mathcal{A}^{-1}(1+\beta_x))\equiv 1,
 \quad \tilde{\beta}_{yy}=\frac{-(\mathcal{A}^{-1}\beta_{xx})(1+\tilde{\beta}_y)^2}{(\mathcal{A}^{-1}(1+\beta_x))}.
 \end{equation*}

Note that $\mathcal{R}$ consists of all the terms of pseudo differentials of order $\leq -1$. Extracting from $\mathcal{R}$ the terms which are not "small", precisely, the terms which are irrelevant to $\beta, \tilde{\beta}$ or $a_i$ $(i=0,2)$, we obtain
\begin{equation*}
\mathcal{R}:=-m_0 \Lambda \partial_{y}   - m_2 \Lambda \partial_{y}+R,
\end{equation*}
where the remainder is given by
\begin{eqnarray*}
\begin{aligned}
R:=&r_1-[({\rm{Id}}-\Lambda K)^{-1}-{\rm{Id}}] \Lambda g_0 \partial_{y} \circ (\mathcal{A}^{-1}(m_0+a_0))\\
&\quad -[({\rm{Id}}-\Lambda K)^{-1}-{\rm{Id}}-\Lambda K]\Lambda g_0 \partial_{y} g_2\partial_{yy}
\\&\quad -[({\rm{Id}}-\Lambda K)^{-1}-{\rm{Id}}] \Lambda g_0 \partial_{y}(g_1+g_3)\partial_{y},\\
r_1:=&-m_0\Lambda \tilde{\beta}_y \partial_{y}-\Lambda g_0 \partial_{y}(\mathcal{A}^{-1}a_0)
-\Lambda(\mathcal{A}^{-1}(a_2(1+\beta_x)))\partial_{y}\\
&-m_2\Lambda(\mathcal{A}^{-1}\beta_x)\partial_{y}+\Lambda (g_0)_{yy} g_2 \partial_{y}
+\Lambda g_0 \partial_{y} (g_1)_{y} +2 \Lambda (g_0)_{y}\partial_{y}g_1 \\&+\Lambda (g_0)_{yy}g_1 -\Lambda g_0 g_1-\Lambda f_0 \Lambda g_0 \partial_{y} g_2 \partial_{yy}
+ \Lambda f_1 \partial_{y} \Lambda g_0 \partial_{y} (g_2)_{y}\partial_{y}
\\&+2\Lambda f_1 \partial_{y} \Lambda (g_0)_{y} \partial_{y}g_2\partial_{y} + \Lambda f_1 \partial_{y} \Lambda (g_0)_{yy}g_2\partial_{y}
-\Lambda f_1 \partial_{y} (g_0g_2)_y\\
&-\Lambda f_1 \partial_{y}\Lambda g_0 g_2 \partial_{y} -2\Lambda (f_1)_y \partial_{y} g_0 g_2 - \Lambda (f_1)_{yy} g_0 g_2 +\Lambda f_1 g_0 g_2
\\
&-\Lambda g_0 \partial_{y} (g_2)_{yy} - 2\Lambda(g_0)_{y} \partial_{y} (g_2)_{y}+\Lambda g_0 (g_2)_{y}
-2\Lambda(g_0)_{y} \partial_{y} (g_2)_{y}\\
& -4\Lambda (g_0)_{yy}\partial_{y}g_2 -2\Lambda(g_0)_{yyy}g_2 - \Lambda(g_0)_{yy}(g_2)_{y}\\
&+2\Lambda(g_0)_{y}g_2+\Lambda g_0 \partial_{y}(g_3)_{y} +2\Lambda(g_0)_{y}\partial_{y}g_3 + \Lambda (g_0)_{yy}g_3
-\Lambda g_0g_3.\\
\end{aligned}
\end{eqnarray*}

In order to get a sharp estimate for $R$, its structure should be explored. According to the assumption, there exist  $\mu_1 \geq \tilde{\mu}_1>0, \delta>0$ such that  if \eqref{beta-ai-sc} is satisfied,
 then for  $s_0 \geq s \geq \mathcal{S}$, $s_0 \leq p \leq s_0+\mu_1- \tilde{\mu}_1,$ the following estimates hold.

From Lemma \ref{Chan-var}, we have $\| \tilde{\beta} \|_{s}^{\gamma, \mathcal{O}} \leq  \| \beta \|_{s}^{\gamma, \mathcal{O}}$ by \eqref{gam-qp} and
$\| \Delta_{12}\tilde{\beta} \|_{p} \leq \| \Delta_{12}\beta \|_{p}$ by \eqref{del-qp}, and
\begin{equation*}
\begin{split}
&\| g_0 \|_{s}^{\gamma, \mathcal{O}} \leq 1+ \| \beta \|_{s+\mu_1}^{\gamma, \mathcal{O}}, \quad
\| (g_0)_y \|_{s}^{\gamma, \mathcal{O}}= \| \tilde{\beta}_{yy} \|_{s}^{\gamma, \mathcal{O}}\leq  \| \beta \|_{s+\mu_1}^{\gamma, \mathcal{O}}.\\
&\| \Delta_{12}g_0 \|_{p}= \| \Delta_{12}\tilde{\beta}_y \|_{p}  \leq \| \Delta_{12}\beta \|_{p+\mu_1}.
\end{split}
\end{equation*}
It follows from \eqref{gam-uf} that
\begin{equation*}
\begin{split}
&\| \mathcal{A}^{-1}\beta_x \|_{s}^{\gamma, \mathcal{O}}  \leq_s \| \beta \|_{s+\mu_1}^{\gamma, \mathcal{O}},   \quad
 \| \mathcal{A}^{-1}a_0 \|_{s}^{\gamma, \mathcal{O}}  \leq_s \| a_0 \|_{s+\mu_1}^{\gamma, \mathcal{O}}
+\| \beta \|_{s+\sigma_1}^{\gamma, \mathcal{O}},\\
&\| \mathcal{A}^{-1}a_2 \|_{s}^{\gamma, \mathcal{O}}  \leq_s \| a_2 \|_{s+\mu_1}^{\gamma, \mathcal{O}}
+\| \beta \|_{s+\mu_1}^{\gamma, \mathcal{O}},\\
\end{split}
\end{equation*}
and from \eqref{del-uf} we derive
\begin{equation*}
\begin{split}
&\| \Delta_{12}(\mathcal{A}^{-1}\beta_x) \|_{p} = \| \Delta_{12}(\mathcal{A}^{-1}(1+\beta_x)) \|_{p}   \leq_p \| \Delta_{12}\beta \|_{p+\mu_1},\\
&\| \Delta_{12}(\mathcal{A}^{-1}a_0) \|_{p}  \leq_p \| \Delta_{12}\beta \|_{p+\mu_1}+\| \Delta_{12}a_0 \|_{p+\mu_1},\\
&\| \Delta_{12}(\mathcal{A}^{-1}a_2) \|_{p}  \leq_p \| \Delta_{12}\beta \|_{p+\mu_1}+\| \Delta_{12}a_2 \|_{p+\mu_1}.\\
\end{split}
\end{equation*}
By Lemma \ref{Ip}, \eqref{gam-uf} and $(g_2)_y=g_1+2g_3$, we obtain
\begin{equation*}
\begin{split}
 &\| g_1 \|_{s}^{\gamma, \mathcal{O}}, \| g_3 \|_{s}^{\gamma, \mathcal{O}}, \| (g_2)_y \|_{s}^{\gamma, \mathcal{O}}, \| \mathcal{A}^{-1}(a_2(1+\beta_x)) \|_{s}^{\gamma, \mathcal{O}}
 \leq_s \| a_2 \|_{s+\mu_1}^{\gamma, \mathcal{O}} + \| \beta \|_{s+\mu_1}^{\gamma, \mathcal{O}},\\
 & \| g_2 \|_{s}^{\gamma, \mathcal{O}} \leq_s 1+ \| a_2 \|_{s+\mu_1}^{\gamma, \mathcal{O}} + \| \beta \|_{s+\mu_1}^{\gamma, \mathcal{O}},\quad
  \| f_0 \|_{s}^{\gamma, \mathcal{O}}, \| f_1 \|_{s}^{\gamma, \mathcal{O}} \leq_s  \| \beta \|_{s+\mu_1}^{\gamma, \mathcal{O}}.\\
\end{split}
\end{equation*}
From $\Delta_{12}(uv)=u(\Delta_{12}v)+(\Delta_{12}u)v$, we deduce that
\begin{equation*}
\begin{split}
&\| \Delta_{12}g_i \|_{p}(i=1,2,3),  \| \Delta_{12}(\mathcal{A}^{-1}(a_2(1+\beta_x))) \|_{p} \leq_p \| \Delta_{12}\beta \|_{p+\mu_1}+ \| \Delta_{12}a_2 \|_{p+\mu_1},\\
&\| \Delta_{12}f_j \|_{p}(j=0,1) \leq_p \| \Delta_{12}\beta \|_{p+\mu_1}.\\
\end{split}
\end{equation*}
In view of \eqref{pd-c0}-\eqref{pd-c1} and Lemma \ref{pd-com}, the following estimates hold:
\begin{equation*}
\begin{split}
 &|K|_{1,s,\alpha}^{\gamma, \mathcal{O}}, |\Lambda K|_{-1,s,\alpha}^{\gamma, \mathcal{O}} \leq_{s, \alpha} \| \beta \|_{s+\mu_1}^{\gamma, \mathcal{O}},\\
 &|r_1|_{-1,s,\alpha}^{\gamma, \mathcal{O}} \leq_{s, \alpha} \| \beta \|_{s+\mu_1}^{\gamma, \mathcal{O}}+ \| a_0 \|_{s+\mu_1}^{\gamma, \mathcal{O}} +\| a_2 \|_{s+\mu_1}^{\gamma, \mathcal{O}},\\
&| \Delta_{12}r_1 |_{-1,p,\alpha}   \leq_{p, \alpha} \| \Delta_{12}\beta \|_{p+\mu_1}+\| \Delta_{12}a_0 \|_{p+\mu_1}+  \| \Delta_{12}a_2 \|_{p+\mu_1},\\
& |\Delta_{12}K |_{1,p,\alpha},  |\Delta_{12}(\Lambda K) |_{-1,p,\alpha} \leq_{p, \alpha} \| \Delta_{12}\beta \|_{p+\mu_1}.\\
\end{split}
\end{equation*}

Similarly, by \eqref{pd-c0}, \eqref{pd-c1}, \eqref{gam-uf} and Lemma \ref{pd-com}, we have
\begin{equation*}
\begin{split}
&|\Lambda g_0 \partial_{y} \circ (\mathcal{A}^{-1}(m_0+a_0))|_{-1, s, \alpha}^{\gamma, \mathcal{O}}\leq_{s, \alpha} 1+\| \beta \|_{s+\mu_1}^{\gamma, \mathcal{O}}+\| a_0 \|_{s+\mu_1}^{\gamma, \mathcal{O}},\\
&|\Lambda g_0 \partial_{y} g_2\partial_{yy}|_{1, s, \alpha}^{\gamma, \mathcal{O}}\leq_{s, \alpha} 1+\| \beta \|_{s+\mu_1}^{\gamma, \mathcal{O}}+\| a_2 \|_{s+\mu_1}^{\gamma, \mathcal{O}},\\
&|\Lambda g_0 \partial_{y}(g_1+g_3)\partial_{y}|_{0, s, \alpha}^{\gamma, \mathcal{O}}\leq_{s, \alpha} \| \beta \|_{s+\mu_1}^{\gamma, \mathcal{O}}+\| a_2 \|_{s+\mu_1}^{\gamma, \mathcal{O}},\\
&|\Delta_{12}\big(\Lambda g_0 \partial_{y} \circ (\mathcal{A}^{-1}(m_0+a_0))\big)|_{-1, p, \alpha} \leq_{p, \alpha} \| \Delta_{12}\beta \|_{p+\mu_1}+\| \Delta_{12}a_0 \|_{p+\mu_1},\\
&|\Delta_{12}(\Lambda g_0 \partial_{y} g_2\partial_{yy})|_{1, p, \alpha}, |\Delta_{12}\big(\Lambda g_0 \partial_{y}(g_1+g_3)\partial_{y}\big)|_{0, p, \alpha} \\&\quad \leq_{p, \alpha} \| \Delta_{12}\beta \|_{p+\mu_1}+\| \Delta_{12}a_2 \|_{p+\mu_1}.\\
\end{split}
\end{equation*}

 According to Lemma \ref{pd-rhop-inv}, the small condition of the operator $\Lambda K$ is fulfilled. Hence we have
  \begin{equation*}
  \begin{aligned}
 & ({\rm{Id}}-\Lambda K)^{-1}={\rm{Id}}+{\rm Op}(c_1)+\mathfrak{R}_{\rho+1},\\
  &  ({\rm{Id}}-\Lambda K)^{-1}-{\rm{Id}}-\Lambda K={\rm Op}(c_2)+\mathfrak{R}_{\rho+1},
 \end{aligned}
 \end{equation*}
   where $c_1 \in S^{-1}, c_2 \in S^{-2}, \mathfrak{R}_{\rho+1} \in \mathfrak{L}_{\rho+1, p}$. They satisfy the following inequalities:
   \begin{equation*}
\begin{split}
   &|c_1|_{-1, s, \alpha}^{\gamma, \mathcal{O}}, |c_2|_{-2, s, \alpha}^{\gamma, \mathcal{O}} \leq_{s, \alpha, \rho} \| \beta \|_{s+\mu_1}^{\gamma, \mathcal{O}},\\
&|\Delta_{12}c_1|_{-1, p, \alpha}, |\Delta_{12}c_2|_{-2, p, \alpha} \leq_{p, \alpha, \rho}   \| \Delta_{12}\beta \|_{p+\mu_1},\\
&\mathbb{M}_{\mathfrak{R}_{\rho+1}}^{\gamma}(s, b)\leq_{s,  \rho} \| \beta \|_{s+\mu_1}^{\gamma, \mathcal{O}}, \quad  0 \leq b \leq \rho-1,\\
&\mathbb{M}_{\Delta_{12}\mathfrak{R}_{\rho+1}}^{\gamma}(p, b)\leq_{p, \rho} \| \Delta_{12}\beta \|_{p+\mu_1}, \quad  0 \leq b \leq \rho-2.\\
\end{split}
\end{equation*}
Hence $R$ has the decomposition: $R:=r_1+r_2+r_3+r_4+\mathfrak{R}_1+\mathfrak{R}_2+\mathfrak{R}_3$,
 where
    \begin{equation*}
\begin{split}
&r_2:= - c_1 \Lambda g_0 \partial_{y} \circ (\mathcal{A}^{-1}(m_0+a_0)),  \quad
r_3:=-c_2 \Lambda g_0 \partial_{y} g_2\partial_{yy},\\
&r_4:=-c_1 \Lambda g_0 \partial_{y}(g_1+g_3)\partial_{y},   \quad
\mathfrak{R}_1:=- \mathfrak{R}_{\rho+1} \Lambda g_0 \partial_{y} \circ (\mathcal{A}^{-1}(m_0+a_0)),\\
&\mathfrak{R}_2:=-\mathfrak{R}_{\rho+1} \Lambda g_0 \partial_{y} g_2\partial_{yy},\quad
\mathfrak{R}_3:=-\mathfrak{R}_{\rho+1}\Lambda g_0 \partial_{y}(g_1+g_3)\partial_{y}.\\
\end{split}
\end{equation*}
Then from Lemma \ref{pd-com},  it follows that
    \begin{equation*}
\begin{split}
&|r_2|_{-2, s, \alpha}^{\gamma, \mathcal{O}}\leq_{s, \alpha, \rho} \| \beta \|_{s+\mu_1}^{\gamma, \mathcal{O}}+\| a_0 \|_{s+\mu_1}^{\gamma, \mathcal{O}},\\
&|r_3|_{-1, s, \alpha}^{\gamma, \mathcal{O}}, |r_4|_{-1, s, \alpha}^{\gamma, \mathcal{O}}\leq_{s, \alpha, \rho} \| \beta \|_{s+\mu_1}^{\gamma, \mathcal{O}}+\| a_2 \|_{s+\mu_1}^{\gamma, \mathcal{O}},\\
&|\Delta_{12}r_2|_{-2, p, \alpha} \leq_{p, \alpha} \| \Delta_{12}\beta \|_{p+\mu_1} +\| \Delta_{12}a_0 \|_{p+\mu_1},\\
&|\Delta_{12}r_3|_{-1, p, \alpha}, |\Delta_{12}r_4|_{-1, p, \alpha} \leq_{p, \alpha} \| \Delta_{12}\beta \|_{p+\mu_1} +\| \Delta_{12}a_2 \|_{p+\mu_1}.\\
\end{split}
\end{equation*}
From Lemma \ref{pd-rhop-com}, we also have $\mathfrak{R}_{i} \in \mathfrak{L}_{\rho, p}, i=1,2,3$ and
   \begin{equation*}
\begin{split}
&\mathbb{M}_{\mathfrak{R}_1}^{\gamma}(s, b)\leq_{s, \rho} \| \beta \|_{s+\mu_1}^{\gamma, \mathcal{O}}+\| a_0 \|_{s+\mu_1}^{\gamma, \mathcal{O}}, \quad  0 \leq b \leq \rho-2,\\
&\mathbb{M}_{\mathfrak{R}_2}^{\gamma}(s, b), \mathbb{M}_{\mathfrak{R}_3}^{\gamma}(s, b) \leq_{s, \rho}  \| \beta \|_{s+\mu_1}^{\gamma, \mathcal{O}}+\| a_2 \|_{s+\mu_1}^{\gamma, \mathcal{O}}, \quad  0 \leq b \leq \rho-2,\\
&\mathbb{M}_{\Delta_{12}\mathfrak{R}_1}(p, b)\leq_{p, \rho} \| \Delta_{12}\beta \|_{p+\mu_1}+\| \Delta_{12}a_0 \|_{p+\mu_1}, \quad  0 \leq b \leq \rho-3,\\
&\mathbb{M}_{\Delta_{12}\mathfrak{R}_2}(p, b), \mathbb{M}_{\Delta_{12}\mathfrak{R}_3}(p, b) \leq_{p, \rho} \| \Delta_{12}\beta \|_{p+\mu_1}+\| \Delta_{12}a_2 \|_{p+\mu_1}, \quad  0 \leq b \leq \rho-3.\\
\end{split}
\end{equation*}
Let $r=\sum_{i=1}^{4}r_i$, $\mathfrak{R}=\sum_{i=1}^{3}\mathfrak{R}_i$,   then \eqref{r-R-bd} and \eqref{delta-r-R-bd} hold.

%If $\beta$ is real and odd, then $\tilde{\beta}$ is real and odd too. Indeed, because $y=\Xi(x)=x+\beta$ is real and odd,  %$\tilde{\beta}=\Xi^{-1}(y)-y$ is real and odd by noting that $\Xi(x)$ and its inverse $\Xi^{-1}(y)$ have the same parity.

Recall Definition \ref{rev-def}, let $u(\varphi, y) \in X$, then $\mathcal{A}u=u(\varphi, x+\beta(\varphi, x)) \in X$ if $\beta$ is odd.
Indeed,
\begin{equation*}
u\big(-\varphi, -x+\beta(-\varphi, -x)\big)=u\big(-\varphi, -x-\beta(\varphi, x)\big)=u\big(\varphi, x+\beta(\varphi, x)\big).
 \end{equation*}
The proof of the case for $u(\varphi, y) \in Y$ is similar. Therefore, $\mathcal{A}$ is reversibility-preserving. Due to the reality of $\beta$, $\mathcal{A}$ is real.  By replacing $\beta$ with $\tilde{\beta}$, $\mathcal{A}^{-1}$ is also real and reversibility-preserving.

Finally we prove $R, \mathcal{L}^{+}$ is real and reversible. In fact,
 $\mathcal{L}^{+}$ is real due to the reality of   $\mathcal{L}$ and $\mathcal{A}^{\pm 1}$. Meanwhile, it is reversible since $\mathcal{L}$ is reversible and $\mathcal{A}^{\pm 1}$ are reversibility-preserving. In view of the reality and reversibility of the opearators
\begin{equation*}
 \omega \cdot \partial_{\varphi} + m_{\infty}\partial_{x} -(m_0+m_2)\Lambda \partial_{x}
\end{equation*}
and $\mathcal{L}^{+}$, $R$ is real and reversible.
Thus we complete the proof.
\end{proof}

In \cite{fgmp}, the reducibility of a class of linear first-order operators on tori has been proved as shown in the next proposition, i.e. the so-called straightening theorem.  Thanks to this theorem, we   choose an appropriate function $\beta$ such that the coefficient of the highest order spatial derivative operator $\partial_y$ of $\mathcal{L}^{+}$ in \eqref{L+} is a constant.

\begin{prop}\label{Str-thm} (Straightening theorem)\cite{fgmp}
Let $\mathcal{O}$ be a compact subset  of $\mathbb{R}^{\nu}$, and $\mathcal{X}$ be a  family of vector fields on $\mathbb{T}^{\nu+1}$
\begin{equation*}
\mathcal{X}:= \omega \cdot \frac{\partial}{\partial{\varphi}}+(m+a(\varphi, x, \omega))\cdot  \frac{\partial}{\partial{x}}, \quad \omega \in \mathcal{O}, m \in \mathbb{R},
\end{equation*}
where $a(\cdot, \cdot, \omega, \mathfrak{J}(\omega)) \in H^s(\mathbb{T}^{\nu+1}, \mathbb{R})$ $(\forall s\geq s_0)$  is  Lipschitz in  $\omega$ and also depends on the variable $\mathfrak{J}$.
There exists $\tilde{\delta}>0, s_1\geq s_0+2 \tau +4$ $(\tau\geq \nu+3),$ such that if
 \begin{equation}\label{sc-strai}
 \gamma^{-1} \| a \|_{s_1}^{\gamma, \mathcal{O}} \leq \tilde{\delta},
     \end{equation}
then  there exists  $m_{+}(\omega, \mathfrak{J}(\omega)) \in \mathbb{R}$ with
 \begin{equation*}
 |m_{+}-m|^{\gamma, \mathcal{O}} \leq \| a \|_{s_1}^{\gamma, \mathcal{O}},
     \end{equation*}
such that for all $\omega \in \mathcal{O}_{+}^{2\gamma}$, where
 \begin{equation}\label{O+-defi}
 \mathcal{O}_{+}^{2\gamma}:= \left\{\omega \in \mathcal{O}:  |\omega \cdot l+m_{+}(\omega)j|\geq \frac{2\gamma}{\langle l, j \rangle^{\tau}},  \quad \forall (l,j) \in \mathbb{Z}^{\nu+1}\setminus \{0\}\right\},
  \end{equation}
  one has
  \begin{equation*}
  |\Delta_{12}m_{+}|\leq 2|\Delta_{12}\langle a \rangle|, \quad \langle a \rangle:=\frac{1}{(2\pi)^{\nu+1}}\int_{\mathbb{T}^{\nu+1}}a {\rm d}\varphi {\rm d}x
  \end{equation*}
 and
there exists a real and smooth map $\hat{\beta}(\varphi, x, \mathfrak{J})$ and a constant $\sigma_{*}>0$ satisfying
 \begin{equation*}
 \begin{split}
 &\| \hat{\beta} \|_s^{\gamma, \mathcal{O}_{+}^{2\gamma}} \leq_s \gamma^{-1}\| a \|_{s+2\tau+4}^{\gamma, \mathcal{O}},  \quad \forall s\geq s_0,\\
 &\| \Delta_{12}\hat{\beta} \|_p\leq C\gamma^{-1}\| \Delta_{12}a \|_{p+\sigma_{*}},  \quad  s_0 \leq p+\sigma_{*}<s_1,\\
 \end{split}
\end{equation*}
so that $W: (\varphi, x)\mapsto (\varphi, x+\hat{\beta}(\varphi, x))$ is a diffeomorphism of $\mathbb{T}^{\nu+1}$ and
for all $\omega \in \mathcal{O}_{+}^{2\gamma}$,
\begin{equation*}
W_{\ast}\mathcal{X}:=\omega \cdot \frac{\partial}{\partial{\varphi}}+W^{-1}\left(\omega \cdot \partial_{\varphi}\hat{\beta}+ (m+a)(1+\hat{\beta}_x)\right)  \frac{\partial}{\partial{x}}
=\omega \cdot \frac{\partial}{\partial{\varphi}}+m_{+}\frac{\partial}{\partial{x}}.
\end{equation*}
Furthermore, if $a(\varphi, x)$ is  even, then $\hat{\beta}$ is an odd function.
\end{prop}
\begin{rmk}
Actually, the sign of inequality appearing   in \eqref{O+-defi} is "$>$" in  \cite{fgmp}. The thesis also holds if we replace "$>$" with
"$\geq$".
\end{rmk}
\begin{thm}\label{Reg-thm}(Regularization)
Let $\rho\geq 3$ and fix $\mathcal{S}>s_0, s_0\leq p < \mathcal{S}$. Suppose for some  $\mu_2\geq \tilde{\mu}_2>0$, the conditions (A1)-(A4) are satisfied with $\mu=\mu_2$. Then there  exists  a constant $m_{\infty}(\omega, \mathfrak{J}(\omega))$ which depends on $\omega \in \mathcal{O}_0$  in a Lipschitz way and the variable $\mathfrak{J}$ with
 \begin{equation}\label{m_inf-bd}
 |m_{\infty}-m_2|^{\gamma, \mathcal{O}_0}\leq C \varepsilon,
  \end{equation}
such that for all $\omega$ in the set
 \begin{equation}\label{O-1}
  \mathcal{O}_1:= \left\{\omega \in \mathcal{O}_0:  |\omega \cdot l+m_{\infty}(\omega)j|\geq \frac{2\gamma}{\langle l, j \rangle^{\tau}}, \quad \forall (l,j) \in \mathbb{Z}^{\nu+1}\setminus \{0 \}\right\},
      \end{equation}
there exists a real bounded linear operator $\Upsilon_1$ such that
  \begin{equation}\label{hat-L}
 \hat{ \mathcal{L}} = \Upsilon_1^{-1} \mathcal{L} \Upsilon_1=\omega \cdot \partial_{\varphi} + m_{\infty}\partial_{x} -(m_0+m_2)\Lambda \partial_{x}+ \hat{R},
  \end{equation}
where the remainder $\hat{R}={\rm Op}(\hat{r})+\hat{\mathfrak{R}}$ with $\hat{r} \in S^{-1}, \hat{\mathfrak{R}} \in \mathcal{L}_{\rho, p}$ satisfies
  \begin{equation}\label{hat-r-R-bd}
  |\hat{r}|_{-1, s, \alpha}^{\gamma, \mathcal{O}_1}, \mathbb{M}_{\hat{\mathfrak{R}}}^{\gamma}(s, b) \leq_{s,\alpha,\rho} \varepsilon \gamma^{-1}
\| \mathfrak{J} \|_{s+\mu_2}^{\gamma, \mathcal{O}_0},        \quad s_0\leq s \leq \mathcal{S}, 0 \leq b \leq \rho-2.
 \end{equation}
For $\omega \in \mathcal{O}_1(\mathfrak{J}_1)\cap \mathcal{O}_1(\mathfrak{J}_2)$,  $s_0 \leq p \leq s_0+\mu_2- \tilde{\mu}_2$ and $0 \leq b \leq \rho-3$, the following estimates hold:
  \begin{equation*}
  |\Delta_{12}m_{\infty}| \leq \varepsilon \| \mathfrak{J}_1 -\mathfrak{J}_2 \|_{s_0+\mu_2},
  \end{equation*}
  and
  \begin{equation}\label{det-r-R-bd}
  |\Delta_{12}\hat{r}|_{-1, p, \alpha}, \mathbb{M}_{\Delta_{12}\hat{\mathfrak{R}}}(p, b) \leq_{p,\alpha,\rho} \varepsilon \gamma^{-1}
\| \mathfrak{J}_1 -\mathfrak{J}_2 \|_{p+\mu_2}.
\end{equation}
 In addition, if $u$ depends in a Lipschitz way on $\omega \in \mathcal{O}_1$, then
 \begin{equation}\label{Ups1u}
\| \Upsilon_1^{\pm 1}u \|_s^{\gamma, \mathcal{O}_1} \leq_s  \| u \|_s^{\gamma, \mathcal{O}_1} + \varepsilon \gamma^{-1} \| \mathfrak{J} \|_{s+\mu_2}^{\gamma, \mathcal{O}_0}\| u \|_{s_0}^{\gamma, \mathcal{O}_1}.
\end{equation}
Finally, the maps $\Upsilon_1, \Upsilon_1^{-1}$ are real and reversibility-preserving, while $\hat{R},  \hat{ \mathcal{L}}$ are real and reversible.
\end{thm}
\begin{proof}
Set $a=a_2$, $m=m_2$, $\mathcal{O}=\mathcal{O}_0$ in Proposition \ref{Str-thm}, $\mu=\mu_2$ in \eqref{sc-J} large enough and $\varepsilon$ in \eqref{sc-a_i} small enough, then the condition \eqref{sc-strai} is satisfied. From Proposition \ref{Str-thm}, we derive  that there exists a constant $m_{\infty}(\omega, \mathfrak{J})$ defined on $\mathcal{O}_0$ and a function $\beta^{\infty}(\omega, \mathfrak{J})$  on $\mathcal{O}_1$ such that
\begin{equation*}
\omega \cdot \partial_{\varphi} \beta^{\infty}+(m_2+a_2)(1+\beta^{\infty}_x)= m_{\infty}.
\end{equation*}
In addition, they satisfy the following estimates:
\begin{equation*}
 |m_{\infty}-m_2|^{\gamma, \mathcal{O}_0}
 \leq \| a_2 \|_{s_1}^{\gamma, \mathcal{O}_0}
 \leq \varepsilon \| \mathfrak{J} \|_{s_1+\eta_0}^{\gamma, \mathcal{O}_0},  \quad \omega \in \mathcal{O}_0,
 \end{equation*}
  \begin{equation*}
  |\Delta_{12}m_{\infty}|
  \leq 2|\Delta_{12}\langle a \rangle|
  \leq 2 \|\Delta_{12}a_2 \|_{s_0}
  \leq 2 \varepsilon \| \mathfrak{J}_1-\mathfrak{J}_2 \|_{s_0+\eta_0}, \quad \omega \in \mathcal{O}_1(\mathfrak{J}_1)\cap \mathcal{O}_1(\mathfrak{J}_2),
  \end{equation*}
\begin{equation*}
\| \beta^{\infty} \|_s^{\gamma, \mathcal{O}_1}
\leq_s \gamma^{-1}\| a_2 \|_{s+2\tau+4}^{\gamma, \mathcal{O}_0}
\leq_s \varepsilon\gamma^{-1}\| \mathfrak{J} \|_{s+2\tau+4+\eta_0}^{\gamma, \mathcal{O}_0},  \quad \forall s \geq s_0,  \omega \in \mathcal{O}_1,
\end{equation*}
and for  $s_0 \leq p+\sigma_{*}<s_1, \omega \in \mathcal{O}_1(\mathfrak{J}_1)\cap \mathcal{O}_1(\mathfrak{J}_2)$,
\begin{equation*}
 \| \Delta_{12}\beta^{\infty} \|_p \leq C\gamma^{-1}\| \Delta_{12}a_2 \|_{p+\sigma_{*}}
\leq_p \varepsilon \gamma^{-1}\| \mathfrak{J}_1-\mathfrak{J}_2  \|_{p+\sigma_{*}+\eta_0}.
\end{equation*}

 From Proposition \ref{Conj},  setting $\beta=\beta^{\infty}$, $\Upsilon_1=\mathcal{A}_{\beta^{\infty}}$, we deduce that
\begin{equation*}
 \Upsilon_1^{-1} \mathcal{L}\Upsilon_1=\omega \cdot \partial_{\varphi} +m_{\infty}\partial_{y}  -(m_0+m_2)\Lambda \partial_{y}+ \hat{R}.
 \end{equation*}
After renaming the space variable $y=x$, we obtain \eqref{hat-L}.

According to Proposition \ref{Conj}, the remainder $\hat{R}={\rm Op}(\hat{r})+\hat{\mathfrak{R}}$, where $\hat{r} \in S^{-1}$ and $\hat{\mathfrak{R}} \in \mathfrak{L}_{\rho, p}\ $, satisfies the following inequalities:
\begin{equation*}
 |\hat{r}|_{-1, s, \alpha}^{\gamma, \mathcal{O}_1}
 \leq_{s, \alpha, \rho}   \| \beta^{\infty} \|_{s+\mu_1}^{\gamma, \mathcal{O}_1}     + \sum_{i=0,2}  \| a_i \|_{s+\mu_1}^{\gamma, \mathcal{O}_0}
\leq_{s, \alpha, \rho} \gamma^{-1}\varepsilon \| \mathfrak{J} \|_{s+\mu_1+2\tau+4+\eta_0}^{\gamma, \mathcal{O}_0}
\end{equation*}
for $s_0\leq s \leq \mathcal{S}$,
\begin{equation*}
\mathbb{M}_{\hat{\mathfrak{R}}}^{\gamma}(s, b)
\leq_{s, \rho}  \| \beta^{\infty} \|_{s+\mu_1}^{\gamma, \mathcal{O}_1}
     + \sum_{i=0,2}  \| a_i \|_{s+\mu_1}^{\gamma, \mathcal{O}_0}
\leq_{s,  \rho} \gamma^{-1}\varepsilon \| \mathfrak{J} \|_{s+\mu_1+2\tau+4+\eta_0}^{\gamma, \mathcal{O}_0}
\end{equation*}
   for $0 \leq b \leq \rho-2,  s_0\leq s \leq \mathcal{S},$
\begin{equation*}
|\Delta_{12}\hat{r}|_{-1, p, \alpha}
\leq_{p, \alpha, \rho} \| \Delta_{12} \beta^{\infty} \|_{p+\mu_1}+\sum_{i=0,2} \| \Delta_{12} a_i \|_{p+\mu_1}
\leq_{p, \alpha, \rho} \gamma^{-1}\varepsilon \| \mathfrak{J}_1-\mathfrak{J}_2 \|_{p+\mu_1+\sigma_{*}+\eta_0}
\end{equation*}
for $ s_0 \leq p \leq \min\{s_0+\mu_1- \tilde{\mu}_1, s_1-\sigma_{*}-\mu_1\}$, and
\begin{equation*}
\mathbb{M}_{\Delta_{12}\hat{\mathfrak{R}}}(p, b)
\leq_{p,  \rho}  \| \Delta_{12}\beta^{\infty} \|_{p+\mu_1}  + \sum_{i=0,2}  \| \Delta_{12}a_i \|_{p+\mu_1}
\leq_{p,  \rho} \gamma^{-1}\varepsilon \| \mathfrak{J}_1-\mathfrak{J}_2 \|_{p+\mu_1+\sigma_{*}+\eta_0}
\end{equation*}
for $ s_0 \leq p \leq \min\{s_0+\mu_1- \tilde{\mu}_1, s_1-\sigma_{*}-\mu_1\}, 0 \leq b \leq \rho-3$.
Hence there exist $\mu_2, \tilde{\mu}_2>0, \tilde{\mu}_2 \leq \mu_2$ such that
the inequalities \eqref{hat-r-R-bd}-\eqref{det-r-R-bd} hold.

Furthermore, if $u$ depends on $\omega \in \mathcal{O}_1$ in a Lipschitz way,  it follows from \eqref{gam-uf} for $ \omega \in \mathcal{O}_1\subset \mathbb{R}^{\nu}$ that
\begin{equation*}
\begin{split}
\| \Upsilon_1u \|_s^{\gamma, \mathcal{O}_1}+\| \Upsilon_1^{-1}u \|_s^{\gamma, \mathcal{O}_1}
&\leq_s\| u \|_s^{\gamma, \mathcal{O}_1}
+ \| \beta^{\infty} \|_{s_0+1}^{\gamma, \mathcal{O}_1} \| u \|_{s}^{\gamma, \mathcal{O}_1}
+\| \beta^{\infty} \|_{s+s_0}^{\gamma, \mathcal{O}_1} \| u \|_{s_0}^{\gamma, \mathcal{O}_1}\\
&\leq_s  \| u \|_s^{\gamma, \mathcal{O}_1} + \varepsilon \gamma^{-1} \| \mathfrak{J} \|_{s+\mu_2}^{\gamma, \mathcal{O}_0}\| u \|_{s_0}^{\gamma, \mathcal{O}_1}.\\
\end{split}
\end{equation*}

Finally, since $m_2 \in \mathbb{R}$ and $a_2$ is a real, even function, $\beta^{\infty}$ is real and odd due to Proposition \ref{Str-thm}. By Proposition \ref{Conj},
 the maps $\Upsilon_1, \Upsilon_1^{-1}$ are real and reversibility-preserving, while $\hat{\mathcal{L}}$, $\hat{R}$ are real and reversible.
\end{proof}
%%%%%%%%%%%%%%%%%%%%%%%%%%%%%%%%%%%%%%%%%%%%%%%%%%%%%%%%%%%%
\renewcommand{\theequation}{\thesection.\arabic{equation}}
\setcounter{equation}{0}
%%%%%%%%%%%%%%%%%%%%%%%%%%%%%%%%%%%%%%%%%%%%%%%%%%%%%%%%%%%%
\section{Diagonalization}

The purpose of this section is to completely diagonalize the linear operator $\hat{\mathcal{L}}$ in \eqref{hat-L}. The proof is carried out by exploiting an iterative KAM scheme.

First we assume that
\begin{equation}\label{k1-k2}
0 < k_1 \leq k_2,   \quad  k_2\gamma^{-\kappa}\leq 1, \quad \kappa>1.
\end{equation}
 We now consider a linear operator $\mathcal{L}_0:=\omega \cdot \partial_{\varphi} + D_0 + R_0$, where
\begin{equation}\label{D0}
D_0= {\rm diag}_{j \in \mathbb{Z}}{\rm i} d_j^{(0)}, \quad d_j^{(0)}:=m_{\infty}j-\frac{(m_0+m_2)j}{1+j^2}, \quad m_0, m_2 \in \mathbb{R},
\end{equation}
and $m_{\infty}: \mathcal{O}_0 \rightarrow \mathbb{R}$ depends on $\omega$ in a Lipschitz way satisfying
\begin{equation}\label{m-bd}
|m_{\infty}-m_2|^{\gamma, \mathcal{O}_0}\leq C k_1.
\end{equation}

In the sequel, we fix
\begin{equation}\label{a0b0-def}
a_0=6\tau+3, \quad \tau \geq \nu+3,  \quad b_0=a_0+1.
\end{equation}

We assume that $R_0, \langle \partial_{\varphi} \rangle^{b_0}R_0$ are defined on $\mathcal{O}_1$ in \eqref{O-1} and they are Lip-$-1$-modulo-tame with modulo-tame constants
\begin{equation}\label{R0-bds}
\mathfrak{M}_{R_0}^{\sharp, \gamma^{\kappa}}(-1, s_0), \mathfrak{M}_{R_0}^{\sharp, \gamma^{\kappa}}(-1, s_0, b_0) \leq k_2.
\end{equation}

Additionally, $m_{\infty}$, $R_0$ and $\langle \partial_{\varphi} \rangle^{b_0}R_0$ also depend on the variable $\mathfrak{J}(\omega)$.
For $\omega \in \mathcal{O}_1(\mathfrak{J}_1)\cap \mathcal{O}_1(\mathfrak{J}_2)$ and  some constant $\sigma>0$, there hold
  \begin{equation}\label{del-m-bd}
  |\Delta_{12}m_{\infty}| \leq k_1 \| \mathfrak{J}_1 -\mathfrak{J}_2 \|_{s_0+\sigma},
    \end{equation}
and
\begin{equation}\label{del-R0-bds}
\begin{split}
\| \langle D_x \rangle^{1/2}\underline{\Delta_{12}R_0}\langle D_x \rangle^{1/2} \|_{\mathfrak{L}(H^{s_0})}&\leq k_2  \|\mathfrak{J}_1- \mathfrak{J}_2\|_{s_0+\sigma},\\
 \| \langle D_x \rangle^{1/2}\underline{\langle \partial_{\varphi}\rangle^{b_0}\Delta_{12}R_0}\langle D_x \rangle^{1/2} \|_{\mathfrak{L}(H^{s_0})}
&\leq k_2  \|\mathfrak{J}_1- \mathfrak{J}_2\|_{s_0+\sigma}.
\end{split}
\end{equation}

% Finally, $R_0$ and $\mathcal{L}_0 $ are real and reversible operators.

Now we introduce some notations: $N_{-1}=1$, $N_0 \in \mathbb{N}$ is large enough, $N_n=N_0^{(\frac{3}{2})^n}$ for $n \in \mathbb{N}$. Clearly, $N_{n+1}=N_n^{\frac{3}{2}}$.

\subsection{Iterative reduction}

\begin{prop}\label{Ite-red} (Iterative reduction)
Assume $\mathcal{S}>s_0$,  $s_0\leq s \leq \mathcal{S}$. If there exist $\tau_0>0, N_0 \in \mathbb{N}$ large enough such that
\begin{equation}\label{sc-R0}
N_0^{\tau_0} \mathfrak{M}_{R_0}^{\sharp, \gamma^{\kappa}}(-1, s_0, b_0) \gamma^{-\kappa} \leq 1,
\end{equation}
then the following results hold:

(1)$_{n}$. For all $n \geq 0$, there exists a sequence of  operators
\begin{equation}\label{djn-def}
\mathcal{L}_{n} = \omega \cdot \partial_{\varphi} + D_n + R_n, \quad D_n= {\rm diag}_{j \in \mathbb{Z}}{\rm i} d_j^{(n)},\quad d_j^{(n)}=d_j^{(0)}+r_j^{(n)},
 \end{equation}
where $d_j^{(n)}, r_j^{(n)}$ are defined on $\mathcal{O}_0$ satisfying
\begin{equation*}
d_j^{(n)}=-d_{-j}^{(n)}\in \mathbb{R}, \quad r_j^{(n)}=-r_{-j}^{(n)}\in \mathbb{R}, \quad r_j^{(0)}=0, \quad \forall j \in \mathbb{Z}
\end{equation*}
and
\begin{equation}\label{rjn-fbds}
 \sup_{j} \langle j\rangle |r_j^{(n)}|^{\gamma^{\kappa}, \mathcal{O}_0}\leq Ck_2.
\end{equation}
The operators $R_n$ are defined on $\mathcal{O}_1 \cap \Omega_n^{\gamma^{\kappa}}$,
where $\Omega_{0}^{\gamma^{\kappa}}:=\mathcal{O}_0$,
\begin{equation}\label{Omega-n}
\begin{split}
\Omega_n^{\gamma^{\kappa}}(\mathfrak{J}):= \Big\{ \omega \in \Omega_{n-1}^{\gamma^{\kappa}}(\mathfrak{J}):
 & |\omega \cdot l + d_j^{(n-1)}- d_{j'}^{(n-1)}| \geq \frac{\gamma^{\kappa}|j-j'|}{\langle l \rangle^{\tau}},  \\
&\forall |l|\leq N_{n-1}, j,j' \in \mathbb{Z}, (j, j', l)\neq (j, j, 0)\Big\}, \quad n \geq 1 \\
 \end{split}
\end{equation}
 and $R_n$  and $\langle \partial_{\varphi}\rangle^{b_0}R_n$ are Lip-$-1$-modulo-tame satisfying
 \begin{equation}\label{Rn-fbd}
 \mathfrak{M}_{R_n}^{\sharp, \gamma^{\kappa}}(-1, s)\leq \mathfrak{M}_{R_0}^{\sharp, \gamma^{\kappa}}(-1, s, b_0)N_{n-1}^{-a_0},
 \quad \mathfrak{M}_{R_n}^{\sharp, \gamma^{\kappa}}(-1, s, b_0)\leq \mathfrak{M}_{R_0}^{\sharp, \gamma^{\kappa}}(-1, s,  b_0)N_{n-1}.
 \end{equation}
For  $n \geq 1$ and $j \in \mathbb{Z}$, there holds
\begin{equation}\label{djn-df-fb}
\langle j\rangle|r_j^{(n)}-r_j^{(n-1)}|^{\gamma^{\kappa}, \mathcal{O}_0} = \langle j\rangle|d_j^{(n)}-d_j^{(n-1)}|^{\gamma^{\kappa}, \mathcal{O}_0}\leq \mathfrak{M}_{R_0}^{\sharp, \gamma^{\kappa}}(-1, s_0,  b_0)N_{n-2}^{-a_0}.
 \end{equation}
Meanwhile, $R_n, \mathcal{L}_n$ are real and reversible operators.

(2)$_{n}$.
For $n \geq 1$, there exists a real linear bounded invertible T\"{o}plitz-in-time map $\Phi_{n-1}:= {\rm{Id}}+ \Psi_{n-1}$ defined on $\mathcal{O}_1\cap \Omega_n^{\gamma^{\kappa}}$ such that
\begin{equation*}
\mathcal{L}_n = \Phi_{n-1}^{-1} \mathcal{L}_{n-1}\Phi_{n-1}.
\end{equation*}
$\Psi_{n-1}$  and $\langle \partial_{\varphi}\rangle^{b_0}\Psi_{n-1}$ are Lip-$-1$-modulo-tame satisfying
\begin{equation*}
\mathfrak{M}_{\Psi_{n-1}}^{\sharp, \gamma^{\kappa}}(-1, s) \leq \gamma^{-\kappa} N_{n-1}^{2\tau+1} N_{n-2}^{-a_0} \mathfrak{M}_{R_0}^{\sharp, \gamma^{\kappa}}(-1, s,  b_0),
  \end{equation*}
\begin{equation*}
\mathfrak{M}_{\Psi_{n-1}}^{\sharp, \gamma^{\kappa}}(-1, s, b_0) \leq \gamma^{-\kappa} N_{n-1}^{2\tau+1} N_{n-2} \mathfrak{M}_{R_0}^{\sharp, \gamma^{\kappa}}(-1, s,  b_0).
  \end{equation*}
In addition, for any $\omega \in \Omega_n^{\gamma_1}(\mathfrak{J}_1)\cap \mathcal{O}_1(\mathfrak{J}_1) \cap \Omega_n^{\gamma_2}(\mathfrak{J}_2) \cap \mathcal{O}_1(\mathfrak{J}_2)$ with $\gamma_1, \gamma_2 \in [\gamma^{\kappa}/2, 2 \gamma^{\kappa}]$, it follows that
 \begin{equation*}
 \| \langle D_x \rangle^{1/2}\underline{\Delta_{12}\Psi_{n-1}} \langle D_x \rangle^{1/2}\|_{\mathfrak{L}(H^{s_0})} \leq  \gamma^{-\kappa} k_2 N_{n-1}^{2\tau} N_{n-2}^{-a_0} \|\mathfrak{J}_1- \mathfrak{J}_2\|_{s_0+\sigma},
   \end{equation*}
 \begin{equation*}
 \| \langle D_x \rangle^{1/2}\underline{\langle \partial_{\varphi}\rangle^{b_0}\Delta_{12}\Psi_{n-1}} \langle D_x \rangle^{1/2} \|_{\mathfrak{L}(H^{s_0})} \leq \gamma^{-\kappa} k_2 N_{n-1}^{2\tau} N_{n-2}  \|\mathfrak{J}_1- \mathfrak{J}_2\|_{s_0+\sigma}.
   \end{equation*}
Furthermore, $\Phi_{n-1}^{\pm 1}$ and $\Psi_{n-1}$ are real and reversibility-preserving.

(3)$_{n}$.
For $n \geq 0$, for any $\omega \in \Omega_n^{\gamma_1}(\mathfrak{J}_1)\cap \mathcal{O}_1(\mathfrak{J}_1) \cap \Omega_n^{\gamma_2}(\mathfrak{J}_2) \cap \mathcal{O}_1(\mathfrak{J}_2)$ with $\gamma_1, \gamma_2 \in [\gamma^{\kappa}/2, 2 \gamma^{\kappa}]$, we have
 \begin{eqnarray}\label{del-Rn-rjn-fbd}
 \begin{aligned}
 &\| \langle D_x \rangle^{1/2}\underline{\Delta_{12}R_n}\langle D_x \rangle^{1/2} \|_{\mathfrak{L}(H^{s_0})} \leq k_2 N_{n-1}^{-a_0} \|\mathfrak{J}_1- \mathfrak{J}_2\|_{s_0+\sigma},\\
& \| \langle D_x \rangle^{1/2}\underline{\langle \partial_{\varphi}\rangle^{b_0}\Delta_{12}R_n}\langle D_x \rangle^{1/2} \|_{\mathfrak{L}(H^{s_0})} \leq k_2 N_{n-1} \|\mathfrak{J}_1- \mathfrak{J}_2\|_{s_0+\sigma},\\
&\langle j\rangle|\Delta_{12}r_j^{(n)}|\leq Ck_2 \|\mathfrak{J}_1- \mathfrak{J}_2\|_{s_0+\sigma}, \quad \forall j \in \mathbb{Z}.
\end{aligned}
\end{eqnarray}

For all $n \geq 1$, $ j \in \mathbb{Z}$, there holds
\begin{equation*}
\langle j\rangle|\Delta_{12}r_j^{(n)}- \Delta_{12}r_j^{(n-1)}| \leq k_2 N_{n-2}^{-a_0} \|\mathfrak{J}_1- \mathfrak{J}_2\|_{s_0+\sigma}.
  \end{equation*}

(4)$_{n}$.
Let  $0 < \rho < \gamma^{\kappa}/2$. For  all $n \geq 0$, if
$k_2 N_{n-1}^{\tau} \|\mathfrak{J}_1- \mathfrak{J}_2\|_{s_0+\sigma} \leq \rho$, then
\begin{equation*}
\Omega_n^{\gamma^{\kappa}}(\mathfrak{J}_1) \subseteq \Omega_n^{\gamma^{\kappa}-\rho}(\mathfrak{J}_2).
\end{equation*}
\end{prop}

\subsection{Proof of Proposition \ref{Ite-red}}

We argue by induction on $n$.

Recalling $r_j^{(0)}=0$, $N_{-1}:=1$ and $\Omega_0^{\gamma^{\kappa}}=\mathcal{O}_0$, the estimates in (1)$_{0}$ and (3)$_{0}$ hold. Since $m_{\infty}, m_0, m_2$ are real,  it follows from \eqref{D0} that $d_j^{(0)}$ is real and odd in $j$. In  (3)$_{0}$, we  note that $\Omega_0^{\gamma_1}(\mathfrak{J}_1)=\Omega_0^{\gamma_2}(\mathfrak{J}_2)=\mathcal{O}_0.$
(4)$_{0}$ is trivial, because $\Omega_0^{\gamma^{\kappa}}(\mathfrak{J}_1) = \Omega_0^{\gamma^{\kappa}-\rho}(\mathfrak{J}_2)=\mathcal{O}_0$.

\subsubsection{The reducibility step}

Now assuming the thesis holds for $0 \leq k \leq n$, we prove that it holds also for $k=n+1$.
First we show how to define $\Psi_n, \Phi_n$ and  $\mathcal{L}_{n+1}$. By the assumptions of the induction, we have
\begin{equation}\label{Phin-Ln}
\begin{aligned}
\mathcal{L}_n \Phi_n=
\Phi_n(\omega \cdot \partial_{\varphi}+ D_n)&+\big((\omega \cdot \partial_{\varphi}\Psi_n) +[D_n, \Psi_n]\\
&+ \Pi_{N_n}R_n\big)+\Pi_{N_n}^{\bot}R_n+R_n\Psi_n,
\end{aligned}
\end{equation}
where $[D_n, \Psi_n]=D_n\Psi_n-\Psi_nD_n$ and $\Pi_{N_n}R_n, \Pi_{N_n}^{\bot}R_n$ are defined in Definition \ref{Lo-var}.

Solving the following so-called  homological equation
\begin{equation}\label{hom-eq}
(\omega \cdot \partial_{\varphi}\Psi_n) +[D_n, \Psi_n]+ \Pi_{N_n}R_n=[R_n], \quad [R_n]:={\rm diag }_{j \in \mathbb{Z}} (R_n)_j^j(0),
\end{equation}
we have

\begin{lemma}\label{Hom-eq}(Homological equation)
For all $\omega \in \Omega_{n+1}^{\gamma^{\kappa}}\cap \mathcal{O}_1$, there exists a unique solution $\Psi_n$ of the homological equation. $\Psi_n$ and $\langle \partial_{\varphi}\rangle^{b_0}\Psi_n$ are Lip-$-1$-modulo-tame operators and the modulo-tame constants satisfy
  \begin{equation}\label{psin-Rn}
  \mathfrak{M}_{\Psi_n}^{\sharp, \gamma^{\kappa}}(-1, s) \leq C \gamma^{-\kappa}N_n^{2\tau+1} \mathfrak{M}_{R_n}^{\sharp, \gamma^{\kappa}}(-1, s),
 \end{equation}
 \begin{equation}\label{b-psin-Rn}
 \mathfrak{M}_{\langle \partial_{\varphi}\rangle^{b_0} \Psi_n}^{\sharp, \gamma^{\kappa}}(-1, s) \leq C \gamma^{-\kappa}N_n^{2\tau+1} \mathfrak{M}_{R_n}^{\sharp, \gamma^{\kappa}}(-1, s, b_0).
 \end{equation}

For any $\omega \in \Omega_{n+1}^{\gamma_1}(\mathfrak{J}_1)\cap \mathcal{O}_1(\mathfrak{J}_1) \cap \Omega_{n+1}^{\gamma_2}(\mathfrak{J}_2)\cap \mathcal{O}_1(\mathfrak{J}_2)$ with $\gamma_1, \gamma_2 \in [\gamma^{\kappa}/2, 2 \gamma^{\kappa}]$, there hold
  \begin{equation}\label{del-psin-Rn}
   \begin{split}
 \| \langle D_x \rangle^{1/2}&\underline{\Delta_{12}\Psi_{n}}\langle D_x \rangle^{1/2} \|_{\mathfrak{L}(H^{s_0})}\\
\leq & \gamma^{-\kappa}N_{n}^{2\tau}\big(\| \langle D_x \rangle^{1/2}\underline{\Delta_{12}R_{n}} \langle D_x \rangle^{1/2}\|_{\mathfrak{L}(H^{s_0})}\\
&+\| \langle D_x \rangle^{1/2}\underline{R_n} \langle D_x \rangle^{1/2}\|_{\mathfrak{L}(H^{s_0})}\|\mathfrak{J}_1- \mathfrak{J}_2\|_{s_0+\sigma}\big),\\
 \end{split}
 \end{equation}
  \begin{equation}\label{b-del-psin-Rn}
     \begin{split}
 \|\langle D_x \rangle^{1/2}&\underline{\langle \partial_{\varphi}\rangle^{b_0}\Delta_{12}\Psi_{n}}\langle D_x \rangle^{1/2}\|_{\mathfrak{L}(H^{s_0})}\\
\leq  & \gamma^{-\kappa} N_{n}^{2\tau}    \big(
  \| \langle D_x \rangle^{1/2}\underline{\langle \partial_{\varphi}\rangle^{b_0}\Delta_{12}R_{n}}\langle D_x \rangle^{1/2} \|_{\mathfrak{L}(H^{s_0})} \\
&+
\| \langle D_x \rangle^{1/2}\underline{\langle \partial_{\varphi}\rangle^{b_0}R_n} \langle D_x \rangle^{1/2}\|_{\mathfrak{L}(H^{s_0})}\|\mathfrak{J}_1- \mathfrak{J}_2\|_{s_0+\sigma}
\big).\\
 \end{split}
 \end{equation}

 Moreover, $\Phi_{n}, \Psi_{n}, \Phi_{n}^{-1}$ are real and reversibility-preserving.
\end{lemma}

\begin{proof}
 For $\omega \in \Omega_{n+1}^{\gamma^{\kappa}}\cap \mathcal{O}_1$, \eqref{hom-eq} is tantamount to the following equation:
\begin{equation*}
{\rm i} \omega \cdot l (\Psi_n)_j^{j'}(l)+{\rm i}d_j^{(n)}(\Psi_n)_j^{j'}(l) - {\rm i}d_{j'}^{(n)}(\Psi_n)_j^{j'}(l)=([R_n])_j^{j'}(l)-(\Pi_{N_n}R_n)_j^{j'}(l)
\end{equation*}
  whose unique solution is
 \begin{equation}\label{psin-def}
 (\Psi_n)_j^{j'}(l):=
 \left \{
\begin{array}{ll}
\frac{-(R_n)_j^{j'}(l)}{{\rm i}(\omega \cdot l+d_j^{(n)}-d_{j'}^{(n)})}, &  {\rm if} \ |l| \leq N_n \ {\rm and} \ (j, j', l)\neq (j, j, 0),\\
0, & {\rm otherwise}.\\
\end{array}
\right.
 \end{equation}
 Note that for all $\omega \in \Omega_{n+1}^{\gamma^{\kappa}}$, the divisors are nontrivial, so the above formula is well defined.

 In the following, we only prove the case of $|l| \leq N_n \ {\rm and} \ (j, j', l)\neq (j, j, 0)$,  the result holds for the other cases obviously.

 It follows from \eqref{psin-def} that
\begin{equation*}
|(\Psi_n)_j^{j'}(l)|  \leq  \gamma^{-\kappa}N_n^{\tau} |(R_n)_j^{j'}(l)|.
\end{equation*}
For any $ \omega, \omega' \in \Omega_{n+1}^{\gamma^{\kappa}}\cap \mathcal{O}_1$,
\begin{equation*}
\begin{split}
(\Psi_n)_j^{j'}(l)(\omega)- &(\Psi_n)_j^{j'}(l)(\omega')
=-\frac{(R_n)_j^{j'}(l)(\omega)- (R_n)_j^{j'}(l)(\omega')}{{\rm i}(\omega \cdot l+d_j^{(n)}(\omega)-d_{j'}^{(n)}(\omega))}\\
&+{\rm i}(R_n)_j^{j'}(l)(\omega')\frac{(\omega'-\omega)\cdot l- (d_j^{(n)}-d_{j'}^{(n)})(\omega)+(d_j^{(n)}-d_{j'}^{(n)})(\omega')}{(\omega \cdot l+d_j^{(n)}(\omega)-d_{j'}^{(n)}(\omega))(\omega' \cdot l+d_j^{(n)}(\omega')-d_{j'}^{(n)}(\omega'))}.\\
 \end{split}
\end{equation*}
 Using $|\cdot|^{lip, \mathcal{O}}\leq \gamma^{-1}|\cdot|^{\gamma, \mathcal{O}}$, \eqref{m-bd} and \eqref{rjn-fbds}, we obtain
\begin{equation*}
\begin{split}
|(\omega'-\omega)\cdot l&- (d_j^{(n)}-d_{j'}^{(n)})(\omega)+(d_j^{(n)}-d_{j'}^{(n)})(\omega')|\\
 \leq & |\omega-\omega'||l|+|m_{\infty}(\omega)-m_{\infty}(\omega')||j-j'|+|r_j^{(n)}(\omega)-r_j^{(n)}(\omega')|\\
 &\qquad +|r_{j'}^{(n)}(\omega)-r_{j'}^{(n)}(\omega')|\\
 \leq &|\omega-\omega'|(|l|+k_1\gamma^{-1}|j-j'|+k_2\gamma^{-\kappa}).\\
 \end{split}
\end{equation*}
From \eqref{k1-k2}, we derive that
 \begin{equation*}
\begin{split}
|\Delta_{\omega, \omega'}(\Psi_n)_j^{j'}(l)|
 \leq &  \gamma^{-\kappa}N_n^{\tau}|\Delta_{\omega, \omega'}(R_n)_j^{j'}(l)|
   +(|l|+|j-j'|)\frac{\langle l\rangle^{2\tau}}{\gamma^{2\kappa}|j-j'|^2} |(R_n)_j^{j'}(l)(\omega')|\\
   \leq &  \gamma^{-\kappa}N_n^{\tau} |\Delta_{\omega, \omega'}(R_n)_j^{j'}(l)|
   + \gamma^{-2\kappa}N_n^{2\tau+1} |(R_n)_j^{j'}(l)(\omega')|.\\
    \end{split}
\end{equation*}
Then we obtain \eqref{psin-Rn} after using Lemma \ref{mod-compa}.

Similarly, it follows that
  \begin{equation*}
\begin{split}
|(\langle \partial_{\varphi} \rangle^{b_0}\Psi_n)_j^{j'}(l)|&= |\langle l \rangle^{b_0}(\Psi_n)_j^{j'}(l)| \leq
 \frac{N_n^{\tau}\gamma^{-\kappa}}{|j-j'|} |\langle l \rangle^{b_0}(R_n)_j^{j'}(l)|\\
 &\leq N_n^{\tau}\gamma^{-\kappa} |(\langle \partial_{\varphi} \rangle^{b_0}R_n)_j^{j'}(l)|\\
  \end{split}
\end{equation*}
and
 \begin{equation*}
\begin{split}
\langle l \rangle^{b_0}|\Delta_{\omega, \omega'}(\Psi_n)_j^{j'}(l)|
   \leq  \gamma^{-\kappa}N_n^{\tau} \langle l \rangle^{b_0} |\Delta_{\omega, \omega'}(R_n)_j^{j'}(l)|
   + \gamma^{-2\kappa}N_n^{2\tau+1} \langle l \rangle^{b_0}|(R_n)_j^{j'}(l)(\omega')|.
      \end{split}
\end{equation*}
   Thus, we derive \eqref{b-psin-Rn} from Lemma \ref{mod-compa}.

 In addition, for any $\omega \in \Omega_{n+1}^{\gamma_1}(\mathfrak{J}_1)\cap \mathcal{O}_1(\mathfrak{J}_1) \cap \Omega_{n+1}^{\gamma_2}(\mathfrak{J}_2)\cap \mathcal{O}_1(\mathfrak{J}_2)$ with $\gamma_1, \gamma_2 \in [\gamma^{\kappa}/2, 2 \gamma^{\kappa}]$, we have
\begin{equation*}
\begin{split}
&\Delta_{12}(\Psi_n)_j^{j'}(l) \\
&= -\frac{\Delta_{12}(R_n)_j^{j'}(l)}{{\rm i}(\omega \cdot l+d_j^{(n)}(\mathfrak{J}_1)-d_{j'}^{(n)}(\mathfrak{J}_1))}\\
 &\quad +{\rm i}(R_n)_j^{j'}(l)(\mathfrak{J}_2)\frac{-\left((d_j^{(n)}-d_{j'}^{(n)})(\mathfrak{J}_1)\right)+ \left((d_j^{(n)}-d_{j'}^{(n)})(\mathfrak{J}_2)\right)}
 {(\omega \cdot l+d_j^{(n)}(\mathfrak{J}_1)-d_{j'}^{(n)}(\mathfrak{J}_1))(\omega \cdot l+d_j^{(n)}(\mathfrak{J}_2)-d_{j'}^{(n)}(\mathfrak{J}_2))}.\\
 \end{split}
\end{equation*}

In view of \eqref{k1-k2}, \eqref{del-m-bd} and \eqref{del-Rn-rjn-fbd}, we obtain
\begin{equation}\label{del-djj'n}
\begin{split}
|\Delta_{12}(d_j^{(n)}-d_{j'}^{(n)})|
\leq & |\Delta_{12}(d_j^{(0)}-d_{j'}^{(0)})|
+|\Delta_{12}r_j^{(n)}|
+ |\Delta_{12}r_{j'}^{(n)}|\\
\leq &|\Delta_{12}m_{\infty}||j-j'|+|\Delta_{12}r_j^{(n)}|+|\Delta_{12}r_{j'}^{(n)}|\\
 \leq &(k_1|j-j'|+k_2)  \|\mathfrak{J}_1- \mathfrak{J}_2\|_{s_0+\sigma}
\leq k_2|j-j'|  \|\mathfrak{J}_1- \mathfrak{J}_2\|_{s_0+\sigma}.\\
 \end{split}
\end{equation}
Therefore
\begin{equation*}
|\Delta_{12}(\Psi_n)_j^{j'}(l)|\leq \gamma^{-\kappa}N_n^{\tau}|\Delta_{12}(R_n)_j^{j'}(l)|+ k_2|(R_n)_j^{j'}(l)|\gamma^{-2\kappa}N_n^{2\tau}\|\mathfrak{J}_1- \mathfrak{J}_2\|_{s_0+\sigma}
\end{equation*}
by noticing that $\gamma_1, \gamma_2 \in [\gamma^{\kappa}/2, 2\gamma^{\kappa}]$.
As a consequence of Lemma \ref{mod-compa} and \eqref{k1-k2}, \eqref{del-psin-Rn} is satisfied.

In the same manner, we obtain
\begin{equation*}
\begin{split}
|\langle l \rangle^{b_0}\Delta_{12}(\Psi_n)_j^{j'}(l)| \leq& \gamma^{-\kappa}N_n^{\tau}|\langle l \rangle^{b_0}\Delta_{12}(R_n)_j^{j'}(l)|\\
&+ k_2|\langle l \rangle^{b_0}(R_n)_j^{j'}(l)|\gamma^{-2\kappa}N_n^{2\tau}\|\mathfrak{J}_1- \mathfrak{J}_2\|_{s_0+\sigma}.\\
\end{split}
\end{equation*}
Thus   \eqref{b-del-psin-Rn} holds.

Finally, since $d_j^{(n)}, d_{j'}^{(n)}$ are real, $d_j^{(n)}=-d_{-j}^{(n)}$, $d_{j'}^{(n)}=-d_{-j'}^{(n)}$ and $R_n$ is real, we have
\begin{equation*}
(\Psi_n)_{-j}^{-j'}(-l)=\frac{-(R_n)_{-j}^{-j'}(-l)}{{\rm i}(\omega \cdot (-l)+d_{-j}^{(n)}-d_{-j'}^{(n)})}
=\overline{\frac{-(R_n)_{j}^{j'}(l)}{{\rm i}(\omega \cdot l+d_{j}^{(n)}-d_{j'}^{(n)})}}
=\overline{(\Psi_n)_{j}^{j'}(l)},
\end{equation*}
which implies that $\Psi_n$ is real.

Since $R_n$ is reversible,
\begin{equation*}
(\Psi_n)_{-j}^{-j'}(-l)=\frac{-(R_n)_{-j}^{-j'}(-l)}{{\rm i}(\omega \cdot (-l)+d_{-j}^{(n)}-d_{-j'}^{(n)})}
=\frac{(R_n)_{j}^{j'}(l)}{-{\rm i}(\omega \cdot l+d_{j}^{(n)}-d_{j'}^{(n)})}
=(\Psi_n)_{j}^{j'}(l),
\end{equation*}
 $\Psi_n$ is reversibility-preserving. Whence $\Phi_n, \Phi_n^{-1}$ are real and reversibility-preserving as well.
 \end{proof}

\subsubsection{The iteration}

$(2)_{n+1}$.
From \eqref{psin-Rn} and \eqref{Rn-fbd}, we derive that
 \begin{eqnarray}\label{psin-fbd}
 \begin{aligned}
 &\mathfrak{M}_{\Psi_{n}}^{\sharp, \gamma^{\kappa}}(-1, s) \leq \gamma^{-\kappa} N_{n}^{2\tau+1} N_{n-1}^{-a_0} \mathfrak{M}_{R_0}^{\sharp, \gamma^{\kappa}}(-1, s,  b_0),\\
 &\mathfrak{M}_{\Psi_{n}}^{\sharp, \gamma^{\kappa}}(-1, s, b_0) \leq \gamma^{-\kappa} N_{n}^{2\tau+1} N_{n-1} \mathfrak{M}_{R_0}^{\sharp, \gamma^{\kappa}}(-1, s,  b_0).
 \end{aligned}
 \end{eqnarray}

By Lemma \ref{L(Hs0)-mod}, \eqref{R0-bds}, \eqref{Rn-fbd}, \eqref{del-Rn-rjn-fbd} and \eqref{del-psin-Rn}-\eqref{b-del-psin-Rn}, we deduce that
\begin{eqnarray} \label{del-psin-fbd}
  \begin{aligned}
&\| \langle D_x \rangle^{1/2}\underline{\Delta_{12}\Psi_{n}}  \langle D_x \rangle^{1/2}\|_{\mathfrak{L}(H^{s_0})} \leq  \gamma^{-\kappa} k_2 N_{n}^{2\tau} N_{n-1}^{-a_0} \|\mathfrak{J}_1- \mathfrak{J}_2\|_{s_0+\sigma},\\
&\| \langle D_x \rangle^{1/2} \underline{\langle \partial_{\varphi}\rangle^{b_0}\Delta_{12}\Psi_{n}}  \langle D_x \rangle^{1/2}\|_{\mathfrak{L}(H^{s_0})} \leq \gamma^{-\kappa} k_2 N_{n}^{2\tau} N_{n-1}  \|\mathfrak{J}_1- \mathfrak{J}_2\|_{s_0+\sigma}.
\end{aligned}
\end{eqnarray}
 Due to \eqref{a0b0-def}, \eqref{sc-R0} and \eqref{psin-fbd}, it follows that
\begin{equation}\label{psins0-bd}
4C(\mathcal{S}, b_0)\mathfrak{M}_{\Psi_{n}}^{\sharp, \gamma^{\kappa}}(-1, s_0) \leq 1/2.
\end{equation}
  Appling Lemma \ref{mod-inv} to $\Psi_{n}$, the map $\Phi_n=1+\Psi_n$ is invertible and
  \begin{equation}\label{Phin-inv}
  \mathfrak{M}_{\Phi_n^{-1}}^{\sharp, \gamma^{\kappa}}(-1, s_0)\leq 2, \quad \mathfrak{M}_{\Phi_n^{-1}}^{\sharp, \gamma^{\kappa}}(-1, s)\leq 1+2\mathfrak{M}_{\Psi_n}^{\sharp, \gamma^{\kappa}}(-1, s).
  \end{equation}
In the following, we denote $\breve{\Psi}_n=\Phi_n^{-1}-{\rm Id}$.
According to Lemma \ref{Hom-eq}, $\Psi_n$ is real and reversibility-preserving.
 Therefore, $\Phi_n^{\pm 1}$ are real and reversibility-preserving as well.

$(1)_{n+1}$.
 Recalling \eqref{Phin-Ln}, for $\omega \in \Omega_{n+1}^{\gamma^{\kappa}}\cap \mathcal{O}_1$, we define that
 \begin{equation*}
 \mathcal{L}_{n+1} :=\Phi_{n}^{-1} \mathcal{L}_{n}\Phi_{n}=\omega \cdot \partial_{\varphi}+D_{n+1}+R_{n+1},
 \end{equation*}
where
 \begin{equation}\label{Rn+1-def}
 D_{n+1}:=D_{n}+[R_n], \quad R_{n+1}:=\Phi_n^{-1}(\Pi_{N_n}^{\bot}R_n+R_n\Psi_n-\Psi_n[R_n]).
  \end{equation}

 The fact that $R_n$ is real and reversible implies that
 \begin{equation*}
 \overline{(R_n)_j^j(0)}=(R_n)_{-j}^{-j}(0)=-(R_n)_j^j(0),
 \end{equation*}
  from which we derive  that $(R_n)_j^j(0)$ is purely imaginary and odd in $j$.
  By the Kirszbraun theorem, $(R_n)_j^j(0)$  can be extended to the whole $\mathcal{O}_0$  preserving Lipschitz weighted norm $|\cdot|^{\gamma^{\kappa}, \Omega_{n+1}^{\gamma^{\kappa}}\cap \mathcal{O}_1}$. The extended function will be denoted by $(\tilde{R}_n)_j^j(0)$. $(\tilde{R}_n)_j^j(0)$ is purely imaginary and odd in $j$ as well. We define
\begin{equation}\label{rjn+1-def}
\begin{split}
D_{n+1}:= &\;{\rm diag}_{j \in \mathbb{Z}} {\rm i} d_j^{(n+1)},\\
  d_j^{(n+1)}:=&\; d_j^{(n)}+\frac{1}{{\rm i}}(\tilde{R}_n)_j^j(0)=d_j^{(0)}+r_j^{(n)}+\frac{1}{{\rm i}}(\tilde{R}_n)_j^j(0)=d_j^{(0)}+r_j^{(n+1)},\\
 r_j^{(n+1)}:=&\;r_j^{(n)}+\frac{1}{{\rm i}}(\tilde{R}_n)_j^j(0).\\
 \end{split}
 \end{equation}
 By inductive hypothesis, $r_j^{(n+1)}, d_j^{(n+1)}$ are real and odd in $j$.

In view of Lemma \ref{[A]-pro}, \eqref{Rn-fbd} and \eqref{rjn+1-def}, it follows that
 \begin{equation*}
  \begin{split}
 \langle j\rangle|r_j^{(n+1)}-r_j^{(n)}|^{\gamma^{\kappa}, \mathcal{O}_0} &=\langle j\rangle|d_j^{(n+1)}-d_j^{(n)}|^{\gamma^{\kappa}, \mathcal{O}_0}
= \langle j\rangle|(\tilde{R}_n)_j^j(0)|^{\gamma^{\kappa},  \mathcal{O}_0}\\
&= \langle j\rangle|(R_n)_j^j(0)|^{\gamma^{\kappa}, \Omega_{n+1}^{\gamma^{\kappa}}\cap \mathcal{O}_1}
\leq \mathfrak{M}_{R_n}^{\sharp, \gamma^{\kappa}}(-1, s_0)\\
&\leq \mathfrak{M}_{R_0}^{\sharp, \gamma^{\kappa}}(-1, s_0,  b_0)N_{n-1}^{-a_0}, \quad \forall j \in \mathbb{Z}.\\
 \end{split}
   \end{equation*}
Note that $r_j^{(0)}=0$. Summing the telescopic series,  \eqref{rjn-fbds} follows by \eqref{R0-bds} for $r_j^{(n+1)}$ instead of $r_j^{(n)}$.
 Moreover, $\mathcal{L}_{n+1}$ is real and reversible since $\mathcal{L}_{n}$ is real and reversible and $\Phi_n^{\pm 1}$ are reversibility-preserving.
 $R_{n+1}$ is real  because all the components in \eqref{Rn+1-def} are real.
Also it is  reversible because $R_n, [R_n], \Pi_{N_n}^{\bot}R_n$ are reversible and $\Phi_n^{-1}, \Psi_n$ are reversibility-preserving.

Let us establish the estimates \eqref{Rn-fbd} for $R_{n+1}$.
Owing to \eqref{psin-Rn},  \eqref{psins0-bd}-\eqref{Rn+1-def}, Lemma \ref{mod-sum-com} and  Lemmas \ref{mod-smoo}-\ref{[A]-pro}, we have
\begin{equation*}
 \begin{aligned}
\mathfrak{M}_{R_{n+1}}^{\sharp, \gamma^{\kappa}}(-1, s)
\leq \;& \mathfrak{M}_{\Pi_{N_n}^{\bot}R_n}^{\sharp, \gamma^{\kappa}}(-1, s)+ \mathfrak{M}_{R_n}^{\sharp, \gamma^{\kappa}}(-1, s)  \mathfrak{M}_{\Psi_n}^{\sharp, \gamma^{\kappa}}(-1, s_0)\\
& + \mathfrak{M}_{R_n}^{\sharp, \gamma^{\kappa}}(-1, s_0)  \mathfrak{M}_{\Psi_n}^{\sharp, \gamma^{\kappa}}(-1, s)
+\left(1+\mathfrak{M}_{\Psi_n}^{\sharp, \gamma^{\kappa}}(-1, s)\right)\\
&\times \left(\mathfrak{M}_{\Pi_{N_n}^{\bot}R_n}^{\sharp, \gamma^{\kappa}}(-1, s_0)+ \mathfrak{M}_{R_n}^{\sharp, \gamma^{\kappa}}(-1, s_0)  \mathfrak{M}_{\Psi_n}^{\sharp, \gamma^{\kappa}}(-1, s_0)\right)\\
\leq \;&\mathfrak{M}_{\Pi_{N_n}^{\bot}R_n}^{\sharp, \gamma^{\kappa}}(-1, s)+ \mathfrak{M}_{R_n}^{\sharp, \gamma^{\kappa}}(-1, s)  \mathfrak{M}_{\Psi_n}^{\sharp, \gamma^{\kappa}}(-1, s_0)\\
&+ \mathfrak{M}_{R_n}^{\sharp, \gamma^{\kappa}}(-1, s_0)  \mathfrak{M}_{\Psi_n}^{\sharp, \gamma^{\kappa}}(-1, s) \\
\leq \; &N_{n}^{-b_0}\mathfrak{M}_{R_{n}}^{\sharp, \gamma^{\kappa}}(-1, s, b_0)+ N_n^{2\tau+1}\gamma^{-\kappa} \mathfrak{M}_{R_n}^{\sharp, \gamma^{\kappa}}(-1, s)  \mathfrak{M}_{R_n}^{\sharp, \gamma^{\kappa}}(-1, s_0).\\
 \end{aligned}
    \end{equation*}
Note that the iterative terms  are quadratic  plus a super-exponentially small term.
By  \eqref{a0b0-def}, \eqref{sc-R0} and \eqref{Rn-fbd}, one gets
 \begin{equation*}
 \mathfrak{M}_{R_{n+1}}^{\sharp, \gamma^{\kappa}}(-1, s)
\leq \mathfrak{M}_{R_0}^{\sharp, \gamma^{\kappa}}(-1, s, b_0)N_{n}^{-a_0}.
    \end{equation*}

In order to obtain the estimates for $\mathfrak{M}_{\langle \partial_{\varphi} \rangle^{b_0}R_{n+1}}^{\sharp, \gamma^{\kappa}}(-1, s)$, we make some preparations at first.
 Exploiting \eqref{sc-R0}, \eqref{Rn-fbd} and \eqref{psin-Rn}-\eqref{b-psin-Rn}, we obtain
 \begin{equation}\label{fmls}
 \begin{split}
& \mathfrak{M}_{R_{n}}^{\sharp, \gamma^{\kappa}}(-1, s_0) \mathfrak{M}_{\Psi_{n}}^{\sharp, \gamma^{\kappa}}(-1, s)
\leq \mathfrak{M}_{R_{n}}^{\sharp, \gamma^{\kappa}}(-1, s),\\
&\mathfrak{M}_{R_{n}}^{\sharp, \gamma^{\kappa}}(-1, s_0) \mathfrak{M}_{\Psi_{n}}^{\sharp, \gamma^{\kappa}}(-1, s, b_0)
\leq \mathfrak{M}_{R_{n}}^{\sharp, \gamma^{\kappa}}(-1, s, b_0),\\
&\mathfrak{M}_{R_{n}}^{\sharp, \gamma^{\kappa}}(-1, s_0, b_0) \mathfrak{M}_{\Psi_{n}}^{\sharp, \gamma^{\kappa}}(-1, s)
 \leq N_n^{2\tau+1}\gamma^{-\kappa}\mathfrak{M}_{R_{n}}^{\sharp, \gamma^{\kappa}}(-1, s_0, b_0) \mathfrak{M}_{R_{n}}^{\sharp, \gamma^{\kappa}}(-1, s),\\
  &\mathfrak{M}_{R_{n}}^{\sharp, \gamma^{\kappa}}(-1, s) \mathfrak{M}_{\Psi_{n}}^{\sharp, \gamma^{\kappa}}(-1, s_0, b_0)
\leq N_n^{2\tau+1}\gamma^{-\kappa}\mathfrak{M}_{R_{n}}^{\sharp, \gamma^{\kappa}}(-1, s_0, b_0) \mathfrak{M}_{R_{n}}^{\sharp, \gamma^{\kappa}}(-1, s).\\
 \end{split}
    \end{equation}
Notice that $\langle \partial_{\varphi} \rangle^{b_0}\Phi_n^{-1}=1+ \langle \partial_{\varphi} \rangle^{b_0}\check{\Psi}_n,$
$\langle \partial_{\varphi} \rangle^{b_0} \Pi_{N_n}^{\bot} R_n=\Pi_{N_n}^{\bot}\langle \partial_{\varphi} \rangle^{b_0}R_n$.
Using
Lemmas \ref{mod-sum-com}-\ref{mod-smoo}, \eqref{sc-R0},
\eqref{Rn-fbd}, \eqref{psins0-bd}-\eqref{Phin-inv} and \eqref{fmls},
we have
\begin{equation*}
 \begin{split}
& \mathfrak{M}_{\langle \partial_{\varphi} \rangle^{b_0}R_{n+1}}^{\sharp, \gamma^{\kappa}}(-1, s)\\
&\leq \mathfrak{M}_{\langle \partial_{\varphi} \rangle^{b_0}\Phi_{n}^{-1}}^{\sharp, \gamma^{\kappa}}(-1, s) \mathfrak{M}_{R_{n}}^{\sharp, \gamma^{\kappa}}(-1, s_0)+\mathfrak{M}_{\langle \partial_{\varphi} \rangle^{b_0}\Phi_{n}^{-1}}^{\sharp, \gamma^{\kappa}}(-1, s_0) \mathfrak{M}_{R_{n}}^{\sharp, \gamma^{\kappa}}(-1, s)
\\
& +\mathfrak{M}_{\Phi_{n}^{-1}}^{\sharp, \gamma^{\kappa}}(-1, s_0) \left(\mathfrak{M}_{R_{n}}^{\sharp, \gamma^{\kappa}}(-1, s, b_0)
 +N_n^{2\tau+1}\gamma^{-\kappa}\mathfrak{M}_{R_{n}}^{\sharp, \gamma^{\kappa}}(-1, s_0, b_0) \mathfrak{M}_{R_{n}}^{\sharp, \gamma^{\kappa}}(-1, s)\right)\\
&+\mathfrak{M}_{\Phi_{n}^{-1}}^{\sharp, \gamma^{\kappa}}(-1, s) \mathfrak{M}_{R_{n}}^{\sharp, \gamma^{\kappa}}(-1, s_0, b_0)\\
&\leq \mathfrak{M}_{R_{0}}^{\sharp, \gamma^{\kappa}}(-1, s, b_0)N_{n-1}+\mathfrak{M}_{R_{0}}^{\sharp, \gamma^{\kappa}}(-1, s, b_0)
\leq \mathfrak{M}_{R_{0}}^{\sharp, \gamma^{\kappa}}(-1, s, b_0)N_{n}.\\
 \end{split}
    \end{equation*}

$(3)_{n+1}$. For any $\omega \in \Omega_{n+1}^{\gamma_1}(\mathfrak{J}_1)\cap \mathcal{O}_1(\mathfrak{J}_1) \cap \Omega_{n+1}^{\gamma_2}(\mathfrak{J}_2) \cap \mathcal{O}_1(\mathfrak{J}_2)$ with $\gamma_1, \gamma_2 \in [\gamma^{\kappa}/2, 2 \gamma^{\kappa}]$,
in view of Lemma \ref{[A]-pro}, \eqref{del-Rn-rjn-fbd} and \eqref{rjn+1-def}, it follows that
\begin{equation*}
\begin{split}
  \langle j\rangle |\Delta_{12}r_j^{(n+1)}- \Delta_{12}r_j^{(n)}|
& =\langle j\rangle|\Delta_{12}(R_{n})_j^j(0)|
\leq \| \langle D_x\rangle^{1/2}\underline{\Delta_{12}R_n }\langle D_x\rangle^{1/2}\|_{\mathfrak{L}(H^{s_0})}\\
& \leq  k_2N_{n-1}^{-a_0}\|\mathfrak{J}_1-\mathfrak{J}_2 \|_{s_0+\sigma}, \quad \forall j \in \mathbb{Z}.\\
 \end{split}
 \end{equation*}
Recall that $r_j^{(0)}=0$.  Summing all the terms above, one obtains
 \begin{equation*}
 \langle j\rangle|\Delta_{12}r_j^{(n+1)}|\leq \sum_{k=0}^{n} \langle j\rangle|\Delta_{12}r_j^{(k+1)}- \Delta_{12}r_j^{(k)}|
\leq k_2 \|\mathfrak{J}_1-\mathfrak{J}_2 \|_{s_0+\sigma}, \quad \forall j \in \mathbb{Z}.
 \end{equation*}

By Lemma \ref{L(Hs0)-mod} and \eqref{psin-fbd}, provided that $N_0$ is large enough, the smallness condition in Lemma \ref{L(Hs0)-inv} is satisfied. Therefore by
 Lemmas \ref{L(Hs0)-mod}, \ref{mod-inv}, \ref{L(Hs0)-inv}, \eqref{sc-R0} and \eqref{psin-fbd}-\eqref{psins0-bd}, we have
 \begin{eqnarray}\label{del-br-psin}
 \begin{aligned}
 &\| \langle D_x\rangle^{1/2}\underline{\Delta_{12}\breve{\Psi}_n} \langle D_x\rangle^{1/2}\|_{\mathcal{L}(H^{s_0})} \leq N_n^{2\tau}N_{n-1}^{-a_0} k_2 \gamma^{-\kappa} \|\mathfrak{J}_1- \mathfrak{J}_2\|_{s_0+\sigma},\\
 &\|\langle D_x\rangle^{1/2} \langle \partial_{\varphi} \rangle^{b_0} \underline{\Delta_{12}\breve{\Psi}_n}\langle D_x\rangle^{1/2} \|_{\mathcal{L}(H^{s_0})} \leq N_n^{2\tau}N_{n-1} k_2 \gamma^{-\kappa} \|\mathfrak{J}_1- \mathfrak{J}_2\|_{s_0+\sigma}.
 \end{aligned}
 \end{eqnarray}
From \eqref{Rn+1-def}, we derive that
 \begin{equation*}
 \begin{split}
 \Delta_{12}R_{n+1}=&\Delta_{12}\breve{\Psi}_{n}(\Pi_{N_n}^{\bot}R_n+R_n\Psi_n-\Psi_n[R_n])+{\Phi}_{n}^{-1}
\cdot \bigg(\Pi_{N_n}^{\bot}\Delta_{12}R_n\\
 &+\Delta_{12}R_n\Psi_n+R_n \Delta_{12}\Psi_n-\Delta_{12}\Psi_n[R_n]- \Psi_n\Delta_{12}[R_n]\bigg).\\
 \end{split}
    \end{equation*}
By Lemmas \ref{L(Hs0)-mod}, \ref{mod-inv}-\ref{L(H_s0)},  \eqref{a0b0-def}, \eqref{sc-R0},  \eqref{Rn-fbd}, \eqref{del-Rn-rjn-fbd}, \eqref{psin-fbd}-\eqref{del-psin-fbd}, \eqref{Phin-inv} and \eqref{del-br-psin}, one gets
 \begin{equation*}
 \begin{split}
\| \langle D_x\rangle^{1/2}& \underline{\Delta_{12}R_{n+1}} \langle D_x\rangle^{1/2}\|_{\mathfrak{L}(H^{s_0})}\\
\leq & \| \langle D_x\rangle^{1/2}\underline{\Delta_{12}\breve{\Psi}_{n}} \langle D_x\rangle^{1/2}\|_{\mathfrak{L}(H^{s_0})}
\cdot \Big( N_n^{-b_0}\mathfrak{M}_{R_n}^{\sharp, \gamma^{\kappa}}(-1, s_0, b_0)\\
&+\mathfrak{M}_{R_n}^{\sharp, \gamma^{\kappa}}(-1,s_0) \mathfrak{M}_{\Psi_n}^{\sharp, \gamma^{\kappa}}(-1, s_0)\Big)\\
&+\mathfrak{M}_{\Phi_n^{-1}}^{\sharp, \gamma^{\kappa}}(-1,s_0)
\cdot \Big( N_n^{-b_0}\|\langle D_x\rangle^{1/2} \underline{\langle \partial_{\varphi} \rangle^{b_0}\Delta_{12}R_n}\langle D_x\rangle^{1/2} \|_{\mathcal{L}(H^{s_0})}\\
&+\|\langle D_x\rangle^{1/2} \underline{\Delta_{12}R_n}\langle D_x\rangle^{1/2} \|_{\mathcal{L}(H^{s_0})}
\cdot \mathfrak{M}_{\Psi_n}^{\sharp, \gamma^{\kappa}}(-1, s_0)\\
&+\|\langle D_x\rangle^{1/2} \underline{\Delta_{12}\Psi_n}\langle D_x\rangle^{1/2} \|_{\mathcal{L}(H^{s_0})}
\cdot \mathfrak{M}_{R_n}^{\sharp, \gamma^{\kappa}}(-1, s_0)\Big)\\
\leq &\big(k_2N_{n-1}N_n^{-b_0}+k_2\gamma^{-\kappa}N_n^{2\tau+1}\mathfrak{M}_{R_0}^{\sharp, \gamma^{\kappa}}(-1, s_0, b_0)N_{n-1}^{-2a_0}\big)\|\mathfrak{J}_1- \mathfrak{J}_2\|_{s_0+\sigma}\\
\leq & k_2N_n^{-a_0}\|\mathfrak{J}_1- \mathfrak{J}_2\|_{s_0+\sigma}\\
 \end{split}
 \end{equation*}
and
 \begin{equation*}
 \begin{split}
\|\langle  D_x \rangle^{1/2}&  \underline{\langle \partial_{\varphi} \rangle^{b_0}\Delta_{12}R_{n+1}}\langle D_x\rangle^{1/2}\|_{\mathfrak{L}(H^{s_0})}\\
\leq &\|\langle D_x\rangle^{1/2} \langle \partial_{\varphi} \rangle^{b_0} \underline{\Delta_{12}\breve{\Psi}_n}\langle D_x\rangle^{1/2} \|_{\mathcal{L}(H^{s_0})}
\cdot \Big(N_n^{-b_0}\mathfrak{M}_{R_n}^{\sharp, \gamma^{\kappa}}(-1, s_0, b_0)\\
&+\mathfrak{M}_{R_n}^{\sharp, \gamma^{\kappa}}(-1,s_0) \mathfrak{M}_{\Psi_n}^{\sharp, \gamma^{\kappa}}(-1, s_0)\Big)
+\| \langle D_x\rangle^{1/2}\underline{\Delta_{12}\breve{\Psi}_n} \langle D_x\rangle^{1/2}\|_{\mathcal{L}(H^{s_0})}\\
&\times \big( \mathfrak{M}_{R_n}^{\sharp, \gamma^{\kappa}}(-1, s_0, b_0)
+\|\langle  D_x \rangle^{1/2}  \underline{\langle \partial_{\varphi} \rangle^{b_0}(R_{n}\Psi_n)}\langle D_x\rangle^{1/2}\|_{\mathfrak{L}(H^{s_0})}
\big)\\
&+\mathfrak{M}_{\Phi_n^{-1}}^{\sharp, \gamma^{\kappa}}(-1, s_0, b_0)
\cdot \big( N_n^{-b_0}\|\langle D_x\rangle^{1/2} \underline{\langle \partial_{\varphi} \rangle^{b_0}\Delta_{12}R_n}\langle D_x\rangle^{1/2} \|_{\mathcal{L}(H^{s_0})}\\
&+ \| \langle D_x\rangle^{1/2}(\underline{\Delta_{12}R_{n}\Psi_n+\Delta_{12}\Psi_{n}R_n})\langle D_x\rangle^{1/2} \|_{\mathfrak{L}(H^{s_0})}
\big)\\
&+\mathfrak{M}_{\Phi_n^{-1}}^{\sharp, \gamma^{\kappa}}(-1, s_0)
\cdot \big(  \| \langle D_x\rangle^{1/2}\underline{\langle \partial_{\varphi} \rangle^{b_0}\Delta_{12}R_{n}}\langle D_x\rangle^{1/2}\|_{\mathfrak{L}(H^{s_0})}\\
& +\| \langle D_x\rangle^{1/2}\underline{\langle \partial_{\varphi} \rangle^{b_0}(\Delta_{12}R_{n}\Psi_n+\Delta_{12}\Psi_{n}R_n)}\langle D_x\rangle^{1/2}\|_{\mathfrak{L}(H^{s_0})} \big)\\
 \leq& _{b_0}k_2N_{n-1}\big(1+\mathfrak{M}_{R_0}^{\sharp, \gamma^{\kappa}}(-1, s_0, b_0)\gamma^{-\kappa}N_n^{2\tau+1}N_{n-1}^{-a_0}\big) \|\mathfrak{J}_1- \mathfrak{J}_2\|_{s_0+\sigma}\\
 \leq & k_2N_{n} \|\mathfrak{J}_1- \mathfrak{J}_2\|_{s_0+\sigma}.\\
 \end{split}
   \end{equation*}

$(4)_{n+1}$.
Assume that $\omega \in \Omega_{n+1}^{\gamma^{\kappa}}(\mathfrak{J}_1)$ and
 \begin{equation}\label{rho-n}
k_2 N_{n}^{\tau } \|\mathfrak{J}_1- \mathfrak{J}_2\|_{s_0+\sigma} \leq \rho, \quad 0<\rho<\gamma^{\kappa}/2.
 \end{equation}
We claim that $\omega \in \Omega_{n+1}^{\gamma^{\kappa}-\rho}(\mathfrak{J}_2).$
From \eqref{Omega-n} it follows  that
 \begin{equation*}
\Omega_{n+1}^{\gamma^{\kappa}}(\mathfrak{J}_1)
\subseteq  \Omega_{n}^{\gamma^{\kappa}}(\mathfrak{J}_1).
 \end{equation*}
By the above inductive hypothesis, we also have
 \begin{equation*}
 \Omega_{n}^{\gamma^{\kappa}}(\mathfrak{J}_1) \subseteq  \Omega_{n}^{\gamma^{\kappa}-\rho}(\mathfrak{J}_2).
\end{equation*}
Since $\rho < \gamma^{\kappa}/2$ implies $\gamma^{\kappa}-\rho> \gamma^{\kappa}/2$, we have
 \begin{equation*}
 \omega \in \Omega_{n}^{\gamma^{\kappa}-\rho}(\mathfrak{J}_2) \subseteq \Omega_{n}^{\gamma^{\kappa}/2}(\mathfrak{J}_2).
  \end{equation*}
Altogether,
 \begin{equation*}
\omega \in \Omega_{n}^{\gamma^{\kappa}}(\mathfrak{J}_1) \cap \Omega_{n}^{\gamma^{\kappa}/2}(\mathfrak{J}_2).
 \end{equation*}
For all $|l|\leq N_n, j\neq j'$,  from \eqref{del-djj'n} and \eqref{rho-n},
it follows that
\begin{equation*}
 \begin{split}
 |\omega \cdot l + d_j^{(n)}(\mathfrak{J}_2)- d_{j'}^{(n)}(\mathfrak{J}_2)|
\geq &|\omega \cdot l + d_j^{(n)}(\mathfrak{J}_1)- d_{j'}^{(n)}(\mathfrak{J}_1)|\\
&- |(d_j^{(n)}-d_{j'}^{(n)})(\mathfrak{J}_2)-(d_j^{(n)}-d_{j'}^{(n)})(\mathfrak{J}_1)|\\
\geq & \gamma^{\kappa} |j-j'| \langle l \rangle^{-\tau} - k_2 |j-j'|   \|\mathfrak{J}_1 -\mathfrak{J}_2 \|_{s_0+\sigma}\\
\geq &(\gamma^{\kappa}-\rho) |j-j'| \langle l \rangle^{-\tau}.\\
 \end{split}
    \end{equation*}
Hence we conclude that $\omega \in \Omega_{n+1}^{\gamma^{\kappa}-\rho}(\mathfrak{J}_2)$.

\subsection{Diagonalization theorem}

\begin{thm}\label{Dia-thm}(Diagonalization)
Assume  $\mathcal{S}>s_0$. There exist $\mu_3 \geq \mu_2$ where $\mu_2$  is given in Theorem \ref{Reg-thm}, $N_0 \in \mathbb{N}, \tau_0>0$ such that if $\mathfrak{J}$ satisfies the small condition in \eqref{sc-J} with $\mu=\mu_3$ and
 \begin{equation}\label{sc-eps}
N_0^{\tau_0}\varepsilon \gamma^{-1-\kappa} \leq 1,  \quad \kappa>1.
\end{equation}
Then the following results hold:

(1). For $j \in \mathbb{Z}$, there exists a sequence
\begin{equation}\label{dj-inf-def}
d_j^{\infty}=d_j^{(0)}+r_j^{\infty},  \quad d_j^{(0)}:=m_{\infty}j-\frac{(m_0+m_2)j}{1+j^2}, \quad r_j^{\infty}=-r_{-j}^{\infty} \in \mathbb{R}
\end{equation}
where $r_j^{\infty}=r_j^{\infty}(\omega, \mathfrak{J})$ depends on $\omega$ in a  Lipschitz way satisfying
 \begin{equation}\label{rj-inf-bd}
 \sup_{j} \langle j \rangle |r_j^{\infty}|^{\gamma^{\kappa}, \mathcal{O}_0}\leq C\varepsilon \gamma^{-1}.
  \end{equation}

(2). For all $\omega \in \mathcal{O}_{\infty}:=\mathcal{O}_{1} \cap \mathcal{O}_{2}$, where $\mathcal{O}_{1}$ is  defined in \eqref{O-1} and
\begin{equation*}
\begin{split}
\mathcal{O}_{2} := \Big\{ \omega \in \mathcal{O}_0: &|\omega \cdot l + d_j^{\infty}- d_{j'}^{\infty}| \geq \frac{2\gamma^{\kappa}|j-j'|}{\langle l \rangle^{\tau}}, \\
& \quad \forall l \in \mathbb{Z}^{\nu}, j,j' \in \mathbb{Z}, (j, j', l)\neq (j, j, 0)\Big\},\\
\end{split}
\end{equation*}
there exists a linear bounded T\"{o}plitz-in-time transformation $\Upsilon_2: \mathcal{O}_{\infty} \times H^s\rightarrow H^s$ with bounded inverse $\Upsilon_2^{-1}$
such that  $\hat{\mathcal{L}}$ in \eqref{hat-L} is conjugated to a constant-coefficient operator, viz.
  \begin{equation*}
\mathcal{L}_{\infty}:=\Upsilon_2^{-1} \hat{\mathcal{L}} \Upsilon_2= \omega \cdot \partial_{\varphi}+D_{\infty}, \quad D_{\infty}:= {\rm diag}_{j} ({\rm i}d_j^{\infty}).
 \end{equation*}
The transformations satisfy the following tame estimates for $s_0 \leq s\leq \mathcal{S}$:
 \begin{equation}\label{Ups2-Id}
\| (\Upsilon_2^{\pm 1}- {\rm Id})u \|_s^{\gamma^{\kappa}, \mathcal{O}_{\infty}} \leq_s  \varepsilon \gamma^{-1-\kappa}\| u \|_s^{\gamma^{\kappa}, \mathcal{O}_{\infty}} + \varepsilon \gamma^{-1-\kappa}\| \mathfrak{J} \|_{s+\mu_3}^{\gamma, \mathcal{O}_0}\| u \|_{s_0}^{\gamma^{\kappa}, \mathcal{O}_{\infty}}.
   \end{equation}
Moreover, for $\omega \in \mathcal{O}_{\infty}(\mathfrak{J}_1)\cap \mathcal{O}_{\infty}(\mathfrak{J}_2)$,
 \begin{equation}\label{del-rj-inf-bd}
\sup_{j} \langle j\rangle|\Delta_{12}r_j^{\infty}| \leq C \varepsilon \gamma^{-1}\| \mathfrak{J}_1 -\mathfrak{J}_2 \|_{s_0+\mu_3}.
\end{equation}
Finally, $\Upsilon_2$ and $\Upsilon_2^{-1}$ are real and reversibility-preserving, while $\mathcal{L}_{\infty}$ is real and reversible.
\end{thm}

 \begin{proof}
 Let $\mathcal{L}_0:=\hat{\mathcal{L}}, R_0:=\hat{R}$ in Theorem \ref{Reg-thm} and $k_1=\varepsilon$, $k_2=\varepsilon\gamma^{-1}$. It remains to prove that the initial assumptions for iteration:  \eqref{R0-bds} and \eqref{del-R0-bds} hold. Indeed, by  Lemmas \ref{S-mod} and \ref{Lrho-mod}, there exists $\mu_3>\mu_2$ such that if \eqref{sc-J} is satisfied with the choice of $\mu=\mu_3$, we have
  \begin{equation}\label{R0-s}
 \begin{split}
\mathfrak{M}_{R_0}^{\sharp, \gamma^{\kappa}}(-1, s),
\mathfrak{M}_{R_0}^{\sharp, \gamma^{\kappa}}(-1, s, b_0)
\leq& \max \{ \mathbb{M}_{\hat{\mathfrak{R}}}^{\gamma^{\kappa}}(s, \rho-2), |\hat{r}|_{-1, s+s_0+2+b_0, 0}^{\gamma^{\kappa}, \mathcal{O}_1} \}\\
\leq& \max \{  \varepsilon \gamma^{-1} \| \mathfrak{J}\|_{s+\mu_2}^{\gamma, \mathcal{O}_0}, \varepsilon \gamma^{-1}
\| \mathfrak{J}\|_{s+\mu_2+s_0+2+b_0}^{\gamma, \mathcal{O}_0} \}\\
\leq&  \varepsilon \gamma^{-1} \| \mathfrak{J}\|_{s+\mu_3}^{\gamma, \mathcal{O}_0}\\
 \end{split}
\end{equation}
and
\begin{equation*}
\begin{split}
\| \langle D_x \rangle^{1/2}\underline{\Delta_{12}R_0}\langle D_x \rangle^{1/2} \|_{\mathfrak{L}(H^{s_0})},&
\| \langle D_x \rangle^{1/2}\underline{\langle \partial_{\varphi}\rangle^{b_0}\Delta_{12}R_0}\langle D_x \rangle^{1/2} \|_{\mathfrak{L}(H^{s_0})}\\
\leq& \max \{ \mathbb{M}_{\Delta_{12}\hat{\mathfrak{R}}}(s_0, \rho-3), |\Delta_{12}\hat{r}|_{-1, s_0+b_0+3, 0} \}\\
\leq &\max \{\varepsilon \gamma^{-1} \| \mathfrak{J}_1-\mathfrak{J}_2\|_{s_0+\mu_2}^{\gamma, \mathcal{O}_0}, \varepsilon \gamma^{-1} \| \mathfrak{J}_1-\mathfrak{J}_2\|_{s_0+\mu_2+b_0+3}^{\gamma, \mathcal{O}_0}  \}\\
\leq & \varepsilon \gamma^{-1} \| \mathfrak{J}_1-\mathfrak{J}_2\|_{s_0+\mu_3}^{\gamma, \mathcal{O}_0},
 \end{split}
\end{equation*}
where $\rho\geq s_0+b_0+3$.
Therefore,
$R_0$ and $\langle\partial_{\varphi}\rangle^{b_0}R_0$ are  Lip-$-1$-modulo-tame operators satisfying
\eqref{R0-bds} and \eqref{del-R0-bds} for $\sigma=\mu_3$.

The smallness condition \eqref{sc-R0} in Proposition \ref{Ite-red} follows by the smallness condition on $\varepsilon$ in \eqref{sc-eps}.  From \eqref{djn-df-fb} in Proposition \ref{Ite-red}, we derive that for all $ j \in \mathbb{Z}$, the sequences $(d_j^{(n)})_{n \in \mathbb{N}}$, $(r_j^{(n)})_{n \in \mathbb{N}}$ in \eqref{djn-def} are Cauchy sequences.
Indeed, by \eqref{R0-bds} and \eqref{djn-df-fb},
\begin{equation*}
\begin{split}
\langle j\rangle |d_j^{(n+m)}-d_j^{(n)}|^{\gamma^{\kappa}, \mathcal{O}_{0}}
&=\langle j\rangle |r_j^{(n+m)}-r_j^{(n)}|^{\gamma^{\kappa}, \mathcal{O}_{0}}\\
  & \leq \sum_{k=n-1}^{n+m-2}N_k^{-a_0}\mathfrak{M}_{R_0}^{\sharp, \gamma^{\kappa}}(-1, s_0, b_0)
  \leq \varepsilon \gamma^{-1}N_{n-1}^{-a_0}.\\
   \end{split}
\end{equation*}
It implies that the sequences $d_j^{(n)}$ and $r_j^{(n)}$ have the limits $d_j^{\infty}$ and $r_j^{\infty}$ respectively.
Furthermore, let $m\rightarrow \infty$, one obtains
 \begin{equation}\label{dj-inf-djn}
 \langle j\rangle |d_j^{\infty}-d_j^{(n)}|^{\gamma^{\kappa}, \mathcal{O}_{0}}=\langle j\rangle |r_j^{\infty}-r_j^{(n)}|^{\gamma^{\kappa}, \mathcal{O}_{0}}   \leq \varepsilon \gamma^{-1}N_{n-1}^{-a_0}.
 \end{equation}
In particular, when $n=0$, we have
\begin{equation*}
\langle j\rangle |r_j^{\infty}-r_j^{(0)}|^{\gamma^{\kappa}, \mathcal{O}_{0}}\leq \varepsilon \gamma^{-1}N_{-1}^{-a_0}.
 \end{equation*}
 Therefore  \eqref{rj-inf-bd} holds.
 By \eqref{del-Rn-rjn-fbd} and \eqref{dj-inf-djn}, let $n \rightarrow \infty$, \eqref{del-rj-inf-bd} follows from the inequalities
 \begin{equation*}
 \begin{split}
 \langle j\rangle|r_j^{\infty}(\mathfrak{J}_1)- r_j^{\infty}(\mathfrak{J}_2)|
   \leq & \langle j\rangle|r_j^{\infty}(\mathfrak{J}_1)- r_j^{(n)}(\mathfrak{J}_1)|
   + \langle j\rangle|r_j^{(n)}(\mathfrak{J}_1)- r_j^{(n)}(\mathfrak{J}_2)| \\
    &+\langle j\rangle|r_j^{\infty}(\mathfrak{J}_2)- r_j^{(n)}(\mathfrak{J}_2)| \\
   \leq & C \varepsilon \gamma^{-1} \| \mathfrak{J}_1- \mathfrak{J}_2  \|_{s_0+\mu_3} + \varepsilon \gamma^{-1} N_{n-1}^{-a_0}.\\
    \end{split}
    \end{equation*}

 Now we need to verify that
\begin{equation*}
\mathcal{O}_{2} \subseteq \cap_{n \geq 0} \Omega_n^{\gamma^{\kappa}}.
\end{equation*}
 For $|l|\leq N_n$, $\langle l \rangle^{\tau}\leq N_n^{\tau}\leq N_{n-1}^{a_0}$, in view of \eqref{dj-inf-djn}, we obtain
  \begin{equation*}
 \begin{split}
|\omega \cdot l+d_j^{(n)}-d_{j'}^{(n)}|
\geq & |\omega \cdot l+d_j^{\infty}-d_{j'}^{\infty}|-|d_j^{(n)}-d_j^{\infty}|-|d_{j'}^{(n)}-d_{j'}^{\infty}|\\
 \geq  & \frac{2\gamma^{\kappa}|j-j'|}{\langle l \rangle^{\tau}}-\frac{2\varepsilon \gamma^{-1}} {N_{n-1}^{a_0}}
 \geq  \frac{\gamma^{\kappa}|j-j'|}{\langle l \rangle^{\tau}}.\\
  \end{split}
    \end{equation*}
Hence  for any $n \in \mathbb{N}$, $\mathcal{O}_{2} \subseteq  \Omega_n^{\gamma^{\kappa}}$.
  Thus the sequence $(\Psi_n)_{n \in \mathbb{N}}$ is well defined on $\mathcal{O}_{\infty}$.

  We define $\Gamma_n:=\Phi_0 \circ \cdots \circ \Phi_n, n\geq0$.
   Noting that   $\Gamma_{n}=\Gamma_{n-1}+\Gamma_{n-1}\Psi_n$, from Lemma \ref{mod-sum-com} and  \eqref{psin-fbd},
   it follows that
  \begin{equation*}
 \begin{split}
 \mathfrak{M}_{\Gamma_{n}}^{\sharp, \gamma^{\kappa}}(-1, s_0)
  \leq &  \mathfrak{M}_{\Gamma_{n-1}}^{\sharp, \gamma^{\kappa}}(-1, s_0)+ \mathfrak{M}_{\Gamma_{n-1}}^{\sharp, \gamma^{\kappa}}(-1, s_0)\mathfrak{M}_{\Psi_n}^{\sharp, \gamma^{\kappa}}(-1, s_0)\\
  \leq & \mathfrak{M}_{\Gamma_{n-1}}^{\sharp, \gamma^{\kappa}}(-1, s_0)\big(1+\mathfrak{M}_{R_{0}}^{\sharp, \gamma^{\kappa}}(-1, s_0, b_0) \gamma^{-\kappa} N_{n}^{2\tau+1}N_{n-1}^{-a_0}\big).\\
   \end{split}
    \end{equation*}
  Iterating the above inequality implies
   \begin{equation}\label{Gamn-s0}
 \begin{split}
 \mathfrak{M}_{\Gamma_n}^{\sharp, \gamma^{\kappa}}(-1, s_0)
  \leq & \mathfrak{M}_{\Gamma_{0}}^{\sharp, \gamma^{\kappa}}(-1, s_0)\prod_{n\geq 1}\left(1+\mathfrak{M}_{R_{0}}^{\sharp, \gamma^{\kappa}}(-1, s_0, b_0) \gamma^{-\kappa} N_{n}^{2\tau+1}N_{n-1}^{-a_0} \right)\\
   \leq &  \mathfrak{M}_{\Phi_{0}}^{\sharp, \gamma^{\kappa}}(-1, s_0)
   \leq C,   \quad  n\geq1.\\
 \end{split}
    \end{equation}
 For   the high norm, from Lemma \ref{mod-sum-com}, \eqref{psin-fbd} and \eqref{Gamn-s0}, we derive
    \begin{equation*}
 \begin{split}
 \mathfrak{M}_{\Gamma_{n}}^{\sharp, \gamma^{\kappa}}(-1, s)
  \leq_s & \mathfrak{M}_{\Gamma_{n-1}}^{\sharp, \gamma^{\kappa}}(-1, s)  \left( 1+\mathfrak{M}_{\Psi_n}^{\sharp, \gamma^{\kappa}}(-1, s_0)\right)+\mathfrak{M}_{\Psi_n}^{\sharp, \gamma}(-1, s)\\
 \leq_s  & \mathfrak{M}_{\Gamma_{n-1}}^{\sharp, \gamma^{\kappa}}(-1, s) \left(1+\mathfrak{M}_{R_{0}}^{\sharp, \gamma^{\kappa}}(-1, s_0, b_0) \gamma^{-\kappa}N_{n}^{2\tau+1} N_{n-1}^{-a_0}\right)\\
    & +\mathfrak{M}_{R_{0}}^{\sharp, \gamma^{\kappa}}(-1, s, b_0) \gamma^{-\kappa} N_{n}^{2\tau+1} N_{n-1}^{-a_0}, \quad n\geq1.\\
 \end{split}
    \end{equation*}
  Iterating the above estimates and using \eqref{sc-R0}, \eqref{psin-fbd} and
  \begin{equation*}
  \prod_{n\geq 1}\left(1+\mathfrak{M}_{R_{0}}^{\sharp, \gamma^{\kappa}}(-1, s_0, b_0) \gamma^{-\kappa} N_{n}^{2\tau+1} N_{n-1}^{-a_0}\right)\leq C,
  \end{equation*}
we obtain
\begin{equation}\label{Gamn-s}
\begin{split}
\mathfrak{M}_{\Gamma_{n}}^{\sharp, \gamma^{\kappa}}(-1, s)
& \leq_s \sum_{n=1}^{\infty}\mathfrak{M}_{R_{0}}^{\sharp, \gamma^{\kappa}}(-1, s, b_0) \gamma^{-\kappa} N_{n}^{2\tau+1} N_{n-1}^{-a_0} + \mathfrak{M}_{\Gamma_{0}}^{\sharp, \gamma^{\kappa}}(-1, s)\\
& \leq_s 1+\mathfrak{M}_{R_{0}}^{\sharp, \gamma^{\kappa}}(-1, s, b_0) \gamma^{-\kappa}.\\
  \end{split}
\end{equation}

From Lemma  \ref{mod-sum-com}, \eqref{sc-R0}, \eqref{psin-fbd},  \eqref{Gamn-s0} and \eqref{Gamn-s}, it follows that
 \begin{eqnarray}\label{Gam-cau}
 \begin{aligned}
 \mathfrak{M}_{\Gamma_{n+m}-\Gamma_{n}}^{\sharp, \gamma^{\kappa}}(-1, s)
   \leq_s &  \sum_{j=n+1}^{n+m} \mathfrak{M}_{\Gamma_{j}-\Gamma_{j-1}}^{\sharp, \gamma^{\kappa}}(-1, s)
   =\sum_{j=n+1}^{n+m} \mathfrak{M}_{\Gamma_{j-1}\Psi_j}^{\sharp, \gamma^{\kappa}}(-1, s)\\
  % \leq_s & \sum_{j=n+1}^{n+m} \left(\mathfrak{M}_{\Gamma_{j-1}}^{\sharp, \gamma^{\kappa}}(-1, s) \mathfrak{M}_{\Psi_{j}}^{\sharp, \gamma^{\kappa}}(-1, s_0)
  % +\mathfrak{M}_{\Gamma_{j-1}}^{\sharp, \gamma^{\kappa}}(-1, s_0) \mathfrak{M}_{\Psi_{j}}^{\sharp, \gamma^{\kappa}}(-1, s) \right)\\
    \leq_s & \sum_{j=n+1} \mathfrak{M}_{R_{0}}^{\sharp, \gamma^{\kappa}}(-1, s, b_0)\gamma^{-\kappa} N_{j}^{-1}\\
    \leq_s & \mathfrak{M}_{R_{0}}^{\sharp, \gamma^{\kappa}}(-1, s, b_0)\gamma^{-\kappa} N_{n+1}^{-1}.
 \end{aligned}
 \end{eqnarray}

Therefore, in the topology induced by the operator norm, as a result of Lemma \ref{L(Hs0)-mod}, $(\Gamma_n)_{n \in \mathbb{N}}$ is a Cauchy sequence and  the limit $\Gamma_{\infty}:=\lim_{n \rightarrow \infty}\Gamma_{n}$ exists.
Let $n=0, m\rightarrow \infty$ in \eqref{Gam-cau}, we deduce that
\begin{equation*}
\mathfrak{M}_{\Gamma_{\infty}-\Gamma_{0}}^{\sharp, \gamma^{\kappa}}(-1, s) \leq \mathfrak{M}_{R_{0}}^{\sharp, \gamma^{\kappa}}(-1, s, b_0)\gamma^{-\kappa}.
\end{equation*}
Also we have
\begin{equation*}
 \mathfrak{M}_{\Gamma_{0}-{\rm Id}}^{\sharp, \gamma^{\kappa}}(-1, s)=\mathfrak{M}_{\Psi_{0}}^{\sharp, \gamma^{\kappa}}(-1, s)
\leq \mathfrak{M}_{R_{0}}^{\sharp, \gamma^{\kappa}}(-1, s, b_0)\gamma^{-\kappa}.
\end{equation*}
Accordingly,
\begin{equation*}
\mathfrak{M}_{\Gamma_{\infty}-{\rm Id}}^{\sharp, \gamma^{\kappa}}(-1, s)\leq \mathfrak{M}_{R_{0}}^{\sharp, \gamma^{\kappa}}(-1, s, b_0)\gamma^{-\kappa}.
\end{equation*}
Let $\Upsilon_2:=\Gamma_{\infty}$, then \eqref{Ups2-Id} follows from   \eqref{R0-bds}, \eqref{R0-s}, Lemma \ref{L(Hs0)-mod} and Lemma \ref{mod-inv}.

Since all the $\Phi_n$'s are real and reversibility-preserving, it is true also for $\Gamma_n$ and  $\Gamma_{\infty}=\Upsilon_2$.
From the fact that $\hat{\mathcal{L}}$ is real and reversible, we derive that $\mathcal{L}_{\infty}$ is real and reversible too.
 \end{proof}

 \begin{rmk}\label{h0s}
 The discussion  in these two preceding sections is based on the phase space $H^s$. In fact, through the almost word-by-word reasoning as in Lemma 4.8 in \cite{bbm-1},  $(\Upsilon_1\Upsilon_2)^{\pm 1}:$ $H_0^s \rightarrow H_0^s$ isomorphically. Since $\mathcal{L}$ is defined on $H_0^s$, we  restrict $\mathcal{L}^{\infty}$ to the  smaller phase space $H_0^s$.
 \end{rmk}
%%%%%%%%%%%%%%%%%%%%%%%%%%%%%%%%%%%%%%%%%%%%%%%%%%%%%%%%%%%%
\renewcommand{\theequation}{\thesection.\arabic{equation}}
\setcounter{equation}{0}
%%%%%%%%%%%%%%%%%%%%%%%%%%%%%%%%%%%%%%%%%%%%%%%%%%%%%%%%%%%%

\section{ Measure estimates }

For $l \in \mathbb{Z}, j, j' \in \mathbb{Z}\setminus \{0\}$, let us define
\begin{equation}\label{PQ-def}
\begin{split}
P_{ljj'}(\gamma^{\kappa}, \tau):=&\{\omega \in \mathcal{O}_0: |\omega \cdot l +d_j^{\infty}-d_{j'}^{\infty}| < 2 |j-j'|\gamma^{\kappa} \langle l \rangle^{-\tau}
\},\\
Q_{l, j}(\gamma, \tau):=&\{\omega \in \mathcal{O}_0: |\omega \cdot l +m_{\infty} j| < 2 \gamma \langle l,j \rangle^{-\tau}
 \}.\\
 \end{split}
\end{equation}
In the following, we assume $j-j'\neq 0$. If not, then $P_{ljj'}=\varnothing$ and the following results hold trivially.

 We need some preliminary results as follows.
\begin{lemma}\label{dj0-dj'0}
Assume $m_2<0$ and $m_0+m_2\geq 0$. If there exists a positive constant $\delta_0$ such that $m_0 < -5 m_2- 4\delta_0,$ then
\begin{equation*}
\delta_0 |j-j'|\leq |d_j^{(0)}-d_{j'}^{(0)}| \leq C |j-j'|,
\end{equation*}
where  $C=|m_{\infty}|+\frac{5}{4}|m_0+m_2|$.
\end{lemma}
\begin{proof}
From \eqref{dj-inf-def}, it follows that
\begin{equation*}
d_j^{(0)}-d_{j'}^{(0)}=(j-j')\Big(m_{\infty}+(m_0+m_2)\frac{j'j-1}{(1+{j'}^2)(1+j^2)}\Big).
\end{equation*}
First, one can deduce that
\begin{equation*}
|m_{\infty}+(m_0+m_2)\frac{j'j-1}{(1+{j'}^2)(1+j^2)}| \leq |m_{\infty}|+\frac{5}{4}|m_0+m_2|=C.
\end{equation*}
Then from \eqref{m_inf-bd}, we derive that
\begin{equation*}
\begin{split}
-m_{\infty}&(1+{j'}^2)(1+j^2)+(m_0+m_2)(1-j'j)\\
\geq & -m_{\infty}(1+{j'}^2)(1+j^2)- (m_0+m_2)\frac{(1+{j'}^2)(1+j^2)}{4}+(m_0+m_2)\\
\geq & \delta_0 (1+{j'}^2)(1+j^2).\\
 \end{split}
\end{equation*}
Hence
\begin{equation*}
\left|m_{\infty}+(m_0+m_2)\frac{j'j-1}{(1+{j'}^2)(1+j^2)}\right|=\left|-m_{\infty}+(m_0+m_2)\frac{1-j'j}{(1+{j'}^2)(1+j^2)}\right|\geq \delta_0.
\end{equation*}
\end{proof}

\begin{lemma}\label{l-j-j'}
If $P_{ljj'}(\gamma^{\kappa}, \tau) \neq \emptyset $ and $\gamma^{\kappa} < \delta_0/6$, then there exists a positive constant $C_1$ such that $|l|\geq C_1|j-j'|$. If $Q_{l, j}(\gamma, \tau) \neq \emptyset $, then $|l|\geq C_2|j|$ for some positive constant $C_2$.
\end{lemma}
\begin{proof}
Recalling that $\mathcal{O}_0 \subset \Omega$ which is a compact set defined in \eqref{Omega}, we use the notation
$\omega_{{\rm max}}>0$ to denote the maximum value of $|\omega|$.
If $P_{ljj'} \neq \emptyset $, then there exists $\omega$ such that
\begin{equation*}
|d_j^{\infty}-d_{j'}^{\infty}|<2 \gamma^{\kappa} \langle l \rangle^{-\tau}|j-j'| +|\omega \cdot l|.
\end{equation*}
By Lemma \ref{dj0-dj'0}, \eqref{m_inf-bd} and \eqref{dj-inf-def}-\eqref{rj-inf-bd}, one obtains
\begin{equation*}
|d_j^{\infty}-d_{j'}^{\infty}|\geq |d_j^{(0)}-d_{j'}^{(0)}|-|r_j^{\infty}|-|r_{j'}^{\infty}| \geq \frac{\delta_0}{2}|j-j'|.
\end{equation*}
Therefore,
\begin{equation*}
|\omega||l|\geq \frac{\delta_0}{2}|j-j'|-2 \gamma^{\kappa} \langle l \rangle^{-\tau}|j-j'|
=|j-j'|(\frac{\delta_0}{2}-\frac{2 \gamma^{\kappa}}{ \langle l \rangle^{\tau}} )
\geq \frac{\delta_0}{6}|j-j'|.
\end{equation*}
Setting $C_1:=\delta_0(6\omega_{{\rm max}})^{-1}$, we have $|l|\geq C_1|j-j'|$.

We prove the second result of this lemma by contradiction. If $Q_{l, j} \neq \emptyset $,  let us assume that $|m_{\infty}j| > 2|\omega \cdot l|$. Recalling \eqref{O_0}, one has
\begin{equation*}
\frac{2 \gamma}{ \langle l,j \rangle^{\tau}} \geq |\omega \cdot l +m_{\infty} j| \geq|\omega \cdot l| \geq \frac{2 \gamma}{ \langle l \rangle^{\tau}}.
\end{equation*}
This leads to a contradiction. Whence $|m_{\infty}j| \leq 2|\omega \cdot l|$, from which we have $|l|\geq C_2|j|$, where $C_2:=|m_\infty|(2\omega_{{\rm max}})^{-1}$.
\end{proof}

\begin{lemma}
The measures of the sets $P_{ljj'}$ and $Q_{l, j}$ in \eqref{PQ-def} satisfy
\begin{equation}\label{P-me}
|P_{ljj'}(\gamma^{\kappa}, \tau)| \leq CL^{\nu-1} \gamma^{\kappa} \langle l\rangle^{-\tau}
\end{equation}
and
\begin{equation}\label{Q-me}
 |Q_{l, j}(\gamma, \tau)| \leq CL^{\nu-1} \gamma \langle l\rangle^{-\tau-1}.
 \end{equation}
\end{lemma}
\begin{proof}
We define the following two functions:
 \begin{equation*}
 \begin{split}
 \phi_P(\omega):=&\omega \cdot l+d_j^{\infty}(\omega)-d_{j'}^{\infty}(\omega)\\
 =&\omega \cdot l+(j-j')\left(m_{\infty}(\omega)+(m_0+m_2)\frac{j'j-1}{(1+{j'}^2)(1+j^2)}\right)+r_j^{\infty}(\omega)-r_{j'}^{\infty}(\omega),\\
\phi_Q(\omega):=&\omega \cdot l +m_{\infty}(\omega)j.\\
  \end{split}
    \end{equation*}
For any $\omega \in P_{ljj'}$, we split $\omega=s\hat{l}+v$ where $\hat{l}:=l/|l|, s\in \mathbb{R}$ and $l\cdot v=0$. Then by Lemma \ref{l-j-j'} and \eqref{rj-inf-bd}, for any $s_1, s_2 \in \mathbb{R}$, we deduce that
\begin{equation*}
 4 \gamma^{\kappa} |j-j'| \langle l \rangle^{-\tau} > |\phi_P(s_1\hat{l}+v)-\phi_P(s_2\hat{l}+v)|
 \end{equation*}
and
 \begin{equation*}
 \begin{split}
|\phi_P(s_1\hat{l}+v)-\phi_P(s_2\hat{l}+v)|
 \geq &|s_1-s_2|\Big(|l|-|j-j'| |m_{\infty}|^{lip, \mathcal{O}_0}\\
 &-|r_j^{\infty}|^{lip, \mathcal{O}_0}-|r_{j'}^{\infty}|^{lip, \mathcal{O}_0}\Big)
\geq \frac{C_1}{2}|j-j'| |s_1-s_2|.\\
 \end{split}
    \end{equation*}
Hence $|s_1-s_2| \leq 8 C_1^{-1} \gamma^{\kappa} \langle l \rangle^{-\tau}$.
By Fubini's theorem, $|P_{ljj'}|\leq C L^{\nu-1} \gamma^{\kappa} \langle l\rangle^{-\tau}$.

For any $\omega \in Q_{l, j}$, similarly, from Lemma \ref{l-j-j'} and \eqref{m_inf-bd}, it follows that
\begin{equation*}
4 \gamma \langle l \rangle^{-\tau} \geq 4 \gamma \langle l,j \rangle^{-\tau}
\geq |\phi_Q(s_1\hat{l}+v)-\phi_Q(s_2\hat{l}+v)|
\end{equation*}
and
\begin{equation*}
\begin{split}
 |\phi_Q(s_1\hat{l}+v)-\phi_Q(s_2\hat{l}+v)|
\geq |s_1-s_2|\big(|l|-|j| |m_{\infty}|^{lip, \mathcal{O}}\big)
\geq \frac{|l|}{2} |s_1-s_2|.
 \end{split}
\end{equation*}
Then $|s_1-s_2| \leq 8 \gamma \langle l \rangle^{-\tau-1}$.
By Fubini's theorem, $|Q_{l, j}|\leq CL^{\nu-1} \gamma \langle l\rangle^{-\tau-1}$.
\end{proof}
From now on, we fix  $\tau_1 \geq \nu+1$, $\tau \geq \tau_1+\nu+2$.
\begin{lemma}\label{PQ-re}
There exists $\gamma_0>0$ such that for any $\gamma< \gamma_0$, if $|j|, |j'| \geq  \langle l \rangle^{\tau_1} \gamma^{-1}$, then
 $P_{ljj'}(\gamma^{\kappa}, \tau) \subseteq Q_{l, j-j'}(\gamma, \tau_1)$.
\end{lemma}
\begin{proof}
We first claim that $|m_{\infty}(j-j')-(d_j^{(0)}-d_{j'}^{(0)})|\leq 2|m_0+m_2||j-j'|\left|\frac{1}{{j'} j}\right|$.
 In fact,
    \begin{equation*}
 \begin{split}
 |m_{\infty}(j-j')&-(d_j^{(0)}-d_{j'}^{(0)})|=|m_0+m_2||j-j'|\left|\frac{-1+j'j}{(1+{j'}^2)(1+j^2)}\right|\\
 \leq& |m_0+m_2||j-j'|\left(\left|\frac{j'j}{(1+{j'}^2)(1+j^2)}\right|+\left|\frac{1}{(1+{j'}^2)(1+j^2)}\right|\right)\\
    \leq & 2|m_0+m_2||j-j'|\left|\frac{j'j}{(1+{j'}^2)(1+j^2)}\right|
    \leq  2|m_0+m_2||j-j'|\left|\frac{j'j}{{j'}^2 j^2}\right|\\
    \leq &2|m_0+m_2||j-j'|\left|\frac{1}{{j'} j}\right|.\\
      \end{split}
    \end{equation*}
Suppose that $\omega \in P_{ljj'}(\gamma^{\kappa}, \tau)$, by Lemma \ref{l-j-j'} and \eqref{rj-inf-bd},
we obtain
 \begin{equation*}
 \begin{split}
 |\omega \cdot l +m_{\infty}(j-j')|=&|\omega \cdot l +d_j^{\infty}-d_{j'}^{\infty}+ m_{\infty}(j-j')-(d_j^{\infty}-d_{j'}^{\infty})|\\
 \leq & |\omega \cdot l +d_j^{\infty}-d_{j'}^{\infty}|+ |m_{\infty}(j-j')-(d_j^{(0)}-d_{j'}^{(0)})|+|r_j^{\infty}|+|r_{j'}^{\infty}|\\
<  &|j-j'|\frac{2 \gamma^{\kappa}}{ \langle l \rangle^{\tau}} +|j-j'|\frac{2|m_0+m_2|}{|j||j'|}+\frac{C\varepsilon\gamma^{-1}}{\min \{|j|,|j'|\}}\\
\leq &\frac{2 \gamma^{\kappa}}{ C_1\langle l \rangle^{\tau-1}} +\frac{2|m_0+m_2|\gamma^2}{C_1 \langle l \rangle^{2\tau_1-1}}+\frac{C\varepsilon }{ \langle l \rangle^{\tau_1}}
  \leq \frac{ 2\gamma}{ \langle l \rangle^{\tau_1}}.\\
   \end{split}
    \end{equation*}
 By Lemma \ref{l-j-j'}, we have
     \begin{equation*}
 \begin{split}
  |\omega \cdot l +m_{\infty}(j-j')|
    < & \frac{2 \gamma^{\kappa}}{ C_1^{\tau}|j-j'|^{\tau-1}} +\frac{2|m_0+m_2|\gamma^2}{C_1^{2\tau_1} |j-j'|^{2\tau_1-1}}+\frac{C\varepsilon }{ C_1^{\tau_1}|j-j'|^{\tau_1}}\\
    < & \frac{ 2\gamma}{ \langle j-j' \rangle^{\tau_1}}.\\
 \end{split}
    \end{equation*}
Thus we conclude that  $\omega \in Q_{l, j-j'}(\gamma, \tau_1)$.
\end{proof}

\begin{thm}
Let $\mathcal{O}_{\infty}$ be the set of parameters in Theorem \ref{Dia-thm}. Then there exists some constant $C>0$ such that
\begin{equation*}
|\mathcal{O}_{0} - \mathcal{O}_{\infty}|\leq C \gamma^{\min\{1, \kappa-1\}} L^{\nu-1}.
\end{equation*}
\end{thm}
\begin{proof}
From \eqref{PQ-def}, we derive that
\begin{equation*}
|\mathcal{O}_{0} - \mathcal{O}_{\infty}| \leq  \left|\bigcup\limits_{\substack{l \in \mathbb{Z}^{\nu}, \\j, j' \in \mathbb{Z}\setminus\{0\}}} P_{ljj'}(\gamma^{\kappa}, \tau)\right| + \left|\bigcup\limits_{\substack{l \in \mathbb{Z}^{\nu},\\ j \in \mathbb{Z}\setminus\{0\}}} Q_{l, j}(\gamma, \tau)\right|.
\end{equation*}
According to Lemma \ref{l-j-j'} and \eqref{Q-me}, the measure of the second summand is estimated by
\begin{equation*}
\begin{split}
\sum_{\substack{l \in \mathbb{Z}^{\nu}, j\neq 0, \\|j|\leq |l|/C_2}} |Q_{l, j}(\gamma, \tau)|
\leq& \sum_{\substack{l \in \mathbb{Z}^{\nu}, j\neq 0, \\ |j|\leq |l|/C_2}} C L^{\nu-1} \gamma \langle l \rangle^{-\tau-1}
\leq  \frac{C}{C_2} L^{\nu-1} \gamma \sum_{l \in \mathbb{Z}^{\nu}}  \langle l \rangle^{-\tau}\\
\leq&  C L^{\nu-1} \gamma.\\
\end{split}
\end{equation*}
For the first one, we split the union set into two groups:
\begin{equation*}
\left|\bigcup\limits_{\substack{l \in \mathbb{Z}^{\nu},\\ j, j' \in \mathbb{Z}\setminus\{0\}}} P_{ljj'}(\gamma^{\kappa}, \tau)\right|=
\left|\bigcup\limits_{\substack{l \in \mathbb{Z}^{\nu},\\ |j|, |j'| \geq  \langle l \rangle^{\tau_1} \gamma^{-1}}} P_{ljj'}(\gamma^{\kappa}, \tau)\right| + \sum_{\substack{l \in \mathbb{Z}^{\nu},  j,j'\neq 0, \\|j| \leq  \langle l \rangle^{\tau_1} \gamma^{-1} \\{\rm or} |j'| \leq  \langle l \rangle^{\tau_1} \gamma^{-1}}} |P_{ljj'}(\gamma^{\kappa}, \tau)|.
\end{equation*}
From \eqref{Q-me} and Lemma \ref{PQ-re} it follows that
\begin{equation*}
 \begin{split}
 \left|\bigcup\limits_{\substack{l \in \mathbb{Z}^{\nu}, \\|j|, |j'| \geq  \langle l \rangle^{\tau_1} \gamma^{-1}}} P_{ljj'}(\gamma^{\kappa}, \tau)\right|
\leq &\left|\bigcup\limits_{\substack{l \in \mathbb{Z}^{\nu}, \\j-j'=h, \\0\neq |h|\leq |l|/C_2 }} Q_{l, h}(\gamma, \tau_1)\right|
\leq  \sum_{\substack{l \in \mathbb{Z}^{\nu},\\ 0\neq |h|\leq |l|/C_2}} |Q_{l, h}(\gamma, \tau_1)|\\
\leq & C L^{\nu-1} \gamma.\\
 \end{split}
 \end{equation*}
By Lemma \ref{l-j-j'} and \eqref{P-me}, we deduce that
\begin{equation*}
\begin{split}
\sum_{\substack{l \in \mathbb{Z}^{\nu}, j,j'\neq 0, \\ |j-j'| \leq |l|/C_1,\\|j| \leq  \langle l \rangle^{\tau_1} \gamma^{-1} \\{\rm or} |j'| \leq  \langle l \rangle^{\tau_1} \gamma^{-1}}} |P_{ljj'}(\gamma^{\kappa}, \tau)|
& \leq C \gamma^{\kappa} L^{\nu-1} \sum_{l \in \mathbb{Z}^{\nu}}\frac{|l|\langle l \rangle^{\tau_1}}{\gamma\langle l \rangle^{\tau}}\\
& \leq C  \gamma^{\kappa-1} L^{\nu-1} \sum_{l \in \mathbb{Z}^{\nu}} \langle l \rangle^{-(\tau-\tau_1-1)}\\
& \leq C L^{\nu-1} \gamma^{\kappa-1}.\\
 \end{split}
\end{equation*}
Thus the proof is completed.
\end{proof}

\begin{proof}(Proof of Theorem \ref{thm2})
We set $\Upsilon=\Upsilon_1 \circ \Upsilon_2$ (see $\Upsilon_1$ in Theorem \ref{Reg-thm} and  $\Upsilon_2$ in Theorem  \ref{Dia-thm}) and take $\mu=\mu_3$  in Theorem \ref{Dia-thm}. Then the bound for $\Upsilon u$ follows from \eqref{Ups1u}
and \eqref{Ups2-Id}.
\end{proof}

\begin{proof}(Proof of Corollary \ref{cor})
Under the action of the transformation $\Upsilon$, \eqref{dynam-eq} is transformed into $g_t=-D_{\infty}g$ which amounts to $(g_j)_t=-{\rm i}d_j^{\infty}g_j, j\in \mathbb{Z}\setminus \{ 0 \}.$  For each $j$, we  solve that $g_j(t)={\rm e}^{-{\rm i}d_j^{\infty}t}g_j(0)$. So the unique solution for the transformed system is
\begin{equation*}
g(t)=\sum_{j\in \mathbb{Z}}g_j(t){\rm e}^{{\rm i}jx}.
\end{equation*}
Due to the definition of the norm of $H^s$, we have $\| g(t)\|_{H^{s}}=\| g(0)\|_{H^{s}}$.
Then from $h=\Upsilon(\omega t)g$ and the bounds for  $\Upsilon^{\pm 1}$ in Theorem \ref{thm2}, we deduce that
\begin{equation*}
\begin{split}
\sup_{t\in \mathbb{R}} \| h(t, \cdot) \|_s
&= \sup_{t\in \mathbb{R}} \| \Upsilon g(t, \cdot) \|_s
\leq \sup_{t\in \mathbb{R}} \| \Upsilon g(t, \cdot) \|_s^{\gamma^{\kappa}, \mathcal{O}_{\infty}}\\
&\leq_s \| g(t, \cdot) \|_s + \varepsilon \gamma^{-1-\kappa}\| \mathfrak{J} \|_{s+\mu}^{\gamma, \mathcal{O}_{0}}
\| g(t, \cdot) \|_{s_0}\\
&\leq_s  \| g(0, \cdot) \|_s + \varepsilon \gamma^{-1-\kappa}\| \mathfrak{J} \|_{s+\mu}^{\gamma, \mathcal{O}_{0}}
\| g(0, \cdot) \|_{s_0}= \tilde{C}(s) \|g(0, \cdot) \|_s \\
&= \tilde{C}(s) \|\Upsilon^{-1}(0) h(0, \cdot) \|_s
\leq_s  \| h(0, \cdot) \|_s + \varepsilon \gamma^{-1-\kappa}\| \mathfrak{J} \|_{s+\mu}^{\gamma, \mathcal{O}_{0}}
\| h(0, \cdot) \|_{s_0}\\
& \leq C(s)  \| h(0, \cdot) \|_s,\\
 \end{split}
\end{equation*}
where the constants $\tilde{C}(s)$ and $C(s)$ depend on $\| \mathfrak{J} \|_{s+\mu}^{\gamma, \mathcal{O}_{0}}$.
\end{proof}

%%%%%%%%%%%%%%%%%%%%%%%%%%%%%%%%%%%%%%%%%%%%%%%%%%%%%%%%%%%%
\renewcommand{\theequation}{\thesection.\arabic{equation}}
\setcounter{equation}{0}
%%%%%%%%%%%%%%%%%%%%%%%%%%%%%%%%%%%%%%%%%%%%%%%%%%%%%%%%%%%%

\bigskip

\vskip 0.5cm

\noindent {\bf Acknowledgments.} The work of Wu is supported by the National NSF of China Grant-11631007. The work of Fu is supported by the National NSF of China Grants-11471259 and 11631007 and the National Science Basic Research Program of Shaanxi Province (Program No. 2019JM-007 and 2020JC-37). The work of Qu is supported by the National NSF of China Grants-11631007, 11971251 and 12111530003.

\begin{appendix}

\section{Lip-$\sigma$-tame operators and pseudo differential operators}

In this appendix, we review the definitions and the basic facts  and
the properties of the Lip-$\sigma$-tame operators and the pseudo differential operators.

\begin{defi}(Lip-$\sigma$-tame operators) \cite{bbhm, bm, fgp-1}
Let the  linear operator $A:=A(\omega)$ be defined on $\mathcal{O} \subset \mathbb{R}^{\nu}$. A is said to be  Lip-$\sigma$-tame
if there exist $\sigma \geq 0$ and a non-decreasing sequence\\ $\{ \mathfrak{M}_A^{\gamma}(\sigma, s) \}_{s=s_0}^\mathcal{S} $ (with possibly $\mathcal{S}=+\infty$)
such that for any $u \in H^{s+\sigma}$,
\begin{equation*}
\sup \limits_{\omega \in \mathcal{O}} \|Au\|_s, \sup \limits_{\substack{\omega, \omega' \in \mathcal{O},\\  \omega \neq \omega'}} \gamma
\| (\Delta_{\omega, \omega'}A) u \|_{s-1}
\leq_s \mathfrak{M}_A^{\gamma}(\sigma, s)\|u\|_{s_0+\sigma} + \mathfrak{M}_A^{\gamma}(\sigma, s_0)\|u\|_{s+\sigma}.
\end{equation*}
We call $\mathfrak{M}_A^{\gamma}(\sigma, s)$ a Lip-$\sigma$-tame constant of the operator $A$ or tame constant for short.
\end{defi}

The following lemma is an  important property of the Lip-$\sigma$-tame operators.

\begin{lemma}\label{Tam-com}(Composition) \cite{bbhm, bm, fgp-1}
Let $A, B$ be respectively Lip-$\sigma_A$-tame and  Lip-$\sigma_B$-tame operators with tame constants respectively $\mathfrak{M}_A^{\gamma}(\sigma_A, s)$ and $\mathfrak{M}_B^{\gamma}(\sigma_B, s)$. Then the composed operator $A\circ B$ is a Lip-$(\sigma_A+\sigma_B)$-tame operator with tame constant
\begin{equation*}
\mathfrak{M}_{AB}^{\gamma}(\sigma_A+\sigma_B, s) \leq \mathfrak{M}_A^{\gamma}(\sigma_A, s) \mathfrak{M}_B^{\gamma}(\sigma_B, s_0+\sigma_A)
+\mathfrak{M}_A^{\gamma}(\sigma_A, s_0) \mathfrak{M}_B^{\gamma}(\sigma_B, s+\sigma_A).
\end{equation*}
\end{lemma}

\begin{defi}\label{pd-def}(Pseudo differential operators) \cite{hom}
Let $u=\sum_{j\in \mathbb{Z}} u_j {\rm e}^{{\rm i}jx}$. A linear operator $A$ is called  pseudo differential operator of order $\leq m $ if its action on
any $H^s(\mathbb{T}, \mathbb{R})$ with $s\geq m$ is given  by
\begin{equation*}
(Au)(x)=\sum_{j\in \mathbb{Z}} a(x, j)u_j {\rm e}^{{\rm i}jx},
\end{equation*}
 where $a(x, j)$, called the  symbol of $A$,  is the restriction to $\mathbb{T}\times \mathbb{Z}$ of the function $a(x, \xi) $ which is $C^{\infty}$-smooth on
$\mathbb{T}\times \mathbb{R}$ and satisfies
 \begin{equation}\label{pds}
 |\partial_x^{\alpha} \partial_{\xi}^{\beta} a(x, \xi)| \leq C_{\alpha, \beta} \langle \xi \rangle^{m-\beta}, \quad \forall \alpha, \beta \in \mathbb{N}.
   \end{equation}

We denote the pseudo differential operator with symbol $a(x, j)$ by ${\rm Op}(a)$, i.e., $A={\rm Op}(a)$. We define the class $OPS^m$ as the set of pseudo differential operators of order at most $m$ and   the class $S^m$ as the set of symbols satisfying \eqref{pds}.
\end{defi}

In the present paper,
we  consider the $\varphi$-dependent families of pseudo differential operators
\begin{equation*}
(Au)(\varphi, x):=\sum_{j\in \mathbb{Z}} a(\varphi, x, j)u_j(\varphi) {\rm e}^{{\rm i}jx},
\end{equation*}
where  $a(\varphi, x, \xi)$ is $C^{\infty}$-smooth on $\mathbb{T}^{\nu+1}\times \mathbb{R}.$

\begin{defi}\cite{hom}
Let $A={\rm Op}(a(\varphi, x, \xi)) \in OPS^m$, we define
 \begin{equation}\label{pd-norm}
|A|_{m, s, \alpha}:=|a|_{m, s, \alpha}:=\max_{0 \leq \beta \leq \alpha}\sup \limits_{\xi \in \mathbb{R}} \|\partial_{\xi}^{\beta}a(\cdot, \cdot, \xi)\|_s \langle \xi \rangle^{-m+\beta}.
\end{equation}
\end{defi}
The norm defined above controls the regularity in $(\varphi, x)$ in the Sobolev norm $\| \cdot \|_s$ and the decay in $\xi$ of the symbols $a(\varphi, x, \xi)$.

 If $A={\rm Op}(a) \in OPS^m$  with symbol $a(\omega, \varphi, x, \xi)$ depending on the parameter
 $\omega \in \mathcal{O} \subset \mathbb{R}^{\nu}$ in a Lipschitz way, we define the weighted Lipschitz norm
 \begin{equation}\label{pd-Lnorm}
 \begin{split}
|A|_{m, s, \alpha}^{\gamma, \mathcal{O}}:=|a|_{m, s, \alpha}^{\gamma, \mathcal{O}}
:=&\sup \limits_{\omega \in \mathcal{O}}|A|_{m, s, \alpha}\\
&+ \gamma \sup \limits_{\substack{\omega, \omega' \in \mathcal{O},\\    \omega \neq \omega'}}
\frac{|{\rm Op}(a(\omega, \cdot, \cdot, \cdot))-{\rm Op}(a(\omega', \cdot, \cdot, \cdot))|_{m, s-1, \alpha}}{|\omega-\omega'|}.\\
 \end{split}
\end{equation}

We point  out that the norms defined in \eqref{pd-norm} and \eqref{pd-Lnorm}  are both non-decreasing in $s, \alpha$ and non-increasing in $m$. Moreover, for a symbol independent of $\xi$,
 \begin{equation}\label{pd-c0}
 |{\rm Op}(a(\varphi, x))|^{\gamma, \mathcal{O}}_{0, s, \alpha}=\|a\|^{\gamma, \mathcal{O}}_s, \quad \forall s \geq 0.
 \end{equation}
 If the symbol depends only on $\xi$, we simply have
\begin{equation}\label{pd-c1}
|a(\xi)|^{\gamma, \mathcal{O}}_{m, s, \alpha}=|a(\xi)|_{m, s, \alpha} \leq C(m, a), \quad \forall s \geq 0.
\end{equation}

We  present some important properties of the pseudo differential operators.
\begin{lemma}\label{pd-com}(Composition) \cite{bbhm, bm, fgp-1}
Let $A={\rm Op}(a)$, $B={\rm Op}(b)$ be pseudo differential operators with symbols $a(\omega, \varphi, x, \xi) \in S^m$ and $b(\omega, \varphi, x, \xi) \in S^{m'}$ respectively. Then  $A \circ B \in OPS^{m+m'}$ with norm
\begin{equation*}
| A \circ B |^{\gamma, \mathcal{O}}_{m+m', s, \alpha} \leq_{m, \alpha} C(s)|A|^{\gamma, \mathcal{O}}_{m, s, \alpha} |B|^{\gamma, \mathcal{O}}_{m', s_0+\alpha+|m|, \alpha}
+C(s_0)|A|^{\gamma, \mathcal{O}}_{m, s_0, \alpha} |B|^{\gamma, \mathcal{O}}_{m', s+\alpha+|m|, \alpha}.
\end{equation*}
\end{lemma}

\begin{lemma}\label{R-tam-pd} \cite{bm, fgp-1}
Let $A ={\rm Op}(a) \in OPS^0$ be a  pseudo differential operator.

(1). If $|A|_{0, s, 0} < \infty (s\geq s_0)$, then there exist $C(s_0), C(s)>0$ such that
\begin{equation*}
\| Au \|_s \leq C(s_0)|A|_{0, s_0, 0}\| u \|_s+C(s)|A|_{0, s, 0}\| u \|_{s_0}.
\end{equation*}

(2). If $A$ is Lipschitz in the parameter $\omega \in \mathcal{O} \subset \mathbb{R}^{\nu}$ and $|A|_{0, s, 0}^{\gamma, \mathcal{O}} < \infty (s\geq s_0)$, then $A$ is a Lip-0-tame operator with tame  constant
\begin{equation*}
\mathfrak{M}_A^{\gamma}(0, s)\leq_s  |A|_{0, s, 0}^{\gamma, \mathcal{O}}.
\end{equation*}
\end{lemma}
The above lemma implies that the norm $|\cdot|_{0, s, 0}^{\gamma, \mathcal{O}}$ controls the action of  a  pseudo differential operator on the Sobolev spaces $H^s$.

The next lemma exhibits the connection between the class $OPS^{-1}$ and the class of Lip-$-1$-modulo-tame operators. It is a direct consequence of Lemma C.8 and Lemma C.10 in \cite{fgp-2}.

\begin{lemma}\label{S-mod}
Assume  $b \in \mathbb{N}$, $A={\rm Op}(a) \in OPS^{-1}$ with its symbol $a(\omega, i(\omega))$ depending on $\omega \in \mathcal{O} \subset \mathbb{R}^{\nu}$ in a Lipschitz way and on $i$ as well. Then $A$ and $\langle \partial_{\varphi}\rangle^{b} A$ are  Lip-$-1$-modulo-tame operators with modulo-tame constants
\begin{equation*}
\mathfrak{M}_A^{\sharp, \gamma}(-1, s)\leq_s |a|_{-1, s+s_0+2, 0}^{\gamma, \mathcal{O}},  \quad \mathfrak{M}_A^{\sharp, \gamma}(-1, s, b)\leq_s |a|_{-1, s+s_0+b+2, 0}^{\gamma, \mathcal{O}}.
\end{equation*}
Moreover,
 \begin{eqnarray*}
 \begin{aligned}
 &\|\langle D_x \rangle^{1/2} \underline{\Delta_{12}A} \langle D_x \rangle^{1/2}\|_{\mathfrak{L}(H^{s_0})}\leq |\Delta_{12}a|_{-1, s_0+b+3, 0},\\
 &\| \langle D_x \rangle^{1/2}\underline{\Delta_{12}\langle \partial_{\varphi}\rangle^{b} A}\langle D_x \rangle^{1/2} \|_{\mathfrak{L}(H^{s_0})} \leq |\Delta_{12}a|_{-1, s_0+b+3, 0}.
 \end{aligned}
 \end{eqnarray*}
 \end{lemma}

\section{The operators $\mathfrak{L}_{\rho, p}$}

We first introduce the classes of operators $\mathfrak{L}_{\rho, p}$ which are sufficiently smooth in the $x$-variable.
\begin{defi}\label{L-rhop}\cite{fgp-1, fgp-2}
Given $\rho, p, \mathcal{S} \in \mathbb{N}$ with possibly $\mathcal{S}= +\infty$, $\rho \geq 3$, $s_0 \leq p < \mathcal{S}$.  We denote by $\mathfrak{L}_{\rho, p}=\mathfrak{L}_{\rho, p}(\mathcal{O})$ the set of the linear operators
$A=A(\omega): H^s(\mathbb{T}^{\nu+1}) \rightarrow H^s(\mathbb{T}^{\nu+1})$, $\omega \in \mathcal{O} \subseteq \mathbb{R}^{\nu}$ with the  following properties:

(1). The operator A is Lipschitz in $\omega$.

(2). For all $\vec{b}=(b_1, \ldots, b_{\nu}) \in \mathbb{N}^{\nu}$ with $0\leq |\vec{b}|\leq \rho-2$ and for any $s_0 \leq s \leq \mathcal{S}$,

\quad (i). The operator $\langle D_x \rangle^{m_1} \partial_{\varphi}^{\vec{b}} A \langle D_x \rangle^{m_2} $ is Lip-0-tame for any $m_1, m_2 \in \mathbb{R}$, $m_1, m_2 \geq 0$ and $m_1+m_2=\rho-|\vec{b}|$.  We set
\begin{equation*}
\mathfrak{M}_{\partial_{\varphi}^{\vec{b}}A}^{\gamma}(-\rho+|\vec{b}|, s):=\sup_{\substack{m_1+m_2=\rho-|\vec{b}|,\\ m_1, m_2 \geq 0}}
\mathfrak{M}_{\langle D_x \rangle^{m_1} \partial_{\varphi}^{\vec{b}} A \langle D_x \rangle^{m_2}}^{\gamma}(0, s).
\end{equation*}

\quad (ii). The operator $\langle D_x \rangle^{m_1} [\partial_{\varphi}^{\vec{b}}A, \partial_x] \langle D_x \rangle^{m_2} $ is Lip-0-tame for any $m_1, m_2 \in \mathbb{R}$, $m_1, m_2 \geq 0$ and $m_1+m_2=\rho-|\vec{b}|-1$. We set
\begin{equation*}
\mathfrak{M}_{[\partial_{\varphi}^{\vec{b}}A, \partial_x]}^{\gamma}(-\rho+|\vec{b}|+1, s):=\sup_{\substack{m_1+m_2=\rho-|\vec{b}|-1, \\m_1, m_2 \geq 0}}
\mathfrak{M}_{\langle D_x \rangle^{m_1} [\partial_{\varphi}^{\vec{b}}A, \partial_x] \langle D_x \rangle^{m_2} }^{\gamma}(0, s).
\end{equation*}

We denote for $0\leq b \leq \rho-2$
\begin{equation*}
\mathbb{M}_A^{\gamma}(s, b):=\max_{0\leq |\vec{b}| \leq b} \max \left\{\mathfrak{M}_{\partial_{\varphi}^{\vec{b}}A}^{\gamma}(-\rho+|\vec{b}|, s),   \mathfrak{M}_{[\partial_{\varphi}^{\vec{b}}A, \partial_x]}^{\gamma}(-\rho+|\vec{b}|+1, s) \right\}.
\end{equation*}

(3). Suppose that the set $\mathcal{O}$ and the operator $A$  depend on  $i=i(\omega)$, and are well defined for $\omega \in \mathcal{O} $   for all $i=i(\omega)$ satisfying \eqref{sc-J}. For $\omega \in \mathcal{O}(i_1)\cap \mathcal{O}(i_2)$,
and for all $\vec{b}=(b_1, \ldots, b_{\nu}) \in \mathbb{N}^{\nu}$ with $0\leq |\vec{b}|\leq \rho-3$,

\quad(i). $\langle D_x \rangle^{m_1} \partial_{\varphi}^{\vec{b}} \Delta_{12}A \langle D_x \rangle^{m_2}$ is a bounded operator on $H^{p}$ in itself for any $m_1, m_2 \in \mathbb{R}$, $m_1, m_2 \geq 0$ and $m_1+m_2=\rho-|\vec{b}|-1$. More precisely, there is a positive constant $\mathfrak{N}_{\partial_{\varphi}^{\vec{b}}\Delta_{12}A}(-\rho+|\vec{b}|+1, p)$ such that for any $h \in H^{p}$,
\begin{equation*}
\sup_{\substack{m_1+m_2=\rho-|\vec{b}|-1,\\ m_1, m_2 \geq 0}}\| \langle D_x \rangle^{m_1} \partial_{\varphi}^{\vec{b}} \Delta_{12}A \langle D_x \rangle^{m_2}  h\|_{p} \leq \mathfrak{N}_{\partial_{\varphi}^{\vec{b}}\Delta_{12}A}(-\rho+|\vec{b}|+1, p) \| h \|_{p}.
\end{equation*}

\quad(ii). $\langle D_x \rangle^{m_1} [\partial_{\varphi}^{\vec{b}} \Delta_{12}A, \partial_x]\langle D_x \rangle^{m_2}$ is a bounded operator on $H^{p}$ in itself for any $m_1, m_2 \in \mathbb{R}$, $m_1, m_2 \geq 0$ and $m_1+m_2=\rho-|\vec{b}|-2$. More precisely, there is a positive constant $\mathfrak{N}_{[\partial_{\varphi}^{\vec{b}}\Delta_{12}A, \partial_x]}(-\rho+|\vec{b}|+2, p)$ such that for any $h \in H^{p}$,
\begin{equation*}
\sup_{\substack{m_1+m_2=\rho-|\vec{b}|-2,\\ m_1, m_2 \geq 0}}\| \langle D_x \rangle^{m_1} [\partial_{\varphi}^{\vec{b}} \Delta_{12}A, \partial_x] \langle D_x \rangle^{m_2}  h\|_{p} \leq \mathfrak{N}_{[\partial_{\varphi}^{\vec{b}}\Delta_{12}A, \partial_x]}(-\rho+|\vec{b}|+2, p) \| h \|_{p}.
\end{equation*}

We define for $0\leq b\leq \rho-3$
\begin{align*}
&\mathbb{M}_{\Delta_{12}A}(p, b)\\
&:= \max_{0\leq |\vec{b}| \leq b} \max\left\{\mathfrak{N}_{\partial_{\varphi}^{\vec{b}}\Delta_{12}A}(-\rho+|\vec{b}|+1, p),   \mathfrak{N}_{[\partial_{\varphi}^{\vec{b}}\Delta_{12}A, \partial_x]}(-\rho+|\vec{b}|+2, p) \right\}.
\end{align*}
\end{defi}

\begin{lemma}\label{pd-rhop-inv}
Let $a \in S^{-1}$, $T \in \mathfrak{L}_{\rho, p}$ with $\rho\geq 3$ and consider the operator $I-({\rm Op}(a)+T)$. There exists a constant $C(\mathcal{S}, \alpha, \rho)$ such that if
\begin{equation*}
C(\mathcal{S}, \alpha, \rho)\big(|a|_{-1, p+(\rho-1)(\rho-2)+3, \rho-2}^{\gamma, \mathcal{O}}+\mathbb{M}_{T}^{\gamma}(s_0, b)\big)<1,
\end{equation*}
where $\mathcal{S}$ is the fixed constant in Definition \ref{L-rhop}, then ${\rm{Id}}-({\rm Op}(a)+T)$ is invertible and
\begin{equation*}
({\rm{Id}}-({\rm Op}(a)+T))^{-1}={\rm{Id}}+{\rm Op}(c_1)+\mathfrak{R}_{\rho},
\end{equation*}
\begin{equation*}
({\rm{Id}}-({\rm Op}(a)+T))^{-1}-{\rm{Id}}-{\rm Op}(a)={\rm Op}(c_2)+\mathfrak{R}_{\rho},
\end{equation*}
where $c_1 \in S^{-1}, c_2 \in S^{-2}, \mathfrak{R}_{\rho} \in \mathfrak{L}_{\rho, p}$. Moreover, for all $s_0 \leq s \leq \mathcal{S}$, the following estimates hold:
\begin{equation*}
|c_1|_{-1, s, \alpha}^{\gamma, \mathcal{O}}, |c_2|_{-2, s, \alpha}^{\gamma, \mathcal{O}}
\leq_{s, \alpha, \rho} |a|_{-1, s+(\rho-2)(\rho-3), \alpha+\rho-3}^{\gamma, \mathcal{O}},
\end{equation*}
\begin{equation*}
|\Delta_{12}c_1|_{-1, p, \alpha}, |\Delta_{12}c_2|_{-2, p, \alpha}\leq_{p, \alpha, \rho}  |\Delta_{12}a|_{-1, p+(\rho-2)(\rho-3), \alpha+\rho-3},
\end{equation*}
\begin{equation*}
\mathbb{M}_{\mathfrak{R}_{\rho}}^{\gamma}(s, b)\leq_{s, \rho} |a|_{-1, s+(\rho-1)(\rho-2)+3, \rho-2}^{\gamma, \mathcal{O}}+\mathbb{M}_{T}^{\gamma}(s, b), \quad 0\leq b \leq \rho-2,
\end{equation*}
\begin{equation*}
\mathbb{M}_{\Delta_{12}\mathfrak{R}_{\rho}}(p, b)\leq_{p, \rho} |\Delta_{12}a|_{-1, p+(\rho-1)(\rho-2)+3, \rho-2}+\mathbb{M}_{\Delta_{12}T}(p, b), \quad 0\leq b \leq \rho-3.
\end{equation*}
\end{lemma}
\begin{proof}
Most of the results have been proved in \cite{fgp-1}. We only have to prove the formulas involving $c_2$.
Since
\begin{equation*}
c_1=\sum_{n=1}^{\rho-1}c^{(n)}=c^{(1)}+\sum_{n=2}^{\rho-1}c^{(n)}=a+\sum_{n=2}^{\rho-1}c^{(n)},
\end{equation*}
where $c^{(n)}:=a\sharp_{<\rho-2}c^{(n-1)}$, $c^{(1)}:=a$.
We define $c_2=\sum_{n=2}^{\rho-1}c^{(n)}$. Then the estimates of $c_2$ follow from the estimates involving $c^{(n)}$ in Lemma B.5 of \cite{fgp-1}.
\end{proof}
\begin{lemma}\label{pd-rhop-com}
Let $a \in S^1$ and $B \in \mathfrak{L}_{\rho+1, p}$. Then ${\rm Op}(a)\circ B, B\circ {\rm Op}(a) \in \mathfrak{L}_{\rho, p}$. Moreover, for all $s_0 \leq s \leq \mathcal{S}$,
\begin{equation*}
\mathbb{M}_{B\circ {\rm Op}(a)}^{\gamma}(s, b), \mathbb{M}_{{\rm Op}(a)\circ B}^{\gamma}(s, b)\leq_{s, \rho} |a|_{1, s+\rho, 0}^{\gamma, \mathcal{O}}\mathbb{M}_{B}^{\gamma}(s_0, b)
+|a|_{1, s_0+\rho, 0}^{\gamma, \mathcal{O}}\mathbb{M}_{B}^{\gamma}(s, b)
\end{equation*}
for  $0\leq b \leq \rho-2$,
and
\begin{equation*}
\begin{split}
\mathbb{M}_{\Delta_{12}(B\circ {\rm Op}(a))}(p, b), \mathbb{M}_{\Delta_{12}({\rm Op}(a)\circ B)}(p, b) \leq_{p, \rho}& |\Delta_{12}a|_{1, p+\rho, 0} \mathbb{M}_{B}^{\gamma}(p, b)\\
&+|a|_{1, p+\rho, 0} \mathbb{M}_{\Delta_{12}B}(p, b)\\
\end{split}
\end{equation*}
for $0\leq b \leq \rho-3$.
\end{lemma}
\begin{proof}
For $0\leq b \leq \rho-2, 0 \leq |\vec{b}| \leq b$, let $|\vec{b}_1|+|\vec{b}_2|=|\vec{b}|, m_1+m_2=\rho-|\vec{b}|$, we  have
\begin{equation*}
\begin{split}
&\langle D_x \rangle^{m_1}\partial_{\varphi}^{\vec{b}_1}{\rm Op}(a)\circ \partial_{\varphi}^{\vec{b}_2}B\langle D_x \rangle^{m_2}\\
&=\langle D_x \rangle^{m_1}\partial_{\varphi}^{\vec{b}_1}{\rm Op}(a)\langle D_x \rangle^{-m_1-1}\circ \langle D_x \rangle^{m_1+1}\partial_{\varphi}^{\vec{b}_2}B\langle D_x \rangle^{m_2}.
\end{split}
\end{equation*}
 Note that $\langle D_x \rangle^{m_1}\partial_{\varphi}^{\vec{b}_1}{\rm Op}(a)\langle D_x \rangle^{-m_1-1}$ is a pseudo-differential operator of order $0$. By Lemma \ref{pd-com}, Lemma \ref{R-tam-pd} and \eqref{pd-c0}-\eqref{pd-c1}, it follows that
\begin{equation*}
\begin{split}
\mathfrak{M}_{\langle D_x \rangle^{m_1}\partial_{\varphi}^{\vec{b}_1}{\rm Op}(a)\langle D_x \rangle^{-m_1-1}}^{\gamma}(0, s)
\leq_s &|\langle D_x \rangle^{m_1}\partial_{\varphi}^{\vec{b}_1}{\rm Op}(a)\langle D_x \rangle^{-m_1-1}|_{0, s, 0}^{\gamma, \mathcal{O}}\\
\leq_s &|\partial_{\varphi}^{\vec{b}_1}{\rm Op}(a)|_{1, s+m_1, 0}^{\gamma, \mathcal{O}}
\leq_s |a|_{1, s+|\vec{b}_1|+m_1, 0}^{\gamma, \mathcal{O}} \leq_s |a|_{1, s+\rho, 0}^{\gamma, \mathcal{O}}.\\
 \end{split}
\end{equation*}
Then by Lemma \ref{Tam-com}, the Lip-0-tame constant of $\langle D_x \rangle^{m_1}\partial_{\varphi}^{\vec{b}}({\rm Op}(a)B)\langle D_x \rangle^{m_2}$ satisfies
\begin{equation*}
\begin{split}
\mathfrak{M}_{\partial_{\varphi}^{\vec{b}}({\rm Op}(a)B)}^{\gamma}(-\rho+|\vec{b}|, s)
\leq_{s, \rho} |a|_{1, s+\rho, 0}^{\gamma, \mathcal{O}}  \mathbb{M}_B^{\gamma}(s_0, b)+ |a|_{1, s_0+\rho, 0}^{\gamma, \mathcal{O}} \mathbb{M}_B^{\gamma}(s, b).
 \end{split}
\end{equation*}
In addition, let $|\vec{b}_1|+|\vec{b}_2|=|\vec{b}|, m_1+m_2=\rho-|\vec{b}|-1$, we deduce that
\begin{equation*}
\begin{split}
\langle D_x \rangle^{m_1}\partial_{\varphi}^{\vec{b}}[{\rm Op}(a)B, \partial_x]\langle D_x \rangle^{m_2}
=&\langle D_x \rangle^{m_1}\partial_{\varphi}^{\vec{b}}([{\rm Op}(a), \partial_x]B)\langle D_x \rangle^{m_2}\\
&+\langle D_x \rangle^{m_1}\partial_{\varphi}^{\vec{b}}({\rm Op}(a)[B, \partial_x])\langle D_x \rangle^{m_2}.\\
 \end{split}
\end{equation*}
For the second summand, by the similar discussion as above, it implies  that
\begin{equation*}
\begin{split}
\mathfrak{M}_{\partial_{\varphi}^{\vec{b}}({\rm Op}(a)[B, \partial_x])}^{\gamma}&(-\rho+|\vec{b}|+1, s)\\
&\leq_{s, \rho} |a|_{1, s+\rho, 0}^{\gamma, \mathcal{O}}  \mathbb{M}_B^{\gamma}(s_0, b)+ |a|_{1, s_0+\rho, 0}^{\gamma, \mathcal{O}} \mathbb{M}_B^{\gamma}(s, b).\\
\end{split}
\end{equation*}
 For the first summand,
\begin{equation*}
\begin{split}
\langle D_x \rangle^{m_1}&\partial_{\varphi}^{\vec{b}_1}[{\rm Op}(a), \partial_x]\partial_{\varphi}^{\vec{b}_2} B\langle D_x \rangle^{m_2}\\
&=\langle D_x \rangle^{m_1}\partial_{\varphi}^{\vec{b}_1}[{\rm Op}(a), \partial_x]\langle D_x \rangle^{-m_1-2}\langle D_x \rangle^{m_1+2}\partial_{\varphi}^{\vec{b}_2} B\langle D_x \rangle^{m_2}.\\
\end{split}
\end{equation*}
Note that $\langle D_x \rangle^{m_1}\partial_{\varphi}^{\vec{b}_1}[{\rm Op}(a), \partial_x]\langle D_x \rangle^{-m_1-2}$ is a pseudo differential operator of order $0$, from Lemma  \ref{pd-com}, Lemma \ref{R-tam-pd}, and \eqref{pd-c0}-\eqref{pd-c1}, we deduce that
\begin{equation*}
\begin{split}
\mathfrak{M}_{\langle D_x  \rangle^{m_1}\partial_{\varphi}^{\vec{b}_1}[{\rm Op}(a), \partial_x]\langle D_x \rangle^{-m_1-2}}^{\gamma}&(0, s)
\leq_s |\langle D_x \rangle^{m_1}\partial_{\varphi}^{\vec{b}_1}[{\rm Op}(a), \partial_x]\langle D_x \rangle^{-m_1-2}|_{0,s,0}^{\gamma, \mathcal{O}}\\
\leq_s &|\partial_{\varphi}^{\vec{b}_1}[{\rm Op}(a), \partial_x]|_{2, s+m_1, 0}^{\gamma, \mathcal{O}}
\leq_s |\partial_{\varphi}^{\vec{b}_1} {\rm Op}(a)|_{1, s+m_1+1, 0}^{\gamma, \mathcal{O}}\\
\leq_s & |a|_{1, s+|\vec{b}_1|+m_1+1, 0}^{\gamma, \mathcal{O}} \leq_s |a|_{1, s+\rho, 0}^{\gamma, \mathcal{O}}.\\
 \end{split}
\end{equation*}
Then by Lemma  \ref{Tam-com}, we obtain
\begin{equation*}
\mathfrak{M}_{\partial_{\varphi}^{\vec{b}}[{\rm Op}(a)B, \partial_x]}^{\gamma}(-\rho+|\vec{b}|+1, s)
\leq_{s, \rho} |a|_{1, s+\rho, 0}^{\gamma, \mathcal{O}}  \mathbb{M}_B^{\gamma}(s_0, b)+ |a|_{1, s_0+\rho, 0}^{\gamma, \mathcal{O}} \mathbb{M}_B^{\gamma}(s, b).
\end{equation*}

  For the proof of the second result, we only need to notice that
\begin{equation*}
 \Delta_{12}({\rm Op}(a)B)=  \big(\Delta_{12}{\rm Op}(a)\big)B+{\rm Op}(a)\Delta_{12}B,
 \end{equation*}
  and
  \begin{equation*}
\begin{split}
\partial_{\varphi}^{\vec{b}} [\Delta_{12}({\rm Op}(a)B), \partial_x]
  =&\partial_{\varphi}^{\vec{b}}([\Delta_{12}{\rm Op}(a), \partial_x]B)+  \partial_{\varphi}^{\vec{b}}((\Delta_{12}{\rm Op}(a))[B, \partial_x])\\
& + \partial_{\varphi}^{\vec{b}}([{\rm Op}(a), \partial_x]\Delta_{12}B)+  \partial_{\varphi}^{\vec{b}}({\rm Op}(a)[\Delta_{12}B, \partial_x]).\\
 \end{split}
\end{equation*}

  The proof for $B\circ {\rm Op}(a)$ is similar.
 \end{proof}

We now show the connection between the class $\mathfrak{L}_{\rho, p}$ and the class of Lip-$-1$-modulo-tame operators. It can be directly derived from Lemma C.8 and Lemma C.10 in \cite{fgp-2}.

\begin{lemma}\label{Lrho-mod}
Assume  $\rho, b \in \mathbb{N}$ satisfy  $\rho \geq s_0+b+3$, $A \in \mathfrak{L}_{\rho, p}$.
Then $A$ and $\langle \partial_{\varphi}\rangle^{b} A$ are  Lip-$-1$-modulo-tame operators with modulo-tame constants
\begin{equation*}
\mathfrak{M}_A^{\sharp, \gamma}(-1, s)\leq_{\rho, s} \mathbb{M}_A^{\gamma}(s, \rho-2),  \quad \mathfrak{M}_A^{\sharp, \gamma}(-1, s, b)\leq_{\rho, s} \mathbb{M}_A^{\gamma}(s, \rho-2) .
\end{equation*}
Moreover,
 \begin{eqnarray*}
 \begin{aligned}
 &\|\langle D_x \rangle^{1/2} \underline{\Delta_{12}A} \langle D_x \rangle^{1/2}\|_{\mathfrak{L}(H^{s_0})}\leq_{\rho} \mathbb{M}_{\Delta_{12}A}(s_0, \rho-3),\\
 &\| \langle D_x \rangle^{1/2}\underline{\Delta_{12}\langle \partial_{\varphi}\rangle^{b} A}\langle D_x \rangle^{1/2} \|_{\mathfrak{L}(H^{s_0})} \leq_{\rho} \mathbb{M}_{\Delta_{12}A}(s_0, \rho-3).
 \end{aligned}
 \end{eqnarray*}
 \end{lemma}

\section{Change of variable}

 The following classical lemma states that the composition operator $u \mapsto u(x+p(x))$ induced by a diffeomorphism of the torus $\mathbb{T}$ is tame.
\begin{lemma}\label{Chan-var}(Change of variable)\cite{bbm-1, fgmp}
suppose  $p \in W^{s, \infty}(\mathbb{T}; \mathbb{R}), s\geq 1$ with $|p|_{1, \infty} \leq 1/2$ and let $f(x)=x+p(x)$. Then

(1). $f:\mathbb{T}\rightarrow \mathbb{T}$ is a diffeomorphism, its inverse is $g(y)=y+q(y)$ with $q \in W^{s, \infty}(\mathbb{T}; \mathbb{R})$ and
 \begin{equation}\label{q-p}
 |q|_{s, \infty} \leq_s |p|_{s, \infty}.
  \end{equation}
Moreover, suppose that $p=p(\lambda)$ depends  on a parameter $\lambda \in \Lambda \subset \mathbb{R}$ and $|D_x p(\lambda)|_{L^{\infty}} \leq 1/2$ for all $\lambda$. Then the following estimate holds:
 \begin{equation}\label{del-qp}
 |q(\lambda_1)-q(\lambda_2)|_{s, \infty} \leq_s  |p(\lambda_1)-p(\lambda_2)|_{s, \infty} + \sup_{\lambda \in \Lambda}|p(\lambda)|_{s+1, \infty} |p(\lambda_1)-p(\lambda_2)|_{ \infty} .
   \end{equation}

(2). If $u \in H^s(\mathbb{T}; \mathbb{C})$, then $u \circ f(x)=u(x+p(x)) \in H^s(\mathbb{T}, \mathbb{C})$, and
 \begin{equation*}
 \| u \circ f \|_s \leq_s \| u \|_s +|D_xp|_{s-1, \infty}\| u \|_1.
 \end{equation*}
Moreover, suppose that $p=p(\lambda)$ depends on a parameter $\lambda \in \Lambda \subset \mathbb{R}$ and suppose that
$|D_x u(\lambda)|_{L^{\infty}} \leq 1/2$ for all $\lambda$. Then the following estimate holds:
 \begin{equation}\label{del-uf}
 \begin{split}
 |(u \circ f)(\lambda_1)- &(u \circ f)(\lambda_2)|_{s, \infty}
\leq_s \big(1+\sup_{\lambda \in \Lambda}|p|_{s, \infty}\big) |u(\lambda_1)-u(\lambda_2)|_{s, \infty} \\
 &+ \left(\sup_{\lambda \in \Lambda}|p(\lambda)|_{s, \infty}+\sup_{\lambda \in \Lambda}|u(\lambda)|_{s+1, \infty}\right) |p(\lambda_1)-p(\lambda_2)|_{s, \infty} .\\
 \end{split}
\end{equation}

(3). Suppose  that $p$ depends in a Lipschitz way on a parameter $\omega \in \mathcal{O}\subset \mathbb{R}^{\nu}$
and $|p(\omega)|_{1, \infty} \leq 1/2$. Then $q=q(\omega)$ is also Lipschitz in $\omega$, and
 \begin{equation}\label{gam-qp}
|q|_{s, \infty}^{\gamma, \mathcal{O}} \leq_s |p|_{s, \infty}^{\gamma, \mathcal{O}},
\end{equation}
 \begin{equation}\label{gam-uf}
\| u \circ f \|_s^{\gamma, \mathcal{O}} \leq_s \| u  \|_s^{\gamma, \mathcal{O}} +\| u  \|_s^{\gamma, \mathcal{O}}|p|_{1, \infty}^{\gamma, \mathcal{O}}+ |p|_{s, \infty}^{\gamma, \mathcal{O}} \| u  \|_2^{\gamma, \mathcal{O}}.
\end{equation}
\end{lemma}

\end{appendix}

\end{document}